\documentclass[reqno]{amsart}
\usepackage{amsmath, amsthm, amssymb, amstext}
\usepackage[left=3cm,right=3cm,top=3cm,bottom=3cm]{geometry}
\usepackage{hyperref}
\usepackage[dvipsnames]{xcolor}
\usepackage{orcidlink}
\hypersetup{pdfborder={0 0 0},colorlinks}
\usepackage{enumitem}
\allowdisplaybreaks
\usepackage{todonotes}

\newtheorem{theorem}{Theorem}
\newtheorem{remark}[theorem]{Remark}
\newtheorem{lemma}[theorem]{Lemma}
\newtheorem{proposition}[theorem]{Proposition}

\newtheorem{definition}[theorem]{Definition}

\newcommand{\Om} {\Omega}
\newcommand{\pa} {\partial}

\newcommand{\al} {\alpha}
\newcommand{\ba} {\beta}

\newcommand{\ga} {\gamma}

\newcommand{\la} {\lambda}
\newcommand{\si} {\sigma}

\newcommand{\lt} {\left}
\newcommand{\ri} {\right}
\newcommand{\e} {\varepsilon}
\newcommand{\ka} {\kappa}
\newcommand{\nl}{ N_{\lambda}}

\def\R{{\mathbb R}}
\def\N{{\mathbb N}}
\newcommand{\h}{\mathcal{H}}
\newcommand{\fp}{\mathcal{G}}
\newcommand{\op}{\mathcal{L}_{p,q}^{a}}
\newcommand{\ops}{\mathcal{L}_{p,q}^{s,b}}

\newcommand{\na} {\nabla}

\newcommand{\I} {\mathfrak{I}}
\newcommand{\s}{\mathcal{S}}
\newcommand{\So} {\mathcal{S}_p}
\newcommand{\nla} {\mathfrak{N}_{\lambda}}
\newcommand{\m}{\mathfrak{m}_{\lambda}}

\newcommand{\g} {\mathfrak{G}}

\newcommand{\mm}[1]{\lVert #1 \rVert}
\newcommand{\mmm}[1]{{\lVert #1 \rVert}_W}
 \newcommand{\p}[1]{\lvert \nabla#1 \rvert}
 \newcommand{\ve}[1]{\lvert #1 \rvert}

\def\dxy{\,{\rm d}x\,{\rm d}y}

\def\dx{\,{\rm d}x}
\def\dy{\,{\rm d}y}
\def\dv{\,{\rm d}\nu}

\numberwithin{theorem}{section}
\numberwithin{equation}{section}

\title{Mixed double phase equations with local and nonlocal operators}

\author[A. Arora]{Anupma Arora\orcidlink{0009-0006-9391-497X}}
\address[A. Arora]{Department of Mathematics, Birla Institute of Technology and Science Pilani, Pilani Campus, Vidya Vihar, Pilani, Jhunjhunu, Rajasthan, 333031, India}
\email{p20200439@pilani.bits-pilani.ac.in}

\author[S. Gupta]{Shilpa Gupta\orcidlink{0000-0002-4080-9782}}
\address[S. Gupta]{Department of Mathematics and Statistics, Indian Institute of Technology Kanpur, Kanpur, 208016, India}
\email{shilpagupta890@gmail.com}

\author[P. Winkert]{Patrick Winkert\orcidlink{0000-0003-0320-7026}}
\address[P. Winkert]{Technische Universit\"{a}t Berlin, Institut f\"{u}r Mathematik, Stra\ss{}e des 17.\,Juni 136, 10623 Berlin, Germany}
\email{winkert@math.tu-berlin.de}


\begin{document}

\begin{abstract}
	In this paper, we study a new class of mixed double phase problems that combine local and nonlocal operators. We consider two different models. The first model is driven by the fractional $p$-Laplacian together with a local double phase operator, while the second model involves the local $p$-Laplacian coupled with a fractional double phase operator. In order to describe the interaction between local and nonlocal effects within the double phase framework, we introduce an appropriate variational setting based on classical and fractional Musielak-Orlicz Sobolev spaces. Within this setting, we establish several existence and multiplicity results for weak solutions by means of variational and topological techniques. In particular, for the problem driven by the fractional $p$-Laplacian and a local double phase operator, we prove the existence of a nonnegative solution using the Nehari manifold method in the presence of concave-convex nonlinearities. We also investigate the associated Brezis-Nirenberg type problem and obtain the existence of infinitely many solutions via genus theory. For the problem governed by the local $p$-Laplacian and a fractional double phase operator, we show the existence of at least two nontrivial constant sign solutions by exploiting the variational structure of the associated energy functional. Furthermore, in the subcritical case, we prove the existence of a least energy sign-changing solution by combining the Poincar\'{e}-Miranda existence theorem with the quantitative deformation lemma.
\end{abstract}

\maketitle

\section{Introduction and main results}

In this work, we introduce and study two classes of mixed operators combining local and nonlocal effects. The first class consists of the fractional $p$-Laplacian coupled with a local double phase operator, while the second class involves the local $p$-Laplacian coupled with a fractional double phase operator. The study of operators of mixed order has recently attracted significant attention, as such operators arise naturally in various contexts. A notable example is their appearance as the superposition of different stochastic processes, such as a classical random walk and a L\'{e}vy flight, which also play an important role in the analysis of optimal animal foraging strategies, see, for example, the works by Dipierro--Proietti Lippi--Valdinoci  \cite{Dipierro-ProiettiLippi-Valdinoci-2023}, Dipierro--Valdinoci \cite{Dipierro-Valdinoci-2021}, Montefusco--Pellacci--Verzini \cite{Montefusco-Pellacci-Verzini-2013}, and Pellacci--Verzini \cite{Pellacci-Verzini-2018}.

We begin by studying a boundary value problem driven by the fractional $p$-Laplacian and a local double phase operator in the presence of concave-convex nonlinearities. More precisely, we consider the problem
\begin{equation}\label{1.1}
	\begin{aligned}
		(-\Delta)_{p}^s u -\op(u)  & =  \lambda\left(w_1(x)|u|^{k-2}u+w_2(x)|u|^{r-2}u\right) && \text{in } \Omega, \\
		u & = 0 &  & \text{in } \R^N\setminus\Omega,
	\end{aligned}
\end{equation}
where $\Omega\subseteq \R^{N}$, $N \geq 2$, is a  bounded domain with smooth boundary $\partial\Omega$ and
\begin{align*}
	(-\Delta)^{s}_{p}u=C \lim\limits_{\varepsilon\to 0}\int_{\R^{N}\setminus B_{\varepsilon}(x)}\dfrac{|u(x)-u(y)|^{p-2}(u(x)-u(y))}{|x-y|^{N+ps}}\dy,
\end{align*}
denotes the fractional $p$-Laplacian. Here $B_{\varepsilon}(x):=\{z\in \mathbb{R}^N\colon |z-x|<\varepsilon\}$ and $C$ is a normalization constant. Moreover,
\begin{align*}
	\op(u)=\operatorname{div}(\p{u}^{p-2} \nabla u+a(x)\p{u}^{q-2} \nabla u)
\end{align*}
is the classical double phase operator. We assume the following hypotheses.
\begin{enumerate}[label=\textnormal{(H$_1$)},ref=\textnormal{H$_1$}]
	\item\label{H1}
		\begin{enumerate}[label=\textnormal{(\roman*)},ref=\textnormal{\roman*}]
			\item\label{H1i}
				$1<p<q<N$, $1<k<p<q<r\leq p^*=\frac{Np}{N-p}$, and $s\in (0,1)$;
			\item\label{H1ii}
				$w_1\in L^{\frac{r}{r-k}}(\Omega)$, $w_2\in L^\infty(\Omega)$ with $w_1,w_2>0$ a.e.\,in $\Omega$ and $0\leq a(\cdot)\in C^{0,1}(\overline{\Omega})$ ; if $r=p^*$, then $w_1\in L^\infty(\Omega)$.
		\end{enumerate}
\end{enumerate}
A function $u\in W$ is said to be a weak solution of \eqref{1.1} if
\begin{align*}
	&\int_{\R^N}\int_{\R^N} \dfrac{|u(x)-u(y)|^{p-2}(u(x)-u(y))(v(x)-v(y))}{|x-y|^{N+ps}}\dxy \\
	& +\int_{\Omega}  (|\nabla u|^{p-2}\nabla u  + a(x) |\nabla u|^{q-2}\nabla u) \cdot\nabla v\dx =\lambda\int_{\Omega} \left(w_1(x)|u|^{k-2}u+w_2(x)|u|^{r-2}u\right) v \dx,
\end{align*}
for all test functions $v\in W$, where the space $W$ is defined in \eqref{space-W}.

Our first existence result concerning problem \eqref{1.1} is the following.

\begin{theorem}\label{t1}
	Let hypotheses \eqref{H1} be satisfied. Then, there exists $\lambda_{0} > 0$ such that for every $\lambda \in (0,\lambda_{0})$, problem \eqref{1.1} admits a nontrivial, nonnegative solution with negative energy.
\end{theorem}

We also consider a Brezis Nirenberg type problem obtained from \eqref{1.1} by assuming $w_2=\frac{1}{\lambda}$ and $r=p^*$. In this case, problem \eqref{1.1} reduces to
\begin{equation}\label{1.2}
	\begin{aligned}
		(-\Delta)_{p}^s u -\op(u)  & =  \lambda w_1(x)|u|^{k-2}u+|u|^{p^*-2}u  && \text{in } \Omega, \\
		u & = 0 &  & \text{in } \R^N\setminus\Omega.
	\end{aligned}
\end{equation}

\begin{theorem}\label{t2}
	Let hypotheses \eqref{H1} with $w_2=\frac{1}{\lambda}$ and $r=p^*$ be satisfied. Then, there exists $\lambda_{1} > 0$ such that for every $\lambda \in (0,\lambda_{1})$, problem \eqref{1.2} admits a nontrivial, nonnegative solution with negative energy.
\end{theorem}

If, in addition, $w_1=1$, we obtain the following multiplicity result.

\begin{theorem}\label{t3}
	Let hypotheses \eqref{H1} with $w_2=\frac{1}{\lambda}$, $w_1=1$, and $r=p^*$ be satisfied. Then there exists $\Lambda>0$ such that for all $\lambda\in(0,\Lambda)$, problem \eqref{1.2} admits infinitely many nontrivial solutions with negative energy.
\end{theorem}

Subsequently, motivated by the interplay between local diffusion and nonlocal phase transitions, we turn our attention to a second problem in which the roles of the local and nonlocal operators are interchanged. In this case, the equation is governed by the local $p$-Laplacian and a fractional double phase operator, and it reads
\begin{align}\label{frac_dbl}
	-\Delta_{p} u -\ops(u) =  f(x,u) \quad\text{in } \Omega, \quad u = 0 \quad\text{in } \R^N\setminus\Omega,
\end{align}
where $\Omega\subseteq \R^{N}$, $N \geq 2$, is a  bounded domain with smooth boundary $\partial\Omega$, $ \Delta_{p}u=\operatorname{div}(\p{u}^{p-2}\na u)$ denotes the local $p$-Laplace operator, and the nonlocal operator $\ops$ is defined by
\begin{align*}
	&\ops(u)\\
	&=\lim\limits_{\varepsilon\to 0}\int_{\R^{N}\setminus B_{\varepsilon}(x)}\left(\dfrac{|u(x)-u(y)|^{p-2}(u(x)-u(y))}{|x-y|^{N+ps}}+b(x,y)\dfrac{|u(x)-u(y)|^{q-2}(u(x)-u(y))}{|x-y|^{N+qs}}\right)\dy.
\end{align*}
For problem \eqref{frac_dbl} we impose the following assumptions:
\begin{enumerate}[label=\textnormal{(H$_2$)},ref=\textnormal{H$_2$}]
	\item\label{H2}
		\begin{enumerate}[label=\textnormal{(\roman*)},ref=\textnormal{\roman*}]
			\item\label{H2i}
				$1<p<q<N$, $0\le b(\cdot,\cdot)\in L^{\infty}(\mathbb{R}^N\times\mathbb{R}^N)$ with $b(x,y)=b(y,x)$ for all $(x,y)\in\mathbb{R}^N\times\mathbb{R}^N$, and $0<s<1$;
			\item\label{H2ii}
				$f\colon \Omega\times \mathbb{R}\to \mathbb{R}$ is a Carath\'{e}odory function  and there exists a constant $c > 0$ such that
				\begin{align*}
					|f(x,t)| \le c \big(1 + |t|^{r-1}\big) \quad \text{for a.a.\,} x \in \Omega \text{ and for all } t\in \mathbb{R},
				\end{align*}
				where $r < p^{*}$;
			\item\label{H2iii}
				$ \displaystyle\lim_{t\to 0} \frac{f(x,t)}{|t|^{p-2}t} = 0$ uniformly for a.a.\,$x\in\Omega$;
			\item\label{H2iv}
				$ \displaystyle\lim_{t\to \pm\infty} \frac{f(x,t)}{|t|^{q-2}t} = \infty$ uniformly for a.a.\,$x\in\Omega$;
			\item\label{H2v}
				for a.a.\,$x \in \Omega$, the map $t \mapsto t f(x,t) - q F(x,t)$ is nondecreasing for $t \ge 0$ and nonincreasing for $t \le 0$, where
				\begin{align*}
					F(x,t) = \int_{0}^{t} f(x,\tau)\,\mathrm{d}\tau
				\end{align*}
			\item\label{H2vi}
				for a.a.\,$x \in \Omega$, the map $t \mapsto \dfrac{f(x,t)}{|t|^{q-1}}$ is strictly increasing on $(-\infty,0)$ and on $(0,\infty)$.
		\end{enumerate}
\end{enumerate}

A function $u\in E$ is called a weak solution of  \eqref{frac_dbl} if
\begin{align*}
	&\int_{\Omega}  |\nabla u|^{p-2}\nabla u \cdot\nabla v\dx  +\int_{Q}|D_s(u)|^{p-2}D_s(u)D_s(v)\dv\\
	&+\int_{Q}b(x,y)|D_s(u)|^{q-2}D_s(u)D_s(v)\dv =\int_{\Omega} f(x,u) v \dx,
\end{align*}
for all test functions $v\in E$ with $Q=\R^{2N}\setminus (\Omega^c\times \Omega^c)$, where the space $E$ is defined in \eqref{space-E}. Here and in the sequel, we use the notation
\begin{align}\label{notion-ds-dv}
	D_s(u)=\frac{u(x)-u(y)}{\ve{x-y}^s} \quad\text{and}\quad \dv=\frac{\dxy}{|x-y|^N}.
\end{align}

Our main results related to problem \eqref{frac_dbl} are the following ones.

\begin{theorem}\label{t4}
	Let hypotheses \eqref{H2} be satisfied. Then problem \eqref{frac_dbl} has at least two nontrivial constant sign solutions
	$u_{*}, v_{*} \in E$ such that
	\begin{align*}
		u_{*}(x) \geq 0 \quad \text{and} \quad v_{*}(x) \leq 0
		\quad \text{for a.a.\,} x \in \Omega.
	\end{align*}
\end{theorem}

\begin{theorem}\label{t5}
	Let hypotheses \eqref{H2} be satisfied. Then problem \eqref{frac_dbl} has a least energy sign-changing solution.
\end{theorem}

Double phase operators originate from the seminal work of Zhikov \cite{Zhikov-1986}, motivated by the mathematical modeling of strongly anisotropic materials whose constitutive laws exhibit different growth behaviors in different regions of the domain. The associated energy functionals typically switch between two distinct polynomial growths, which leads to differential operators of the form
\begin{align*}
	\operatorname{div}(\p{u}^{p-2} \nabla u+a(x)\p{u}^{q-2} \nabla u).
\end{align*}
Such operators naturally fit into the framework of Musielak-Orlicz Sobolev spaces and present substantial analytical difficulties due to their nonhomogeneous and nonuniformly elliptic structure.

The spectral and variational theory for double phase operators was initiated by Colasuonno--Squassina \cite{Colasuonno-Squassina-2016}, who investigated the associated eigenvalue problem and established the existence of variational eigenvalues together with a detailed analysis of their qualitative properties. Since then, a rapidly growing literature has addressed the existence, multiplicity, and qualitative behavior of solutions to elliptic problems driven by double phase operators, mainly by means of variational techniques. In this direction, we refer to the works of Liu--Dai  \cite{Liu-Dai-2018, Liu-Dai-2020}, who studied ground state solutions and multiple solutions in $\R^N$, respectively. Problems involving concave convex nonlinearities were investigated by Kim--Kim--Oh--Zeng \cite{Kim-Kim-Oh-Zeng-2022} and by Mishra--Silva--Tripathi \cite{Mishra-Silva-Tripathi-2023}. Multiplicity results for problems with nonlinear boundary conditions were obtained by Amoroso--Crespo-Blanco--Pucci--Winkert \cite{Amoroso-Crespo-Blanco-Pucci-Winkert-2024}. Moreover, Farkas--Fiscella--Winkert \cite{Farkas-Fiscella-Winkert-2022} established multiplicity results in the presence of critical growth of order $p^*$. To deal with the lack of compactness caused by the critical exponent, they performed a refined convergence analysis of the gradients. Further contributions concerning elliptic problems with critical nonlinearities can be found in Arora--Fiscella--Mukherjee--Winkert \cite{Arora-Fiscella-Mukherjee-Winkert-2022} and Bahrouni--Fiscella--Winkert \cite{Bahrouni-Fiscella-Winkert-2026}. For problems involving critical growth switching between the Sobolev exponents $p^*$ and $q^*$, we refer to the works of Ho--Winkert \cite{Ho-Winkert-2023}, Ha--Ho \cite{Ha-Ho-2024,Ha-Ho-2025}, and Farkas--Fiscella--Ho--Winkert \cite{Farkas-Fiscella-Ho-Winkert-2026}. Kirchhoff type double phase problems were investigated by Crespo-Blanco--Gasi\'{n}ski--Winkert \cite{Crespo-Blanco-Gasinski-Winkert-2024}, who considered equations of the form
\begin{equation}\label{dbl_kirc}
	\begin{aligned}
		M\left(\int_{\Omega}\left( |\nabla u|^p+a(x)|\nabla u|^q\right)\dx\right)\op(u) & =  f(x,u) && \text{in } \Omega, \\
		u & = 0 &  & \text{on } \partial\Omega,
	\end{aligned}
\end{equation}
where $\Omega\subseteq \R^{N}$, $N \geq 2$, is a  bounded domain with Lipschitz boundary $\partial\Omega$, $1<p<N$, $1<p<q\leq p^*=\frac{Np}{N-p}$ and the Kirchhoff function is given by $M(t)=m_1+m_2t^{\sigma-1}$.  By combining variational methods with the Poincar\'{e}-Miranda existence theorem and the quantitative deformation lemma, the authors established the existence of two constant sign solutions together with a least energy sign changing solution.

All the literature discussed above concerns double phase operators involving the local
$p$- and $q$-Laplacians. By contrast, the analysis of double phase operators involving nonlocal terms is relatively recent. In this direction, Cheng--Bai \cite{Cheng-Bai-2025} investigated a fractional double phase problem of the form
\begin{equation}\label{frac_dbl_first}
	\begin{aligned}
		M\left(\int_{Q}\left( \dfrac{|u(x)-u(y)|^p}{|x-y|^{N+ps}}+b(x,y)\dfrac{|u(x)-u(y)|^q}{|x-y|^{N+qs}}\right)\dxy\right) \ops(u) & =f(x,u) && \text{in } \Omega, \\
		u & = 0 &  & \text{in } \R^N\setminus\Omega,
	\end{aligned}
\end{equation}
where $\Omega\subseteq \R^{N}$ is a  bounded smooth domain, $0<s<1<p<q<N/s$ and $M(t)=m_1+m_2t^{\sigma-1}$. The nonlinearity $f$ exhibits both singular and Choquard type growth. By employing the Nehari manifold method, Cheng--Bai \cite{Cheng-Bai-2025} proved the existence of two distinct weak solutions to problem \eqref{frac_dbl_first}. Subsequently, Zeng--Lu--R\u{a}dulescu--Winkert \cite{Zeng-Lu-Radulescu-Winkert-2026} studied an inclusion problem involving a fractional double phase operator with a logarithmic perturbation. By means of the sub- and supersolution method, they established several existence results. Moreover, the authors also investigated the fractional counterpart of problem \eqref{dbl_kirc}, namely a Kirchhoff-type problem of the form \eqref{frac_dbl_first}. By combining variational techniques with the Poincar\'{e}-Miranda existence theorem and the quantitative deformation lemma, they proved the existence of two constant sign solutions together with one sign changing solution.

We now turn to elliptic problems involving mixed operators, namely operators obtained as a superposition of local and nonlocal components. Such operators naturally arise in models where classical diffusion interacts with long-range effects. In recent years, problems combining local and nonlocal features have attracted considerable attention, and a growing body of literature has been devoted to the study of existence, multiplicity, and qualitative properties of solutions. In particular, several works have investigated equations of the form
\begin{align*}
	-\Delta_pu+(-\Delta)_q^su=f \quad \text{in } \Omega,
\end{align*}
where $p,q\in (1,\infty)$ and $s\in (0,1)$. In the case $p=q=2$, we mention the contributions of Biagi--Dipierro--Valdinoci--Vecchi \cite{Biagi-Dipierro-Valdinoci-Vecchi-2022}, where existence results, maximum principles, and regularity properties were established. Problems involving Hardy type potentials were studied by Biagi--Esposito--Montoro--Vecchi \cite{Biagi-Esposito-Montoro-Vecchi-2025}, while the Brezis Nirenberg problem was addressed in Biagi--Dipierro--Valdinoci--Vecchi \cite{Biagi-Dipierro-Valdinoci-Vecchi-2025}. We also mention the work by Balci \cite{Balci-2023} about nonlocal and mixed models with Lavrentiev gap. Singular and critical nonlinearities were considered by Biagi--Vecchi \cite{Biagi-Vecchi-2024} , and by Anthal--Giacomoni--Sreenadh \cite {Anthal-Giacomoni-Sreenadh-2025} in the presence of singular and critical Choquard type nonlinearities. Multiplicity results and the existence of sign changing solutions via descending flow methods were obtained by Su--Valdinoci--Wei--Zhang \cite{Su-Valdinoci-Wei-Zhang-2024}. Finally, Kirchhoff type problems with critical growth were investigated by Tripathi \cite{Tripathi-2024} using the Nehari manifold method. In the case $p=q$, we refer to the work of Garain--Ukhlov \cite{Garain-Ukhlov-2022}, where existence, uniqueness, and symmetry properties were established for singular problems. Critical nonlinearities were studied by da Silva--Fiscella--Viloria \cite{daSilva-Fiscella-Viloria-2024}, while Bhowmick--Ghosh \cite{Bhowmick-Ghosh-2025} obtained the existence of sign changing solutions by combining the Nehari manifold method with Browder degree theory. When  $p\neq q$, Dhanya--Giacomoni--Jana \cite{Dhanya-Giacomoni-Jana-2025} derived existence and multiplicity results for concave-convex nonlinearities in the presence of sign changing weights. More recently, Malhotra--Pandey--Sreenadh \cite {Malhotra-Pandey-Sreenadh-2025} investigated a mixed local and nonlocal problem with variable exponents. In their work, the authors introduced the appropriate functional framework and proved existence results for singular and superlinear nonlinearities by means of the Nehari manifold technique.

Despite the extensive literature on double phase problems and the rapidly growing work on mixed local-nonlocal operators, these two research directions have so far evolved largely independently. Existing results on double phase problems mainly concern operators that are either purely local or purely nonlocal, whereas the theory of mixed operators has been developed almost exclusively for homogeneous growth structures, typically involving the superposition of a local  $p$-Laplacian and a fractional $q$-Laplacian. To the best of our knowledge, problems involving a genuine coupling of local and nonlocal effects within a double phase framework have not yet been investigated. In particular, the interaction between a fractional $p$-Laplacian and a local double phase operator, as well as the converse situation involving a local $p$-Laplacian coupled with a fractional double phase operator, appears to be completely unexplored.

Motivated by these gaps, the present paper investigates two new classes of mixed local and nonlocal double phase problems, namely \eqref{1.1} and \eqref{frac_dbl}. By developing an appropriate variational framework based on classical and fractional Musielak-Orlicz Sobolev spaces, we establish existence and multiplicity results through Nehari manifold techniques, genus theory, and a combination of variational and topological methods. Our findings substantially extend both the theory of double phase problems and the theory of mixed local and nonlocal operators, by capturing the delicate interaction between local and nonlocal effects within a double phase setting. Since both models involve double phase operators, the analysis naturally requires working in generalized Musielak-Orlicz Sobolev spaces. In particular, we introduce two function spaces specifically tailored to the mixed double phase structure and establish their main properties, which play a crucial role in the variational approach developed in this work.

In problem \eqref{1.1}, we deal with concave-convex nonlinearities in the presence of positive weights. The existence of solutions is established by means of the Nehari manifold approach. We also investigate the case in which the weight satisfies $w_2=\frac{1}{\la}$ and the exponent is critical, namely $r=p^*$, which corresponds to the Brezis-Nirenberg type problem. Inspired by the works of Farkas--Fiscella--Winkert \cite{Farkas-Fiscella-Winkert-2022}, da Silva--Fiscella--Viloria \cite{daSilva-Fiscella-Viloria-2024} and Dhanya--Giacomoni--Jana \cite{Dhanya-Giacomoni-Jana-2025}, we address this setting by combining the Nehari manifold method with genus theory. In this way, we obtain both existence and multiplicity results for the associated problem.

For problem \eqref{frac_dbl}, we deal with a subcritical nonlinearity that does not satisfy the Ambrosetti-Rabinowitz condition. We prove the existence of three distinct solutions, namely one positive solution, one negative solution, and one solution that changes sign. The existence of the constant sign solutions is obtained through a variational critical point argument of mountain pass type, while the sign-changing solution is constructed by working on a suitable constraint set and applying the Poincar\'{e}-Miranda existence theorem. The variational strategy and the corresponding existence results for this problem are inspired by the works of Crespo-Blanco--Gasi\'{n}ski--Winkert \cite{Crespo-Blanco-Gasinski-Winkert-2024}  and Zeng--Lu--R\u{a}dulescu--Winkert \cite{Zeng-Lu-Radulescu-Winkert-2026}.

The paper is organized as follows. In Section \ref{Section_2}, we introduce the Sobolev spaces that provide the functional framework for our analysis. In addition, we collect several preliminary lemmas and technical results that will be used throughout the paper. Section \ref{Section_3} is devoted to the study of the fractional $p$-Laplacian combined with a double phase operator, where existence and multiplicity results are obtained by means of the Nehari manifold technique and topological arguments, see Theorems \ref{t1}--\ref{t3}. Finally, in Section \ref{Section_4}, we investigate the local $p$-Laplacian coupled with a fractional double phase operator, and prove the existence of constant sign solutions as well as a least energy solution that changes sign by using variational tools such as the Poincar\'{e}-Miranda existence theorem and the Deformation Lemma, see Theorems \ref{t4} and \ref{t5}.

\section{Functional setting and preliminaries}\label{Section_2}

In order to develop the variational framework associated with the problems under consideration, it is necessary to precisely describe the function spaces in which weak solutions are sought. This section is devoted to recalling the main properties of classical and fractional Sobolev spaces, together with their generalized counterparts in the sense of Musielak. Furthermore, we introduce two function spaces specifically designed to treat problems of the form \eqref{1.1} and \eqref{frac_dbl}. We investigate their fundamental properties, which are essential for the variational methods employed in the subsequent analysis.

Throughout the paper, $c>0$ denotes a generic constant, whose value may change from line to line. For $a \in (1,\infty)$, we denote by $a' := \frac{a}{a-1}$ the conjugate exponent of $a$. For  $x \in \mathbb{R}^N$ and $r>0$, we write
\begin{align*}
	B_r(x) := \left\{ y \in \mathbb{R}^N \colon  |y-x| < r \right\}
\end{align*}
for the open ball of radius $r$ centered at $x$. When the center is the origin, we simply write $B_r := B_r(0)$. Unless otherwise stated, $\Omega \subseteq \R^N$ denotes a bounded domain with smooth boundary, and $|\Omega|$ stands for the Lebesgue measure of $\Omega$. Moreover, for $t\in\R$, we define $t^{\pm}=\max\{\pm t,0\}$, the positive and negative parts of $t$. Clearly, $t=t^+-t^-$.

For $r\in[1,\infty)$, we denote by $L^r(\Om)$  the usual Lebesgue spaces endowed with the norm $\|u\|_r = \Big(\int_{\Omega} |u|^r \dx \Big)^{1/r}$. Let $M(\Omega)$ be the space of all measurable functions $u\colon\Omega\to\R$. Given a measurable function $w\colon \Omega \to (0,\infty)$ and $1 \leq r < \infty$, we define the weighted Lebesgue space $L^{r}(\Omega, w)$ by
\begin{align*}
	L^{r}(\Omega, w) := \left\{u\in M(\Omega) \colon
	\int_{\Omega}w(x) |u(x)|^{r} \dx < \infty \right\},
\end{align*}
which is equipped with the norm
\begin{align*}
	\|u\|_{r, w} :=
	\left( \int_{\Omega} w(x)|u(x)|^{r} \dx \right)^{\frac{1}{r}}.
\end{align*}
We denote by $W^{1,r}(\Om)$ and $W_0^{1,r}(\Om)$ the usual Sobolev spaces endowed with the norms
\begin{align*}
	\|u\|_{1,r} = \|u\|_r + \|\nabla u\|_r
	\quad \text{and} \quad
	\|u\|_{1,r,0} = \|\nabla u\|_r.
\end{align*}
For $1<r<N$, the Sobolev embedding $W_0^{1,r}(\Omega)\hookrightarrow L^\ell(\Omega)$ for $\ell \in [1,r^*]$ holds, where $r^*=\frac{Nr}{N-r}$. More precisely, there exists a constant $S_\ell>0$ such that
\begin{align*}
	\|u\|_{\ell} \leq S_\ell \|\nabla u\|_{r}\quad\text{for all }u \in W_0^{1,r}(\Omega).
\end{align*}

In order to treat the critical Sobolev term, we introduce the best Sobolev constant defined by
\begin{align}\label{a2}
	\So := \inf_{u \in W_0^{1,p}(\Omega)\setminus\{0\}}
	\frac{\|\nabla u\|_{p}^p}{\|u\|_{p^{*}}^p},
\end{align}
for $1<p<N$. It is well known that $\So$ is strictly positive and finite. Next, let $s\in(0,1)$ and $r \in [1,\infty)$. Then, the fractional Sobolev space $W^{s,r}(\R^N)$ is defined as
\begin{align*}
	W^{s,r}(\R^N):=\left\{u\in L^r(\R^N)\colon [u]_{s,r}^r<\infty\right\},
\end{align*}
equipped with the norm
\begin{align*}
	\mm{u}_{s,r}=\left(\mm{u}_r^r+[u]_{s,r}^r\right)^{\frac{1}{r}},
\end{align*}
where
\begin{align*}
	[u]_{s,r}^r=\int_{\R^N }\int_{\R^N}\frac{\ve{u(x)-u(y)}^r}{\ve{x-y}^{N+sr}}\dxy
\end{align*}
is the Gagliardo seminorm. For a comprehensive treatment of fractional Sobolev spaces and their properties, we refer to Di Nezza--Palatucci--Valdinoci \cite{DiNezza-Palatucci-Valdinoci-2012}.

Define,
\begin{align*}
	\fp(x,y,t)=\int_{0}^{|t|} g(x,y,\tau)\tau\,\mathrm{d}\tau,
\end{align*}
where $ g\colon \Omega\times\Omega\times [0,\infty)\to [0,\infty)$ is a given function. Recall that, $\fp\colon \Omega\times\Omega\times \R\to[0,\infty)$ is called a generalized $N$-function (denoted by $\fp\in N(\Omega\times\Omega)$) if it satisfies the following conditions:
\begin{enumerate}
	\item[\textnormal{(i)}]
		For a.a.\,$x,y\in\Omega$, the map $\fp(x,y,\cdot)\colon \R\to[0,\infty)$ is continuous, even, increasing, and convex.
	\item[\textnormal{(ii)}]
		For every $t\geq 0$, the map $\fp(\cdot,\cdot,t)\colon \Omega\times\Omega\to[0,\infty)$ is measurable.
	\item[\textnormal{(iii)}]
		$\fp(x,y,t)=0$ if and only if $t=0$ for a.a.\,$(x,y)\in\Omega\times \Omega$.
	\item[\textnormal{(iv)}]
		$\displaystyle\lim_{t\to 0}\frac{\fp(x,y,t)}{t}=0$ and $\displaystyle\lim_{t\to \infty}\frac{\fp(x,y,t)}{t}=\infty$ for a.a.\,$(x,y)\in\Omega\times \Omega$.
\end{enumerate}
We say that a generalized $N$-function $\fp$ satisfies the weak $\Delta_{2}$-condition if there exist $\delta_{0}>0$ and a nonnegative function $k\in L^{1}(\Omega)$ such that
\begin{align*}
	\fp(x,y,2t)\leq \delta_{0}\fp(x,y,t)+k(x)
\end{align*}
for a.a.\,$(x, y) \in \Omega\times\Omega$ and for all $t \geq 0$. If $k=0$, then $\fp$ is said to satisfy the  $\Delta_{2}$-condition.

We say that $\fp\in N(\Omega\times\Omega)$ is locally integrable if, for any $t>0$ and for every compact set $K\subseteq\Omega$,
\begin{align*}
	\int_{K\times K}\fp(x,y,t)\dxy< \infty \quad\text{and}\quad \int_{K}\fp(x,x,t)\dx<\infty.
\end{align*}
Similarly, one defines a generalized $N$-function $\h\colon \Omega\times \R\to[0,\infty)$ (denoted by $\h\in N(\Omega)$).

Let $\h\in N(\Omega)$. The Musielak-Orlicz Lebesgue space $L^{\h}(\Omega)$ is defined by
\begin{align*}
	L^{\h}(\Omega)=\left\{ u\in M(\Omega)\colon \int_{\Omega}\h\left( x,\tau|u|\right) \dx<\infty \text{ for some }\tau>0\right\}.
\end{align*}
Equipped with the Luxemburg norm
\begin{align*}
	\| u\|_{\h}=\inf\left\{ \tau>0\colon \int_{\Omega}\h\left( x,\frac{|u|}{\tau}\right) \dx\leq 1\right\},
\end{align*}
the space $L^{\h}(\Omega)$  becomes a Banach space, see Musielak \cite{Musielak-1983}.

The following result can be found in Youssfi--Ahmida \cite[Remark B.1, Theorem B.3 and Theorem 2.2]{Youssfi-Ahmida-2020}.

\begin{theorem}
	Let $\h\in N(\Omega)$.
	\begin{enumerate}
		\item[\textnormal{(i)}]
			If $\h$ satisfies the $\Delta_{2}$-condition, then $L^{\h}(\Omega)$ is a separable and reflexive Banach space.
		\item[\textnormal{(ii)}]
			If $\h$ is locally integrable then $C_c^{\infty}(\Omega)$ is dense in $L^{\h}(\Omega)$.
	\end{enumerate}
\end{theorem}

For further background on Musielak-Orlicz Lebesgue spaces, we refer to Fukagai--Ito--Narukawa \cite{Fukagai-Ito-Narukawa-2006}.

For a given $\h\in N(\Omega)$, we define the Musielak-Orlicz Sobolev space $W^{1,\h}(\Om)$ by
\begin{align*}
	W^{1,\h}(\Om):=\left\{u\in L^{\h}(\Om) \colon  \p{u} \in L^{\h}(\Om)\right\}.
\end{align*}
It is equipped with the norm
\begin{align*}
	\mm{u}_{1,\h}:=\mm{u}_{\h}+\mm{\na u}_{\h}.
\end{align*}
The subspace $W^{1,\h}_0(\Omega)$ is defined as the closure of $C_c^\infty(\Omega)$ with respect to $\|\cdot\|_{1,\h}$. Since the operator in problem \eqref{1.1} involves the local double phase operator, we define the generalized $N$-function $\h\colon  \Omega\times [0,\infty) \to[0, \infty)$ defined by
\begin{align*}
	\h(x,t):=t^{p}+a(x)t^{q}
\end{align*}
under hypothesis \eqref{H1}\eqref{H1i}. It is well known that $W^{1,\h}(\Omega)$ and $W^{1,\h}_0(\Omega)$ are separable, reflexive Banach spaces, see Colasuonno--Squassina \cite[Proposition 2.14]{Colasuonno-Squassina-2016}.

The following result can be found in Liu--Dai \cite[Proposition 2.1]{Liu-Dai-2018}

\begin{lemma}
	Let $u\in L^{\h}(\Om)$, define the modular
	\begin{align*}
		\rho_{\h}(u)=\int_{\Omega}\h(x,\ve{u})\dx,
	\end{align*}
	and let $c>0$. Then the following assertions hold:
	\begin{enumerate}
		\item[\textnormal{(i)}]
			for $u \neq 0$, one has $\mm{u}_{\h}=c$ if and only if $\rho_{\h}(u/c)=1$;
		\item[\textnormal{(ii)}]
			$\mm{u}_{\h}<1$ (resp.\,$=1,>1$) if and only if $\rho_{\h}(u)<1$ (resp.\,$=1,>1$);
		\item[\textnormal{(iii)}]
			if $\mm{u}_{\h}<1,$ then $\mm{u}_{\h}^{q}\leq \rho_{\h}(u) \leq \mm{u}_{\h}^{p}$;
		\item[\textnormal{(iv)}]
			if $\mm{u}_{\h}>1,$ then $\mm{u}_{\h}^{p}\leq \rho_{\h}(u) \leq \mm{u}_{\h}^{q}$;
		\item[\textnormal{(v)}]
			$\mm{u}_{\h} \to 0$ if and only if $\rho_{\h}(u) \to 0$.
	\end{enumerate}
\end{lemma}

Since problem \eqref{1.1} also involves the fractional Laplacian together with a double phase operator, we introduce the function space
\begin{align}\label{space-W}
	W := \overline{C_c^{\infty}(\Omega)}^{\,\|\cdot\|_{L^{\xi}}},
\end{align}
where the Luxemburg type norm is defined by
\begin{align*}
	\|u\|_{L^{\xi}} := \inf\left\{\tau>0 \colon  \xi\left(\tfrac{u}{\tau}\right) \le 1 \right\},
\end{align*}
and the corresponding modular is given by
\begin{align*}
	\xi(u) := \int_{\R^N}\int_{\R^N} \frac{|u(x)-u(y)|^p}{|x-y|^{N+sp}}\dxy
	+ \int_{\R^N}\Big(|\nabla u|^{p} + a(x)|\nabla u|^{q}\Big)\dx.
\end{align*}
For simplicity, in the sequel we denote $\|\cdot\|_{L^\xi}$ by $\|\cdot\|_{W}$.

The proof of the following lemma follows along the same lines as in Liu--Dai \cite[Proposition 2.1]{Liu-Dai-2018}.

\begin{lemma}\label{rel}
	Let $u\in W$ and $c>0$. Then the following assertions hold:
	\begin{enumerate}
		\item[\textnormal{(i)}]
			for $u \neq 0$ one has $\mm{u}_{L^{\xi}}=c$ if and only if $\xi(u/c)=1$;
		\item[\textnormal{(ii)}]
			$\mm{u}_{L^{\xi}}<1$ (resp.\,$=1,>1$) if and only if $\xi(u)<1$ (resp.\,$=1,>1$);
		\item[\textnormal{(iii)}]
			if $\mm{u}_{L^{\xi}}<1$, then $\mm{u}_{L^{\xi}}^{q}\leq \xi(u) \leq \mm{u}_{L^{\xi}}^{p}$;
		\item[\textnormal{(iv)}]
			if $\mm{u}_{L^{\xi}}>1$, then $\mm{u}_{L^{\xi}}^{p}\leq \xi(u) \leq \mm{u}_{L^{\xi}}^{q}$;
		\item[\textnormal{(v)}]
			$\mm{u}_{L^{\xi}} \to 0$ if and only if $\xi(u) \to 0$.
	\end{enumerate}
\end{lemma}

Here and in the sequel, every function $u \in W^{1,\mathcal{H}}_0(\Omega)$
is identified with its zero extension to $\mathbb{R}^N$. Moreover, we extend the coefficient $a \colon \Omega \to [0,\infty)$ by setting
\begin{align*}
	a(x) := 0 \quad \text{for } x \in \mathbb{R}^N\setminus\Omega,
\end{align*}
and we continue to denote this extension by $a$. Accordingly, the generalized $N$-function $\mathcal{H} \colon \Omega \times [0,\infty) \to [0,\infty)$
is extended to $\mathbb{R}^N \times [0,\infty)$ by defining
\begin{align*}
	\mathcal{H}(x,t) :=
	\begin{cases}
		\mathcal{H}(x,t), & \text{if } (x,t) \in \Omega \times [0,\infty), \\
		0, & \text{if } (x,t) \in (\mathbb{R}^N\setminus\Omega) \times [0,\infty).
	\end{cases}
\end{align*}
With this convention, both $a$ and $\mathcal{H}$ are understood as functions defined on the whole space $\mathbb{R}^N$.

We now establish the following result.

\begin{proposition}\label{prop-def-w}
	Let $\Omega \subset \R^N$ be a bounded domain. Then the norms
	$\|\cdot\|_{L^{\xi}}$ and $\|\cdot\|_{1,\h}$ are equivalent on $C_c^{\infty}(\Omega)$. In particular,
	\begin{align}\label{space-W-2}
		W = \overline{C_c^{\infty}(\Omega)}^{\,\|\cdot\|_{1,\h}}
		= \left\{u\in W^{1,\h}(\R^N)\colon  u|_{\Omega}\in W_0^{1,\h}(\Omega), u=0  \text{ in } \R^N\setminus \Omega \right\}.
	\end{align}
	Moreover, $W$ is a reflexive Banach space, and the embedding
	$W \hookrightarrow L^r(\Omega)$ is continuous for any $r \in [1,p^*]$ and compact for any $r \in [1,p^*)$.
\end{proposition}

\begin{proof}
	Let $u \in C_c^{\infty}(\Omega)$. By Liu--Dai \cite[Lemma 2.7]{Liu-Dai-2020}, the embedding $W^{1,\h}(\R^N)\hookrightarrow W^{1,p}(\R^N)$ is continuous. Furthermore, from Di Nezza--Palatucci--Valdinoci \cite[Proposition 2.2]{DiNezza-Palatucci-Valdinoci-2012}, $W^{1,p}(\R^N)$ embeds continuously into $W^{s,p}(\R^N)$. Hence, there exist constants $c_1,c_2>0$ such that
	\begin{align}\label{eq:2.1}
		[u]_{s,p} \leq c_1\big(\|u\|_{p}+\|\nabla u\|_{p}\big)
		\leq c_2\big(\|u\|_{\h}+\|\nabla u\|_{\h}\big).
	\end{align}
	Set $\tau_1=[u]_{s,p}$ and $\tau_2=\|\nabla u\|_{\h}$. Using the definition of the modular $\xi$, we compute
	\begin{align*}
		\xi \left(\frac{u}{3(\tau_1+\tau_2)}\right)
		& = \int_{\R^N}\int_{\R^N}\left|\frac{u(x)-u(y)}{3(\tau_1+\tau_2)}\right|^p \frac{1}{|x-y|^{N+sp}}\dxy \\
		& \quad + \int_{\R^N}\left(\left|\frac{\nabla u}{3(\tau_1+\tau_2)}\right|^p + a(x)\left|\frac{\nabla u}{3(\tau_1+\tau_2)}\right|^q\right)\dx \\
		& \leq \frac{1}{3^p}\left(\frac{\tau_1}{\tau_1+\tau_2}\right)^p +  \frac{1}{3^p}\left(\frac{\tau_2}{\tau_1+\tau_2}\right)^p+ \frac{1}{3^q}\left(\frac{\tau_2}{\tau_1+\tau_2}\right)^q \leq \frac{2}{3^p}\leq 1,
	\end{align*}
	which implies
	\begin{align*}
		\|u\|_{L^{\xi}} \leq 3( \tau_1+\tau_2)=3([u]_{s,p}+\|\nabla u\|_{\h}).
	\end{align*}
	Combining this with \eqref{eq:2.1}, we obtain
	\begin{align}\label{eq:2.2}
		\|u\|_{L^{\xi}} \leq c \big(\|u\|_{\h}+\|\nabla u\|_{\h}\big).
	\end{align}

	Since $C_c^{\infty}(\Omega)\subset W_0^{1,\h}(\Omega)$, the Poincar\'{e} inequality gives
	\begin{align*}
		\|u\|_{\h}\leq c \|\nabla u\|_{\h},
	\end{align*}
	see Crespo-Blanco--Gasi\'{n}ski--Harjulehto--Winkert \cite[Proposition 2.18]{Crespo-Blanco-Gasinski-Harjulehto-Winkert-2022}.
	Moreover, by definition of the norms, $\|\nabla u\|_{\h}\leq \|u\|_{L^{\xi}}$.
	Therefore,
	\begin{align}\label{eq:2.3}
		\|u\|_{\h}+\|\nabla u\|_{\h} \leq c \|u\|_{L^{\xi}}.
	\end{align}
	From \eqref{eq:2.2} and \eqref{eq:2.3}, the equivalence of the norms
	$\|\cdot\|_{L^{\xi}}$ and $\|\cdot\|_{1,\h}$ on $C_c^{\infty}(\Omega)$ follows. Since $\Omega$ is bounded and the norms are equivalent on $C_c^{\infty}(\Omega)$, the characterization of $W$ in \eqref{space-W-2} follows. Because $W^{1,\h}(\R^N)$ is separable and reflexive and $W$ is a closed subspace, $W$ is also reflexive.

	Finally, the embedding $W_0^{1,\h}(\Omega) \hookrightarrow L^r(\Omega)$ is continuous for $r \in [1,p^*]$ and compact for $r \in [1,p^*)$, see Liu--Dai \cite{Liu-Dai-2018}. Consequently, the same embedding properties hold for $W$.
\end{proof}

Since $W \hookrightarrow L^\ell(\Omega)$ for every $\ell \in [1,p^*]$, there exists a constant $C_\ell>0$ such that
\begin{align*}
	\|u\|_{\ell} \leq C_\ell \|u\|_{W}\quad \text{for all }u \in W.
\end{align*}

For any generalized $N$-function $\fp\colon \Omega\times\Omega\times \R\to[0,\infty),$ we define
the function $ g_x\colon \Omega\times [0,\infty)\to[0,\infty)$ such that
\begin{align*}
	g_x(x,t)=g(x,x,t)\quad\text{and}\quad \fp_x(x,t)=\int_{0}^{|t|} g_x(x,\tau)\tau \,\mathrm{d}\tau.
\end{align*}
For a given generalized $N$-function $\fp$ and $s\in(0,1)$, the fractional Musielak-Orlicz Sobolev space is defined by
\begin{align*}
	W^{s,\fp}(\Omega)=\left\{ u\in L^{\fp_{x}}(\Omega)\colon \int_{\Omega}\int_{\Omega}\fp\left(x,y, \dfrac{\tau|u(x)-u(y)|}{|x-y|^{s}}\right)\dfrac{\dxy}{|x-y|^{N}} <\infty \text{ for some } \tau>0\right\}.
\end{align*}
The space $W^{s,\fp}(\Omega)$ is equipped with the  norm
\begin{align*}
	\| u\|_{W^{s,\fp}(\Omega)}=\| u\|_{L^{\fp_{x}}(\Omega)}+[u]_{s,\fp,\Omega},
\end{align*}
where
\begin{align*}
	[u]_{s,\fp,\Omega}=\inf\left\{ \tau>0\colon \int_{\Omega}\int_{\Omega}\fp\left(x,y, \dfrac{|u(x)-u(y)|}{\tau|x-y|^{s}}\right)\dfrac{\dxy}{|x-y|^{N}}\leq 1\right\}.
\end{align*}
Recall that if $\mathcal{G}$ is a generalized $N$-function satisfying the $\Delta_{2}$-condition, then $W^{s,\mathcal{G}}(\Omega)$ is a separable and reflexive Banach space, see Azroul--Benkirane--Shimi--Srati \cite{Azroul-Benkirane-Shimi-Srati-2022}.

Since problem \eqref{frac_dbl} involves the fractional double phase structure, we introduce the generalized $N$-function $\fp\colon \Omega\times\Omega\times [0,\infty) \to[0, \infty)$ by
\begin{align*}
	\fp(x,y,t):=t^{p}+b(x,y)t^{q},
\end{align*}
under hypothesis \eqref{H2}\eqref{H2i}. It is straightforward to verify that the generalized $N$-function $\fp(x,y,t)=t^{p}+b(x,y)t^{q}$ satisfies the $\Delta_{2}$-condition and is locally integrable. Because problem \eqref{frac_dbl} involves the nonlocal operator under Dirichlet boundary conditions, we define the associated fractional Sobolev type space
\begin{align*}
	X(\Omega)=\left\lbrace u\colon \R^N\to\R \colon  u|_{\Omega}\in L^{\fp_{x}}(\Omega),\int_{Q}\fp\left(x,y, \dfrac{\tau|u(x)-u(y)|}{|x-y|^{s}}\right)\dfrac{\dxy}{|x-y|^{N}} <\infty   \text{ for some } \tau>0\right\rbrace,
\end{align*}
where $Q=\R^{2N}\setminus (\Omega^c\times \Omega^c)$. The space $X(\Omega)$ is equipped with the norm
\begin{align*}
	\| u\|_{X(\Omega)}=\| u\|_{L^{\fp_{x}}(\Omega)}+[u]_{X},
	\end{align*}
where
\begin{align*}
	[u]_{X}=\inf\left\{  \tau>0\colon \ \int_{Q}\fp\left(x,y, \dfrac{|u(x)-u(y)|}{\tau|x-y|^{s}}\right)\dfrac{\dxy}{|x-y|^{N}}\leq 1\right\}.
\end{align*}
We define the subspace
\begin{align*}
	W_{0}^{s,\mathcal{G}}(\Omega)=\{u\in X(\Omega)\colon  u=0 \text{ a.e.\,in } \  \R^{N}\setminus \Omega\},
\end{align*}
which is a normed space with norm
\begin{align*}
	\| u\|_{W_{0}^{s,\mathcal{G}}(\Omega)}=\| u\|_{L^{\fp_{x}}(\Omega)}+[u]_{X}.
\end{align*}
If $u\in W_{0}^{s,\mathcal{G}}(\Omega)$, then
\begin{align*}
	\int_{Q}\fp\left(x,y, \dfrac{|u(x)-u(y)|}{\tau|x-y|^{s}}\right)\dfrac{\dxy}{|x-y|^{N}}=\int_{\R^N}\int_{\R^N}\fp\left(x,y, \dfrac{|u(x)-u(y)|}{\tau|x-y|^{s}}\right)\dfrac{\dxy}{|x-y|^{N}}.
\end{align*}
since the integrand vanishes whenever both points lie outside $\Omega$. Consequently,
\begin{align*}
	W_{0}^{s,\mathcal{G}}(\Omega)=\left\{u\in W^{s,\fp}(\R^N)\colon  u=0 \ \text{ a.e.\,in } \R^{N}\setminus \Omega\right\},
\end{align*}
which is a well-known space studied by Azroul--Benkirane--Shimi--Srati \cite{Azroul-Benkirane-Shimi-Srati-2023}.

Next, we state the generalized Poincar\'{e}'s inequality, see Azroul--Benkirane--Shimi-Srati.

\begin{theorem}
	Let $\Omega$ be a bounded domain of $ \R^{N}$ with smooth boundary and $0<s<1.$ Then there exists a positive constant $c$ such that
	\begin{align*}
		\| u\|_{L^{\mathcal{G}_{x}}(\Omega)}\leq c [u]_{s,\mathcal{G},\R^N}\quad \text{for all } u\in  W_{0}^{s,\mathcal{G}}(\Omega).
	\end{align*}
\end{theorem}
As a consequence, the seminorm $[\cdot]_{s,\mathcal{G},\R^N}$ defines a norm on $W_{0}^{s,\mathcal{G}}(\Omega)$, which is equivalent to the norm $\| \cdot\|_{W^{s,\mathcal{G}}(\Omega)}$.

From now on, we assume that $\Omega$ is a smooth bounded domain in $\R^N$.

Taking into account the interplay between the local $p$-Laplacian and the fractional double phase operator in problem \eqref{frac_dbl}, we introduce the space
\begin{align}\label{space-E}
	E := W_0^{1,p}(\Omega)\cap W_{0}^{s,\mathcal{G}}(\Omega),
\end{align}
which is a Banach space endowed with the norm,
\begin{align}\label{space-E-norm}
	\|u\|_{E} = \|u\|_{W_0^{1,p}(\Omega)}+\|u\|_{W_{0}^{s,\mathcal{G}}(\Omega)}
	= \|\nabla u\|_{p} + [u]_{s,\mathcal{G},\R^N}.
\end{align}
This choice is natural, since in general there is no continuous embedding between $W_0^{1,p}(\Omega)$ and  $W_{0}^{s,\mathcal{G}}(\Omega)$. Hence neither space alone captures the full structure of the operator in \eqref{frac_dbl}. The intersection space $E$ therefore provides the appropriate functional framework for the variational analysis of the problem.

By definition, there is a continuous embedding $E\hookrightarrow W_0^{1,p}(\Omega)$. Since $W_0^{1,p}(\Omega)$ is continuously embedded in $L^{r}(\Omega)$ for $1\leq r\leq p^*$ and compactly embedded for $1\leq r<p^*$, we deduce that
\begin{align}\label{embed}
	E \hookrightarrow L^{r}(\Omega) \quad \text{continuously for } 1 \leq r\leq p^*, \quad E \hookrightarrow L^{r}(\Omega) \quad \text{compactly for } 1 \leq r<p^*.
\end{align}
Hence, for every $r \in [1,p^*]$, there exists a constant $\mathbb{S}_r>0$ such that
\begin{align*}
	\|u\|_{r} \leq \mathbb{S}_r \|u\|_{E} \quad\text{for all } u \in E.
\end{align*}
Define the modular functional
\begin{align*}
	\eta(u) := \int_{\Omega}|\nabla u|^p\dx+ \int_{Q}|D_s(u)|^p \dv+\int_{Q}b(x,y)|D_s(u)|^q \dv,
\end{align*}
where we used the notation introduced in \eqref{notion-ds-dv}.

The proof of the following lemma is analogous to that of Liu--Dai \cite[Proposition 2.1]{Liu-Dai-2018} and is therefore omitted.

\begin{lemma}\label{re2}
	Let $u\in E$ and $c>0$. Then the following assertions hold:
	\begin{enumerate}
		\item[\textnormal{(i)}]
			for $u \neq 0$ one has $\mm{u}_{E}=c$ if and only if $\eta(u/c)=1$;
		\item[\textnormal{(ii)}]
			$\mm{u}_{E}<1$ (resp.\,$=1,>1$) if and only if $\eta(u)<1$ (resp.\,$=1,>1$);
		\item[\textnormal{(iii)}]
			if $\mm{u}_{E}<1$, then $\mm{u}_{E}^{q}\leq \eta(u) \leq \mm{u}_{E}^{p}$;
		\item[\textnormal{(iv)}]
			if $\mm{u}_{E}>1$, then $\mm{u}_{E}^{p}\leq \eta(u) \leq \mm{u}_{E}^{q}$;
		\item[\textnormal{(v)}]
			$\mm{u}_{E} \to 0$ if and only if $\eta(u) \to 0$.
	\end{enumerate}
\end{lemma}

\begin{remark}
	Define the Luxemburg-type norm associated with the modular functional $\eta$ by
	\begin{align*}
		\|u\|_{\eta} := \inf\left\{ \tau>0 \colon \eta\left(\frac{u}{\tau}\right)\le 1 \right\}.
	\end{align*}
	Then $\|\cdot\|_{\eta}$ is a norm on $E$, and it is equivalent to the norm \eqref{space-E-norm}. Moreover, the norm-modular relations stated in Lemma \ref{re2} hold with respect to the Luxemburg norm $\|\cdot\|_{\eta}$. Since equivalent norms generate the same topology on $E$, we may use $\|\cdot\|_{E}$ and $\|\cdot\|_{\eta}$ interchangeably in the sequel.
\end{remark}

Next, we recall several notions that will be used in the variational analysis. We begin with the definition of the genus and the related topological preliminaries, see Rabinowitz \cite{Rabinowitz-1986}.

\begin{definition}
	Let $X$ be a Banach space and $A \subset X\setminus\{0\}$. The set $A$ is said to be symmetric with respect to the origin if $u\in A$ implies $-u \in A$. The genus $\ga(A)$ of $A$ is defined as the least positive integer $n$ such that there exists an odd mapping $\phi \in C(A,\R^n\setminus\{0\})$. If no such $n$ exists, then $\gamma(A)=\infty$.
\end{definition}

Define
\begin{align*}
	\Sigma=\left\{U\subset X \setminus \{0\}\colon U \ \text{is closed in } X \text{ and symmetric with respect to the origin}\right\}.
\end{align*}

\begin{proposition} \label{genus}
	Let $A,B \in \Sigma$. Then the following assertions hold:
	\begin{enumerate}
		\item[\textnormal{(i)}]
			If there exists an odd continuous mapping from $A$ to $B$, then $\gamma(A) \leq \gamma(B).$
		\item[\textnormal{(ii)}]
			If there exists an odd homeomorphism from $A$ to $B$, then $\gamma(A) = \gamma(B).$
		\item[\textnormal{(iii)}]
			If $\gamma(B) < \infty$, then $\gamma(A \setminus B) \geq \gamma(A) - \gamma(B)$.
		\item[\textnormal{(iv)}]
			The $n$-dimensional sphere $S^n$ has genus $n+1$ by the Borsuk-Ulam theorem.
		\item[\textnormal{(v)}]
			If $A$ is compact, then $\gamma(A) < \infty$ and there exists $\delta > 0$ such that
			\begin{align*}
				N_\delta(A):= \{ x \in X \colon  \operatorname{dist}(x,A) \leq \delta \} \in\Sigma\quad \text{and}\quad \gamma(N_\delta(A)) = \gamma(A).
			\end{align*}
	\end{enumerate}
\end{proposition}

\begin{definition}
	Let $X$ be a Banach space and $I\colon X\to \R $ a $C^{1}$-functional.
	\begin{enumerate}
		\item[\textnormal{(i)}]
			A sequence $\{u_{n}\}_{n\in\mathbb{N}}\subset X$ is called a Cerami sequence at level $c$ (\textnormal{(C)$_c$}-sequence for short), if
			\begin{align*}
				I(u_{n})\to c\quad \text{and}\quad (1+\|u_{n}\|)\|I'(u_{n})\|_{*}\to 0.
			\end{align*}
			We say that $I$ satisfies the Cerami condition at level $c$ (\textnormal{(C)$_c$}-condition for short) if every \textnormal{(C)$_c$}-sequence has a convergent subsequence. Moreover, $I$ satisfies the Cerami condition (\textnormal{(C)}-condition for short) if it satisfies the \textnormal{(C)$_c$}-condition for every $c\in\mathbb{R}$.
		\item[\textnormal{(ii)}]
			A sequence $\{u_{n}\}_{n\in\mathbb{N}}\subset X$  is called a  Palais-Smale sequence at level $c$ (\textnormal{(PS)$_c$}-sequence for short) if it satisfies
			\begin{align*}
				I(u_n)\to c \quad \text{and} \quad \|I'(u_n)\|_{*}\to 0.
			\end{align*}
			We say that $I$ satisfies the Palais-Smale condition at level $c$ (\textnormal{(PS)$_c$}-condition for short) if every \textnormal{(PS)$_c$}-sequence admits a convergent subsequence. Moreover, the functional $I$ satisfies the Palais-Smale condition  (\textnormal{(PS)}-condition for short) if it satisfies the \textnormal{(PS)$_c$}-condition for every $c\in\mathbb{R}$.
\end{enumerate}
\end{definition}

We recall the following versions of the mountain pass theorem, see Ambrosetti--Rabinowitz \cite{Ambrosetti-Rabinowitz-1973} and Cerami \cite{Cerami-1978,Cerami-1980}.

\begin{theorem}\label{ce2}
	Let $X$ be a Banach space and let $I\colon  X\to\mathbb{R}$ be a $C^1$-functional satisfying $I(0)=0$. Suppose that there exist constants $\rho, \alpha>0$ such that
	\begin{enumerate}
		\item[\textnormal{(i)}]
			$I(u)\geq \alpha$ whenever $\|u\|=\rho$;
		\item[\textnormal{(ii)}]
			there exists $v\in X$ with $\|v\|>\rho$ such that $I(v)\leq 0$.
	\end{enumerate}
	Then $I$ admits a \textnormal{(C)$_c$}-sequence at the level
	\begin{align}\label{C-condition}
		c=\inf_{\gamma\in\Gamma}\max_{t\in[0,1]} I(\gamma(t)),
	\end{align}
	where $\Gamma=\{\gamma \in C([0,1],X)\colon \gamma(0)=0, \gamma(1)=v\}$.
\end{theorem}

\begin{theorem}\label{ce1}
	Let $X$ be a Banach space and let $I\colon X\to\mathbb{R}$ be a $C^1$-functional such that $I(0)=0$ and $I$ satisfies the \textnormal{(C)$_c$}-condition with $c$ as in \eqref{C-condition}. Suppose that there exist $\rho, \alpha>0$ such that
	\begin{enumerate}
		\item[\textnormal{(i)}]
			$I(u)\geq \alpha$ whenever $\|u\|=\rho$;
		\item[\textnormal{(ii)}]
			there exists $v\in X$ with $\|v\|>\rho$ such that $I(v)\leq 0$.
	\end{enumerate}
	Then
	\begin{align*}
		c:=\inf_{\gamma\in\Gamma}\max_{t\in[0,1]} I(\gamma(t))>\alpha
	\end{align*}
	is a critical value of $I$ with $\Gamma$ as in Theorem \ref{ce2}.
\end{theorem}

Next, we recall the quantitative deformation lemma and the Poincar\'{e}–Miranda existence theorem, which will be employed to obtain a least energy sign-changing solution for the problem \eqref{frac_dbl}, see Willem \cite[Lemma 2.7]{Willem-1996} and Dinca--Mawhin \cite[Corollary 2.2.15]{Dinca-Mawhin-2021}, respectively.

\begin{lemma}\label{defm}
	Let $X$ be a Banach space, $\varphi \in C^{1}(X;\mathbb{R})$, $\emptyset \neq M \subseteq X$,
	$c \in \mathbb{R}$, $\varepsilon, \delta_0 > 0$ such that for all
	$u \in \varphi^{-1}([c - 2\varepsilon, c + 2\varepsilon]) \cap M_{2 \delta_0}$, there holds $\|\varphi'(u)\|_* \geq 8\varepsilon/\delta_0$, where
	$M_{r} = \{y \in X \colon d(y,M) = \inf_{u \in M} \|u - y\| < r \}$
	for any $r> 0$. Then there exists $\Pi \in C([0,1] \times X; X)$ such that
	\begin{enumerate}
		\item[\textnormal{(i)}]
			$\Pi(t,u) = u$, if $t = 0$ or if $u \notin \varphi^{-1}([c - 2\varepsilon, c + 2\varepsilon]) \cap M_{2\delta_0}$;
		\item[\textnormal{(ii)}]
			$\varphi(\Pi(1,u)) \leq c-\varepsilon$ for all $u \in\varphi^{-1}((-\infty,c+\varepsilon])\cap M$;
		\item[\textnormal{(iii)}]
			$\Pi(t,\cdot)$ is a homeomorphism of $X$ for all $t\in [0,1]$;
		\item[\textnormal{(iv)}]
			$\mm{\Pi(t,u)-u}\leq \delta_0$ for all $u \in X$ and $t \in [0,1]$;
		\item[\textnormal{(v)}]
			$\varphi(\Pi(\cdot,u))$ is decreasing for all $u\in X$;
		\item[\textnormal{(vi)}]
			$\varphi(\Pi(t,u))<c$ for all $u\in \varphi^{-1}((-\infty,c])\cap M_{\delta_0}$ and $t\in[0,1]$.
	\end{enumerate}
\end{lemma}

\begin{theorem} \label{IVP}
	Let $P=[-a_1,a_1] \times \cdots \times [-a_N,a_N]$ with $a_i > 0$ for $i=1,\ldots,N$
	and let $d \colon P \to \R^N$ be continuous. If for each $i \in \{1,\ldots,N\}$ one has
	\begin{align*}
		d_i(t) \le 0 \quad &\text{when } t \in P \text{ and } t_i = -a_i,\\
		d_i(t) \ge 0 \quad &\text{when } t \in P \text{ and } t_i = a_i.
	\end{align*}
	Then $d$ has at least one zero point in $P$.
\end{theorem}

We also recall the following known estimates, see Crespo-Blanco \cite[Lemma 1.1]{Crespo-Blanco-2025}.

\begin{lemma}\label{inequalities}
	Let $r > 1$, for any $\xi, \eta \in \R^N$
	\begin{align*}
		\left( |\xi|^{r-2} \xi - |\eta|^{r-2} \eta \right)
		\cdot \left( \xi - \eta \right) \geq
		C_r |\xi - \eta|^r, \quad&\text{if }r \geq 2,\\
		\left( |\xi| + |\eta| \right)^{2 - r} \left(  |\xi|^{r-2} \xi - |\eta|^{r-2} \eta \right)
		\cdot \left( \xi - \eta \right)  \geq
		C_r |\xi - \eta|^2 \quad&\text{if }1 < r < 2
	\end{align*}
	where
	\begin{align*}
		C_r =
		\begin{cases}
			\min \{ 2^{2-r}, 2^{-1} \} & \text{if } r \geq 2,  \\
			r-1& \text{if } 1 < r < 2.
		\end{cases}
	\end{align*}
	The constant $C_r$ is not necessarily optimal.
\end{lemma}

\section{\texorpdfstring{$p$}{p}-fractional Laplacian with double phase}\label{Section_3}

In this section, we study problem \eqref{1.1} and prove Theorems \ref{t1}--\ref{t3}. Throughout this section, we assume that \eqref{H1} holds. The associated energy functional $I_{\la}\colon W\to\R$ corresponding to \eqref{1.1} is defined by
\begin{align*}
	I_{\la}(u)= & \frac{1}{p}\int_{\R^N}\int_{\R^N} \dfrac{|u(x)-u(y)|^p}{|x-y|^{N+sp}}\dxy+\int_{\Omega}\lt(\frac{|\nabla u|^p}{p}+a(x)\frac{|\nabla u|^q}{q}\ri)\dx \\
	& -\lambda\int_{\Omega}\left(\frac{w_1(x)}{k}|u|^{k}+\frac{w_2(x)}{r}|u|^{r}\right)\dx.
\end{align*}
To prove Theorem \ref{t1}, we apply the Nehari manifold approach. First, we define
\begin{equation}\label{def-Ap-Bl}
	\begin{aligned}
	A_p(u,v)&=\int_{\R^N}\int_{\R^N} \dfrac{|u(x)-u(y)|^{p-2}(u(x)-u(y))(v(x)-v(y))}{|x-y|^{N+ps}}\dxy,\\ B_\ell(a(x),u,v)&=\int_{\Omega}a(x) |\nabla u|^{\ell-2}\nabla u\cdot \nabla v\dx.
	\end{aligned}
\end{equation}
Next, we fix the constants
\begin{align}
	\la_*&:=\lt(\frac{(r-p)}{(r-k)}\ri)^{\frac{r-p}{r-k}}\lt(\frac{1}{\mm{w_1}_{\frac{r}{r-k}}S_r^k}\ri)^{\frac{r-p}{r-k}} \lt(\frac{(p-k)}{(r-k) \mm{w_2}_{\infty} S_r^r}\ri)^{\frac{(p-k)}{(r-k)}}, \label{lambda1}\\
	\overline{\la}_*&:=\frac{(p^*-p)}{(p^*-k)}\frac{1}{S_{p^*}^k\mm{w_1}_{\frac{p^*}{p^*-k}}}\lt(\frac{(p-k)}{(r-k)S_{p^*}^{p^*}}\ri)^{\frac{p-k}{p^*-p}},\label{lambda2}\\
	\tilde{c} &= \left(\frac{1}{k}-\frac{1}{r}\right)\left[\left(\frac{p}{k}\left(\frac{1}{p}-\frac{1}{r}\right)\left(\frac{1}{k}-\frac{1}{r}\right)^{-1}\right)^{\frac{-k}{p}}\|w_1\|_{r/r-k} S_r^{k}\right]^{\frac{p}{p-k}}.\label{tilde{c}}
\end{align}
The Nehari manifold associated with $I_{\la}$ is defined by
\begin{align*}
	N_{\la}:=\{u\in W\setminus \{0\}\colon \langle I'_{\la}(u),u \rangle=0\}
\end{align*}
which is equivalent to
\begin{align*}
	N_{\la}=\left\{u\in W\setminus \{0\}\colon  [u]^p_{s,p}+\mm{\nabla u}_p^p+\mm{\na u}^q_{q,a} =\lambda\int_{\Omega}(w_1(x)|u|^{k}+w_2(x)|u|^{r})\dx\right\}.
\end{align*}
For each $u \in W$, we introduce the fibering map $\si_u\colon [0,\infty) \to \R $ defined by $\si_u(t):=I_{\la}(tu)$. A direct computation gives
\begin{align*}
	\si_u(t)& =\frac{t^p}{p}[u]^p_{s,p}+\int_{\Omega}\left(\frac{t^p}{p}\p{u}^p+a(x)\frac{t^q}{q}\p{u}^q\right)\dx \\
	&\quad  -\la \int_{\Omega}\left(w_1(x)\frac{t^k}{k}\ve{u}^k+w_2(x)\frac{t^r}{r}\ve{u}^{r}\right)\dx, \\
	\si'_u(t)  & =t^{p-1}[u]^p_{s,p}+\int_{\Omega}\left(t^{p-1}\p{u}^p+a(x) t^{q-1}\p{u}^q\right)\dx\\
	&\quad -\la  \int_{\Omega}\left(w_1(x) t^{k-1}\ve{u}^k+t^{r-1}w_2(x)\ve{u}^{r}\right)\dx, \\
	\si''_u(t) & =(p-1)t^{p-2}([u]_{s,p}^p+\mm{\na u}_p^p)+(q-1)t^{q-2}\mm{\na u}_{q,a}^q\\
	&\quad  -\la \int_{\Omega}\bigg((k-1)t^{k-2}w_1(x)\ve{u}^k +(r-1)t^{r-2}w_2(x)\ve{u}^{r}\bigg)\dx.
\end{align*}
By the definition of the fibering map, it follows that $tu\in \nl$ if and only if $\si_u'(t)=0$.

\begin{lemma}\label{coer}
	The functional $I_{\la}$ is coercive and bounded from below on the Nehari manifold $\nl$.
\end{lemma}

\begin{proof}
	Let $u\in \nl$. Then, by applying Lemma \ref{rel} and H\"{o}lder's inequality, we have
	\begin{align*}
		I_{\la}(u) & =\lt(\frac{1}{p}-\frac{1}{r}\ri)([u]_{s,p}^p+\mm{\na u}_p^p)+\lt(\frac{1}{q}-\frac{1}{r}\ri)\mm{\na{u}}_{q,a}^{q}-\la\lt(\frac{1}{k}-\frac{1}{r}\ri)\int_{\Omega}w_1(x)\ve{u}^k\dx \\
		& \geq \lt(\frac{1}{q}-\frac{1}{r}\ri)\xi(u)-\la\lt(\frac{1}{k}-\frac{1}{r}\ri) \mm{w_1}_{{\frac{r}{r-k}}}\mm{u}_r^k \\
		& \geq  \lt(\frac{1}{q}-\frac{1}{r}\ri)\min(\mmm{u}^p, \mmm{u}^q)-\la\lt(\frac{1}{k}-\frac{1}{r}\ri) \mm{w_1}_{{\frac{r}{r-k}}} C_r^k\mmm{u}^k.
	\end{align*}
	Since $1<k<p<q$, it follows that $I_{\la}(u)\to\infty$ as  $\mmm{u}\to\infty$. Therefore, the functional $I_{\la}$ is coercive and bounded from below on the Nehari manifold $\nl$.
\end{proof}

We now decompose the Nehari manifold  $\nl$ into three disjoint subsets according to the sign of the second derivative of the associated fibering map $\si_u$. Specifically, we define
\begin{align*}
	\nl^+&:=\{u\in \nl\colon  \si_u''(1)>0\},\\
	\nl^-&:=\{u\in \nl\colon  \si_u''(1)<0\},\\
	\nl^0&:=\{u\in \nl\colon  \si_u''(1)=0\}.
\end{align*}
Clearly, we have the decomposition
\begin{align*}
	\nl=\nl^+\cup\nl^-\cup \nl^0.
\end{align*}
We also define the corresponding infimum levels of the energy functional on these subsets by
\begin{align*}
	m_{\la}=\displaystyle\inf_{u\in \nl}I_{\la}(u)\quad\text{and}\quad m_{\la}^{\pm}=\displaystyle\inf_{u \in \nl^{\pm}}I_{\la}(u).
\end{align*}

We now establish the following structural property of the Nehari manifold. The first result is standard and can be proved as in the paper by Dr\'{a}bek--Pohozaev \cite{Drabek-Pohozaev-1997}.

\begin{lemma}
	If $u$ is a minimizer of $I_{\la}$ on $\nl$ and $u\not\in\nl^0$, then $u$ is a critical point of $I_{\la}$.
\end{lemma}

\begin{lemma}\label{empty}
	Let $\la \in (0, \la_*)$, where $\la_*$ is defined in \eqref{lambda1}. Then the set $\nl^0$ is empty.
\end{lemma}

\begin{proof}
	Assume by contradiction that $\nl^0 \ne \emptyset$. Then there exists $u\in\nl^0$. By the definition of $\nl^0$, we have
	\begin{align}
		[u]^p_{s,p}+\mm{\nabla u}_p^p+\mm{\na u}^q_{q,a} & =\lambda\int_{\Omega}(w_1(x)|u|^{k}+w_2(x)|u|^{r})\dx,\label{eq:nehari1}\\
		(p-1)([u]^p_{s,p}+\mm{\nabla u}_p^p)+(q-1)\mm{\na u}^q_{q,a} & =\lambda\int_{\Omega}((k-1)w_1(x)|u|^{k}+(r-1)w_2(x)|u|^{r})\dx.\label{eq:nehari2}
	\end{align}
	Subtracting suitable multiples of equation \eqref{eq:nehari1} from equation \eqref{eq:nehari2}, we obtain
	\begin{align}
		(p-k)[u]^p_{s,p}+(p-k)\mm{\nabla u}_p^p+(q-k)\mm{\na u}^q_{q,a}=\lambda\int_{\Omega}(r-k)w_2(x)|u|^{r}\dx,\label{eq:rel1} \\
		(r-p)[u]^p_{s,p}+(r-p)\mm{\nabla u}_p^p+(r-q)\mm{\na u}^q_{q,a}=\lambda\int_{\Omega}(r-k)w_1(x)|u|^{k}\dx.\label{eq:rel2}
	\end{align}

	Now, define the auxiliary functional
	\begin{align*}
		L_{\la}(u):=\frac{r-p}{r-k}[u]^p_{s,p}+\frac{r-p}{r-k}\mm{\nabla u}_p^p+\frac{r-q}{r-k}\mm{\na u}^q_{q,a}-\lambda\int_{\Omega}w_1(x)|u|^{k}\dx.
	\end{align*}
	Using \eqref{eq:rel2}, we see that $L_{\la}(u)=0$ for all $u\in \nl^0$. On the other hand, by \eqref{eq:rel1} and the Sobolev embedding, we estimate
	\begin{align*}
		0\geq (p-k) \mm{\nabla u}_p^p -\la (r-k) \mm{w_2}_{\infty} S_r^r \mm{\nabla u}_p^r,
	\end{align*}
	which implies
	\begin{align*}
		\mm{\nabla u}_p \geq \lt(\frac{p-k}{\la(r-k) \mm{w_2}_{\infty} S_r^r}\ri)^{\frac{1}{r-p}}.
	\end{align*}
	Substituting this lower bound into the expression of $ L_{\la}(u)$, and using H\"{o}lder's inequality together with the Sobolev embedding once again, we obtain
	\begin{align*}
		L_{\la}(u) & \geq \frac{r-p}{r-k}\mm{\nabla u}_p^p-\lambda\int_{\Omega}w_1(x)|u|^{k}\dx\\
		& \geq \frac{r-p}{r-k}\mm{\nabla u}_p^p-\la \mm{w_1}_{\frac{r}{r-k}}S_r^k \mm{\nabla u}_p^k\\
		& \geq\mm{\nabla u}_p^k \lt(\frac{r-p}{r-k}\cdot\lt(\frac{p-k}{\la(r-k) \mm{w_2}_{\infty} S_r^r}\ri)^{\frac{p-k}{r-p}}-\la \mm{w_1}_{\frac{r}{r-k}}S_r^k\ri).
	\end{align*}
	Therefore, for every $\la \in (0,\la_*)$, where $\la_*$ is defined in \eqref{lambda1}, the term in brackets is positive. This implies that $L_{\la}(u)>0$, which contradicts the fact that $L_{\la}(u)=0$ for all $u\in\nl^0$. Hence, $\nl^0=\emptyset$ for all $\la \in (0,\la_*)$.
\end{proof}

\begin{lemma}\label{geometry}
	Let $\la \in (0, \la_*)$, where $\la_*$ is defined in \eqref{lambda1}.  For each $u\in W\setminus\{0\}$, there exists a unique number $t_m=t_m(u,\la)>0$ such that $\si_u''(t_m)=0$. Furthermore, there exist constants $t_1$ and $t_2$ satisfying $t_1<t_m<t_2$ such that $t_1u\in \nl^+$ and $t_2u\in \nl^-$. In addition, the first derivative of the fibering map satisfies
	\begin{align*}
		\si_u'(t)<0 \quad\text{for all }t\in [0,t_1)\quad\text{and}\quad \si_u'(t)>0\quad \text{for all }t\in (t_1,t_2).
	\end{align*}
\end{lemma}

\begin{proof}
	Define the auxiliary function  $\ga_u\colon (0,\infty) \to \R$ by
	\begin{align*}
		\ga_u(t):=t^{p-k}[u]_{s,p}^{p}+t^{p-k}\mm{\na u}_p^p+t^{q-k}\mm{\na u}_{q,a}^q-\la t^{r-k}\int_{\Omega}w_2(x) \ve{u}^{r}\dx.
	\end{align*}
	From the structure of the Nehari manifold, $tu\in \nl$ if and only if $\ga_u(t)=\la \int_{\Omega}w_1(x)\ve{u}^k\dx$. Moreover, if $tu\in \nl$, then $\si_{tu}''(1)=t^{k+1}\ga_u'(t)$.

	\textbf{Claim:} There exists a unique $t_m>0$ such that $\ga_u'(t_m)=0$.

	To verify this, we compute
	\begin{align*}
		\ga_u'(t) & =t^{p-k-1}\lt((p-k)[u]_{s,p}^{p}+(p-k)\mm{\na u}_p^p+(q-k)t^{q-p}\mm{\na u}_{q,a}^q-\la t^{r-p}(r-k)\int_{\Omega}w_2(x) \ve{u}^{r}\dx\ri) \\
		          & = t^{p-k-1}h_u(t),
	\end{align*}
	where we define
	\begin{align*}
		h_u(t)&:=(p-k)[u]_{s,p}^{p}+(p-k)\mm{\na u}_p^p-g_u(t),\\
		g_u(t)&:=-(q-k)t^{q-p}\mm{\na u}_{q,a}^q+\la t^{r-p}(r-k)\int_{\Omega}w_2(x) \ve{u}^{r}\dx.
	\end{align*}
	Observe that for small $t>0$, we have $g_u(t)<0$, while $g_u(t) \to \infty$ as $t\to \infty$. Therefore, $g_u$ crosses zero exactly once, and hence there exists a unique $T_0$ such that $g_u(T_0)=0$. Consequently, there exists a unique $t_m>T_0$ such that  $g_u(t_m)=(p-k)[u]_{s,p}^{p}+(p-k)\mm{\na u}_p^p$, and thus $\ga'_u(t_m)=0$. Moreover, $\ga'_u(t)>0$ for $t\in (0,t_m)$ and $\ga'_u(t)<0$ for $t\in(t_m,\infty)$. To estimate $t_m$, we use $\ga'_u(t_m)=0$ and H\"{o}lder's inequality to obtain
	\begin{align*}
		(p-k)\mm{\na u}_p^p \leq \la t_m^{r-p}(r-k)\int_{\Omega}w_2(x) \ve{u}^{r}\dx \leq \la  t_m^{r-p}(r-k) \mm{w_2}_{\infty}S_r^r\mm{\na u}_p^r.
	\end{align*}
	This implies
	\begin{align}\label{l4.31}
		t_m\geq \frac{1}{\mm{\na u}_p}\lt(\frac{(p-k)}{\la S_r^r(r-k)\mm{w_2}_{\infty}}\ri)^{\frac{1}{r-p}}:=t_{0}.
	\end{align}
	From \eqref{l4.31}, we obtain
	\begin{align*}
		\ga_u(t_m)
		& \geq \ga_u(t_{0})
		\geq t_{0}^{p-k}\mm{\na u}_p^p-\la t_{0}^{r-k}\mm{w_2}_{\infty}S_r^r\mm{\na u}_p^r \\
		& =\mm{\na u}_p^k\lt(\frac{p-k}{\la S_r^r(r-k)\mm{w_2}_{\infty}}\ri)^{\frac{p-k}{r-p}}-\mm{\na u}_p^k\la \mm{w_2}_{\infty}S_r^r\lt(\frac{p-k}{\la S_r^r(r-k)\mm{w_2}_{\infty}}\ri)^{\frac{r-k}{r-p}} \\
		& \geq \mm{\na u}_p^k\lt(\frac{1}{\la} \ri)^{\frac{p-k}{r-p}}\lt(\frac{p-k}{ S_r^r(r-k)\mm{w_2}_{\infty}}\ri)^{\frac{p-k}{r-p}}\frac{r-p}{r-k}.
	\end{align*}
	If $\la <\la_*$, then
	\begin{align*}
		0<\la \int_{\Omega}w_1(x)\ve{u}^k\dx < \la_*\mm{w_1}_{\frac{r}{r-k}}S_r^k \mm{\na u}_p^k \leq \ga_u(t_m).
	\end{align*}
	Hence, the equation $\ga_u(t)=\la \int_{\Omega}w_1(x)\ve{u}^k\dx$ admits exactly two distinct solutions. Therefore, there exist $t_1,t_2$ with $t_1<t_m<t_2$ such that $\ga_u(t_1)=\ga_u(t_2)=\la\int_{\Omega}w_1(x)\ve{u}^k\dx$, that is, $t_1u, t_2u \in \nl$. Moreover, since $\ga_u(t)$ is increasing on $(0,t_m)$ and decreasing on $(t_m, \infty)$, we have $\ga_u'(t_1)>0$ and $\ga_u'(t_2)<0$. Hence $t_1u \in \nl^+$ and $ t_2u \in \nl^-$.
\end{proof}

\begin{lemma} \label{inf}
	The quantity $m_{\la}^{+}=\displaystyle\inf_{u \in \nl^{+}}I_{\la}(u)$ is strictly negative, that is,  $m_{\la}^+<0$.
\end{lemma}

\begin{proof}
	Let $u\in \nl^+$. Then $\si_u'(1)=0$ and $\si_u''(1)>0$. From $\si_u''(1)>0$ we get
	\begin{align}\label{4.11}
		\frac{r-p}{r-k}\left([u]^p_{s,p}+\mm{\nabla u}_p^p\right)+\frac{r-q}{r-k}\mm{\nabla u}_{q,a}^q \leq \la \int_{\Omega}w_1(x)\ve{u}^k\dx.
	\end{align}
	Using $\si_u'(1)=0$ together with \eqref{4.11}, we compute
	\begin{align*}
		I_{\la}(u) & =\lt(\frac{1}{p}-\frac{1}{r}\ri)([u]_{s,p}^p+\mm{\na u}_p^p)+\lt(\frac{1}{q}-\frac{1}{r}\ri)\mm{\na{u}}_{q,a}^{q}-\la\lt(\frac{1}{k}-\frac{1}{r}\ri)\int_{\Omega}w_1(x)\ve{u}^k\dx \\
		& \leq\lt(\frac{1}{p}-\frac{1}{r}- \frac{r-p}{kr}\ri)([u]_{s,p}^p+\mm{\na u}_p^p)+\lt(\frac{1}{q}-\frac{1}{r}-\frac{r-q}{rk}\ri)\mm{\na{u}}_{q,a}^{q}\\
		& =\frac{r-p}{r}\lt(\frac{1}{p}-\frac{1}{k}\ri)([u]_{s,p}^p+\mm{\na u}_p^p)+\frac{r-q}{r}\lt(\frac{1}{q}-\frac{1}{k}\ri)\mm{\na{u}}_{q,a}^{q}<0
	\end{align*}
	since $k<p<q$. It follows that $m_{\la}^{+}=\displaystyle\inf_{u \in \nl^{+}}I_{\la}(u)<0$.
\end{proof}

\begin{lemma}\label{imp}
	Let $\la \in (0,\la_*)$, where $\la_*$ is defined in \eqref{lambda1}, and let $u \in \nl$. Then there exist $\e>0$ and a differentiable function $g\colon B_{\e}(0) \to (-\infty, \infty)$ such that $g(0)=1$ and $g(z)(u-z)\in \nl$ for all $z \in B_{\e}(0) $. Moreover,
	\begin{equation}\label{4.13}
		\begin{aligned}
			&\langle g'(0), z \rangle\\
			&=\frac{pB_p(1,u,z)+q B_q(a(x),u,z) +p A_p(u,z)-\la\int_{\Omega}(k w_1(x)\ve{u}^{k-2}+rw_2(x)\ve{u}^{r-2})uz\dx}{(p-k)([u]^p_{s,p}+\mm{\na u}_p^p)+(q-k)\mm{\na u}_{q,a}^{q}-\lambda(r-k)\int_{\Omega}w_2(x)\ve{u}^r\dx},
		\end{aligned}
	\end{equation}
	where $A_p$ and $B_q$ are defined in \eqref{def-Ap-Bl}.
\end{lemma}

\begin{proof}
	Fix $u\in \nl$ and define $G_u\colon W \times \R \to \R$ by
	\begin{align*}
		G_u(z,t):=\langle I_{\la}'(t(u-z)), t(u-z)\rangle,
	\end{align*}
	that is,
	\begin{align*}
		G_u(z,t)
		& =t^p([u-z]_{s,p}^p+\mm{\na u-\na z}_p^p)+t^q \mm{\na u-\na z}_{q,a}^q \\
		&\quad -\la \int_{\Omega}\lt(w_1(x)t^k\ve{u-z}^k+w_2(x)t^r \ve{u-z}^r\ri)\dx.
	\end{align*}
	Since $u \in \nl$, it holds $\langle I'_{\la}(u),u\rangle=G_u(0,1)=0$.

	Next, we compute the partial derivative of $G_u$ with respect to $t$ at $(0,1)$. We have
	\begin{align*}
		\frac{\partial}{\partial t}G_u(0,1)=(p-k)([u]_{s,p}^p+\mm{\na u}_p^p)+(q-k)\mm{\na u}_{q,a}^q-\la (r-k)\int_{\Omega}w_2(x)\ve{u}^r\dx\neq0,
	\end{align*}
	where the inequality follows from Lemma \ref{empty}.

	Therefore, the assumptions of the Implicit Function Theorem are satisfied at $(z,t)=(0,1)$. Hence there exist $\e>0$ and a $C^1$-function $g\colon B_{\e}(0) \to \R$ such that $g(0)=1$ and $g(z)(u-z)\in \nl$ for all $z\in B_{\e}(0)$. Since both $u$ and $g(z)(u-z)$ belong to $\nl$, we obtain
	\begin{align}\label{4.14}
		[u]^p_{s,p}+\mm{\nabla u}_p^p+\mm{\na u}^q_{q,a}-\lambda\int_{\Omega}(w_1(x)|u|^{k}+w_2(x)|u|^{r})\, dx=0
	\end{align}
	and
	\begin{equation}\label{4.15}
		\begin{aligned}
			&g^p(z)[u-z]^p_{s,p}+g^p(z)\mm{\nabla u-\na z}_p^p+g^q(z)\mm{\na u-\na z}^q_{q,a}\\
			&-\lambda g^k(z)\int_{\Omega}w_1(x)\ve{u-z}^{k}\dx- \la g^r(z)\int_{\Omega}w_2(x)|u-z|^{r}\dx=0.
		\end{aligned}
	\end{equation}
	Moreover, the following limits hold as $z\to 0$ in $W$:
	\begin{align*}
		& \lim_{z \to 0}\frac{ g^p(z)\mm{\nabla u-\na z}_p^p-\mm{\na u}_p^p}{z}=p \langle g'(0),z \rangle\mm{\na u}_p^p-p\int_{\Omega}\ve{\na u}^{p-2}\na u\cdot\na z \dx, \\
		& \lim_{z \to 0}\frac{ g^q(z)\mm{\nabla u-\na z}_{q,a}^q-\mm{\na u}_{q,a}^q}{z}=q \langle g'(0),z \rangle\mm{\na u}_{q,a}^q-q\int_{\Omega}a(x)\ve{\na u}^{q-2}\na u\cdot\na z \dx, \\
		& \lim_{z \to 0} \frac{g^p(z)[u-z]^p_{s,p}-[u]^p_{s,p}}{z}=p\langle g'(0),z \rangle[u]_{s,p}^{p}-pA_p(u,z),\\
		& \lim_{z \to 0} \int_{\Omega}w_1(x)\frac{g^k(z)\ve{u-z}^k-\ve{u}^k}{z}\,dx=k\langle g'(0),z \rangle\int_{\Omega}w_1(x)\ve{u}^k\dx-k\int_{\Omega}w_1(x)\ve{u}^{k-2}uz\dx,\\
		& \lim_{z \to 0} \int_{\Omega}w_2(x)\frac{g^r(z)\ve{u-z}^r-\ve{u}^r}{z}\,dx=r\langle g'(0),z \rangle\int_{\Omega}w_2(x)\ve{u}^r\dx-r\int_{\Omega}w_2(x)\ve{u}^{r-2}uz\dx.
	\end{align*}
	Substituting these limits into the difference of \eqref{4.15} and \eqref{4.14} and simplifying yields \eqref{4.13}. This completes the proof.
\end{proof}

\begin{lemma}\label{ps_sequence}
	Let $\la \in(0,\la_*)$, where $\la_*$ is defined in \eqref{lambda1}. Then there exists a minimizing sequence $\{u_n\}_{n\in\mathbb{N}}\subset \nl$ such that
	\begin{align*}
		I_{\la}(u_n)=m_{\la}+o_n(1)\quad\text{and}\quad I'_{\la}(u_n)=o_n(1).
	\end{align*}
\end{lemma}

\begin{proof}
	From Lemma \ref{coer}, the functional $I_{\la}$ is coercive and bounded from below on $\nl$. Hence, by Ekeland's variational principle, there exists a sequence $\{u_n\}_{n\in\mathbb{N}}\subset \nl$ such that
	\begin{align}\label{4.17}
		I_{\la}(u_n)\leq m_{\la}+\frac{1}{n} \quad\text{and}\quad I_{\la}(u_n) \leq I_{\la}(v)+\frac{1}{n}\mm{v-u_n}_W \quad \text{for all }v \in \nl.
	\end{align}
	Since $m_{\la}^+<0$ by Lemma~\ref{inf}, it follows that for large $n$
	\begin{align*}
		I_{\la}(u_n)\leq m_{\la}+\frac{1}{n}\leq m_{\la}^++\frac{1}{n}<0,
	\end{align*}
	so $u_n \not\equiv 0$. Using H\"{o}lder's inequality, Lemma \ref{rel} and the embedding estimate $\|u_n\|_r\leq C_r \|u_n\|_W$, we obtain
	\begin{align*}
		0>I_{\la}(u_n)
		& =\lt(\frac{1}{p}-\frac{1}{r} \ri)([u_n]^p_{s,p}+\mm{\nabla u_n}_p^p)+\lt(\frac{1}{q}-\frac{1}{r} \ri)\mm{\nabla u_n}_{q,a}^q\\
		& \quad-\la \lt(\frac{1}{k}-\frac{1}{r} \ri) \int_{\Omega}w_1(x)\ve{u_n}^k\dx \\
		& \geq \lt(\frac{1}{q}-\frac{1}{r} \ri) \xi( u_n)-\la \lt(\frac{1}{k}-\frac{1}{r} \ri) \mm{w_1}_{\frac{r}{r-k}}\mm{u_n}_{r}^k \\
		& >\lt(\frac{1}{q}-\frac{1}{r} \ri) \min(\mm{u_n}_W^p,\mm{u_n}_W^{q})-\la \lt(\frac{1}{k}-\frac{1}{r} \ri)\mm{w_1}_{\frac{r}{r-k}}C_r^k\mm{u_n}_W^k.
	\end{align*}
	Hence,
	\begin{align*}
		\mm{u_n}_W \leq \lt(\la \frac{(r-k)q}{k(r-q)} \mm{w_1}_{\frac{r}{r-k}}C_r^k\ri)^{\frac{1}{\alpha-k}},
	\end{align*}
	where $\alpha=p$ if $\mm{u_n}_W\geq 1$ and $\alpha=q$ otherwise. Thus, $\{u_n\}_{n\in\mathbb{N}}$ is bounded in $W$.

	To prove $I'_{\la}(u_n) \to 0$ as $n \to \infty$, we use Lemma \ref{imp}. For each $u_n$, there exists $\e_n>0$ and a $C^1$-map $g_n\colon B_{\e_n}(0) \to (0,\infty)$ such that $g_n(z)(u_n-z)\in \nl$.

	Fix $n$ large enough so that $u_n \not\equiv 0$, and let $\kappa\in(0,\e_n)$. For a given $u \in W$, we define
	\begin{align*}
		\nu_{\kappa}:=\frac{\kappa u}{\mm{u}_{W}} \quad\text{and}\quad \eta_{\kappa}:=g_{n}(\nu_{\ka})(u_n-\nu_{\ka}).
	\end{align*}
	From \eqref{4.17}, we get
	\begin{align*}
		I_{\la}(\eta_{\ka})-I_{\la}(u_n)\geq -\frac{1}{n}\mm{\eta_{\ka}-u_n}_W.
	\end{align*}
	Using Taylor's expansion at $u_n$,
	\begin{align*}
		\langle I'_{\la}(u_n),\eta_{\ka}-u_n\rangle+o_n(\mm{\eta_{\ka}-u_n}_W)\geq -\frac{1}{n}\mm{\eta_{\ka}-u_n}_W.
	\end{align*}
	Substituting $\eta_{\ka}-u_n=-\nu_{\ka}+(g_n(\nu_{\ka})-1)(u_n-\nu_{\ka})$, we deduce
	\begin{align*}
		\left\langle I'_{\la}(u_n),\frac{u}{\mm{u}_W}\right\rangle\leq \frac{1}{n \ka} \mm{\eta_{\ka}-u_n}_W+\frac{(g_n(\nu_\ka)-1)}{\ka}\langle I'_{\la}(u_n)-I'_{\la}(\eta_{\ka}), u_n-\nu_{\ka}\rangle+\frac{1}{\ka}o_n(\mm{\eta_{\ka}-u_n}_W).
	\end{align*}
	Letting $\ka \to 0$ for fixed $n$, we deduce
	\begin{align}\label{4.21}
		\left\langle I'_{\la}(u_n),\frac{u}{\mm{u}_W}\right\rangle\leq  \frac{c}{n}(1+\mm{g'_n(0)}_{W'}),
	\end{align}
	for some $c>0$.

	It remains to show that $\mm{g'_n(0)}_{W'}$ is bounded. From Lemma \ref{imp} and H\"{o}lder's inequality, we have
	\begin{align*}
		\langle g_n'(0), z \rangle\leq\frac{C\mm{z}_{W}}{(p-k)([u_n]^p_{s,p}+\mm{\na u_n}_p^p)+(q-k)\mm{\na u_n}_{q,a}^{q}-\la(r-k)\int_{\Omega}w_2(x)\ve{u_n}^r\dx},
	\end{align*}
	where $C$ is a positive constant. Suppose  $\mm{g'_n(0)}_{W'}$ is not bounded. Then, along a subsequence, the denominator goes to $0$. However, since $u_n \in \nl$, it follows that $L_{\la}(u_n)=o_n(1)$ and
	\begin{align*}
		(p-k)\mm{\na u_n}_p^p
		& \leq (p-k)([u_n]^p_{s,p}+\mm{\na u_n}_p^p)+(q-k)\mm{\na u_n}_{q,a}^{q} \\
		& \leq \la_* (r-k) \mm{w_2}_{\infty} S_r^{r}\mm{\na u_n}_p^r+o_n(1),
	\end{align*}
	which leads to
	\begin{align*}
		\mm{\na u_n}_p \geq \lt(\frac{(p-k)}{(r-k)\la_*S_r^r}\ri)^{\frac{1}{r-p}}+o_n(1).
	\end{align*}
	Thus, for $n$ sufficiently large, we have $\mm{\na u_n}_p>c>0$. Then we get
	\begin{align*}
		L_{\la}(u_n)
		& \geq \frac{r-p}{r-k}\mm{\na u_n}_p^p-\la \mm{w_1}_{\frac{r}{r-k}}S_r^k \mm{\na u_n}_p^k  \\
		& > \mm{\na u_n}_p^k \lt(\frac{r-p}{r-k} \lt(\frac{p-k}{(r-k)\la_*S_r^r} \ri)^\frac{p-k}{r-p} -\la_* \mm{w_1}_{\frac{r}{r-k}}S_r^k\ri)
		=0,
	\end{align*}
	which is a contradiction to $L_{\la}(u_n)=o_n(1)$. Hence, $\mm{g'_n(0)}_{W'}$ is bounded, and from \eqref{4.21} we conclude that $I'_{\la}(u_n)=o_n(1)$.
\end{proof}

\begin{lemma}\label{almost_cvg}
	Let $\lambda > 0$ and $\{u_n\}_{n\in\mathbb{N}} \subset W$ be a bounded \textnormal{(PS)$_c$}-sequence for some $c \in \R$. Then, up to a subsequence, we have
	\begin{align*}
		\nabla u_n(x) \to \nabla u(x) \quad \text{a.e.\,in } \Omega  \text{ as } n \to \infty.
	\end{align*}
\end{lemma}

\begin{proof}
	The proof follows from Farkas--Fiscella--Winkert \cite[Lemma~3.1]{Farkas-Fiscella-Winkert-2022} and da Silva--Fiscella--Viloria \cite[Lemma~2.2]{daSilva-Fiscella-Viloria-2024}. For the sake of brevity, we omit the details.
\end{proof}

\begin{lemma}\label{l5.1}
	Suppose that $\{u_n\}_{n\in\mathbb{N}}\subset \nl$ is a \textnormal{(PS)$_c$}-sequence for the functional $I_{\lambda}$ at the level $c$, and that $u_{n}\rightharpoonup u$ weakly in $W$. Then $u$ is a critical point of $I_{\lambda}$. Moreover, the energy of the weak limit satisfies
	\begin{align*}
		I_{\lambda}(u)\geq -\tilde{c}\lambda^{\frac{p}{p-k}},
	\end{align*}
	where $\tilde{c}$ is defined in \eqref{tilde{c}}.
\end{lemma}

\begin{proof}
	Since $\{u_n\}_{n\in\mathbb{N}}$ is a \textnormal{(PS)$_c$}-sequence for $I_{\lambda}$, we have
	\begin{align*}
		\langle I'_{\la}(u_n),u_n\rangle&= [u_n]^p_{s,p}+\mm{\nabla u_n}_p^p+\mm{\na u_n}^q_{q,a}-\lambda\int_{\Omega}(w_1(x)|u_n|^{k}+w_2(x)|u_n|^{r})\dx =o_n(1),\\
		I_{\lambda}(u_n)&=\frac{1}{p}[u_n]^p_{s,p}+\frac{1}{p}\mm{\nabla u_n}_p^p+\frac{1}{q}\mm{\na u_n}^q_{q,a}-\frac{\lambda}{k}\int_{\Omega}w_1(x)|u_n|^{k}\dx-\frac{\lambda}{r}\int_{\Omega}w_2(x)|u_n|^{r}\dx\\
		& =c+o_n(1).
	\end{align*}
	Using Lemma~\ref{rel} and the Sobolev embedding theorem, we estimate
	\begin{align*}
		c_1+c_2\|u_n\|_{W} & \geq I_{\lambda}(u_n)-\frac{1}{r} \langle I'_{\la}(u_n),u_n\rangle   \\& = \left(\frac{1}{p}-\frac{1}{r}\right)[u_n]^p_{s,p}+\left(\frac{1}{p}-\frac{1}{r}\right)\mm{\nabla u_n}_p^p+\left(\frac{1}{q}-\frac{1}{r}\right)\mm{\na u_n}^q_{q,a}     \\
		& \quad-\lambda\left(\frac{1}{k}- \frac{1}{r}\right)\int_{\Omega}w_1(x)|u_n|^{k}\dx   \\
		& \geq\left(\frac{1}{q}-\frac{1}{r}\right) \xi(u_n)-\lambda \left(\frac{1}{k}-\frac{1}{r}\right)\mm{w_1}_{\frac{r}{r-k}}\mm{u_n}_{r}^k \\
		& \geq\left(\frac{1}{q}-\frac{1}{r}\right) \min\{\|u_n\|_{W}^{p}, \|u_n\|_{W}^{q}\}-\lambda \left(\frac{1}{k}-\frac{1}{r}\right)\mm{w_1}_{\frac{r}{r-k}}C_r^k\mm{u_n}_W^k,
	\end{align*}
	for some $c_1,c_2>0$. Hence,
	\begin{align*}
		c_3\geq \left(\frac{1}{q}-\frac{1}{r}\right) \min\{\|u_n\|_{W}^{p}, \|u_n\|_{W}^{q}\}-\lambda c_4\left(\frac{1}{k}-\frac{1}{r}\right)\|u_n\|^{k}_{W},
	\end{align*}
	for some $c_3,c_4>0$. Since $k<p<q<r$, the last inequality implies that $\{u_n\}_{n\in\mathbb{N}}$ is bounded in $W$.

	From this and the reflexivity of $W$, there exists $u\in W$ such that $u_{n}\rightharpoonup u$ weakly in $W$, up to a subsequence not relabeled. Therefore, we have the compactness properties
	\begin{align}\label{convergence}
		\begin{cases}
			u_n\to u, & \text{strongly in } L^\delta(\Omega) \text{ for }  \delta \in [1, p^*), \\
			u_n  \to u,  & \text{a.e.\,in } \Omega, \\
			|u_n|^{\delta-2}u_n \rightharpoonup |u|^{\delta-2}u, & \text{weakly in } L^{\delta'}(\Omega)  \text{ for } \delta \in (1, p^*] \\
			|\nabla u_n|^{\delta-2}  \nabla u_n & \text{is bounded in }  L^{\delta^{'}}(\Omega)  \text{ for }  \delta \in \{p,q\}.
		\end{cases}
	\end{align}
	By Lemma 2.2 of Chen--Mosconi--Squassina \cite{Chen-Mosconi-Squassina-2018} and \eqref{convergence}, we have
	\begin{equation}\label{cp1}
		\begin{aligned}
			\lim_{n\to\infty}A_p(u_n,v) &=A_p(u,v), \quad \text{for all }v\in W,\\
			\lim_{n\to\infty}B_p(1,u_n,v)&=B_p(1,u,v) , \quad \text{for all }v\in W.
		\end{aligned}
	\end{equation}
	Let
	\begin{align}\label{def-K}
		K := \{x \in \Omega \colon a(x)=0\}.
	\end{align}
	Since $a(\cdot)$ is Lipschitz continuous, the set $\Omega\setminus K$ is an open subset of $\R^N$. Moreover, $ \{ \ve{\na u_n}^{q-2}\na u_n\}_{n\in\mathbb{N}}$ is bounded in  $L^{q'}(\Om\setminus K, a(x))$. Using \eqref{convergence}, Lemma \ref{almost_cvg} and Proposition A.8 of Autuori--Pucci \cite{Autuori-Pucci-2013}, we obtain
	\begin{align*}
		\lim_{n \to \infty} B_q(a(x),u_n,u)=\lim_{n \to \infty}\int_{\Omega\setminus K} a(x)\ve{\na u_n}^{q-2}\na u_n\cdot \na u \dx
		= B_q(a(x),u,u).
	\end{align*}
	Moreover, from \eqref{convergence}, we deduce
	\begin{equation}\label{cp5}
		\begin{aligned}
			\int_{\Omega} w_1(x)|u_n|^{k-2}u_n v \dx&\to \int_{\Omega} w_1(x)|u|^{k-2}u v\dx, \quad \text{for all } v\in W,\\
			\int_{\Omega}w_2(x)|u_n|^{r-2}u_n v\dx &\to \int_{\Omega} w_2(x)|u|^{r-2}u v\dx, \quad \text{for all } v\in W.
		\end{aligned}
	\end{equation}
	Combining \eqref{cp1} and \eqref{cp5} gives
	\begin{align*}
		&\langle I'_{\la}(u_n),v\rangle -\langle I'_{\la}(u),v\rangle\\
		& =(A_p(u_n,v) -A_p(u,v))+ (B_p(1,u_n,v)-B_p(1,u,v)) \\
		& \quad+(B_q(a(x),u_n,v)-B_q(a(x),u,v))-\lambda \int_{\Omega} w_1(x)(|u_n|^{k-2}u_n -|u|^{k-2}u)v\dx \\
		& \quad -\lambda \int_{\Omega} w_2(x)(|u_n|^{r-2}u_n -|u|^{r-2}u)v\dx=o_n(1).
	\end{align*}
	Since $I'_{\la}(u_n)\to 0$ in $W'$, it follows that $\langle I'_{\la}(u),v\rangle=0$ for all $v\in W$. Hence, $u$ is a critical point of $I_{\la}$.

	Finally, since $u$ is a critical point of $I_{\la}$, we may write
	\begin{align*}
		I_{\lambda}(u)\geq\left(\frac{1}{p}-\frac{1}{r}\right)\mm{\nabla u}_p^p-\lambda\left(\frac{1}{k}-\frac{1}{r}\right)\int_{\Omega}w_1(x)|u|^{k}\dx.
	\end{align*}
	Using H\"{o}lder’s inequality, the Sobolev embedding theorem, and Young’s inequality, we obtain the estimate
	\begin{align*}
		&\lambda \int_{\Omega} w_1(x)|u|^k\dx\\
		&\leq \left(\frac{p}{k}\left(\frac{1}{p}-\frac{1}{r}\right)\left(\frac{1}{k}-\frac{1}{r}\right)^{-1}\right)^{\frac{k}{p}}\mm{\nabla u}_p^k  \lambda\left(\frac{p}{k}\left(\frac{1}{p}-\frac{1}{r}\right)\left(\frac{1}{k}-\frac{1}{r}\right)^{-1}\right)^{\frac{-k}{p}}\|w_1\|_{r/(r-k)} S_r^{k} \\
		& \leq\left(\left(\frac{1}{p}-\frac{1}{r}\right)\left(\frac{1}{k}-\frac{1}{r}\right)^{-1}\right)\mm{\nabla u}_p^p+A \lambda^{\frac{p}{p-k}},
	\end{align*}
	where
	\begin{align*}
		A=\left[\left(\frac{p}{k}\left(\frac{1}{p}-\frac{1}{r}\right)\left(\frac{1}{k}-\frac{1}{r}\right)^{-1}\right)^{\frac{-k}{p}}\|w_1\|_{r/r-k} S_r^{k}\right]^{\frac{p}{p-k}}.
	\end{align*}
	Thus, we deduce
	\begin{align*}
		I_\lambda(u) \geq-\left(\frac{1}{k}-\frac{1}{r}\right) A \lambda^{\frac{p}{p-k}}=-\tilde{c}\lambda^{\frac{p}{p-k}},
	\end{align*}
	which completes the proof.
\end{proof}

\begin{lemma}\label{ps_condition}
	Suppose that $\{u_n\}_{n\in\mathbb{N}}\subset \nl$ is a \textnormal{(PS)$_c$}-sequence for the functional $I_{\lambda}$ at the level $c$. Then $\{u_n\}_{n\in\mathbb{N}}$ has a convergent subsequence in $W$ provided that
	\begin{align}\label{level}
		-\infty<c<c_{0}:=\left(\frac{1}{q}-\frac{1}{p^{*}}\right)\frac{\So^{\frac{N}{p}}}{(\lambda \|w_2\|_{\infty})^{\frac{N-p}{p}}}-\tilde{c}\lambda^{\frac{p}{p-k}},
	\end{align}
	where $\tilde{c}$ is defined in \eqref{tilde{c}}.
\end{lemma}

\begin{proof}
	Since  $\{u_n\}_{n\in\mathbb{N}}$ is a \textnormal{(PS)$_c$}-sequence, Lemma \ref{l5.1} implies that $\{u_n\}_{n\in\mathbb{N}}$ is bounded in $W$. By reflexivity, there exists $u\in W$ such that, up to a subsequence, $u_{n}\rightharpoonup u$ weakly in $W$. By Lemma \ref{l5.1}, the weak limit $u$ is a critical point of $I_{\lambda}$. Moreover, since $u_n\to u$ strongly in $L^{\delta}(\Omega)$ for every $\delta\in[1,p^{*})$, we have
	\begin{align*}
		\int_{\Omega}w_1(x)|u_n|^{k}\dx\to \int_{\Omega}w_1(x)|u|^{k}\dx.
	\end{align*}
	Applying the Brezis-Lieb lemma, we obtain the decompositions
	\begin{align}
		\frac{1}{p}[u_n-u]^p_{s,p}+\frac{1}{p}\mm{\nabla u_n-\nabla u}_p^p+\frac{1}{q}\mm{\nabla u_n- \nabla u}^q_{q,a}-\frac{\lambda}{r}\|u_n-u\|_{r,w_2}^{r}+I_{\lambda}(u)&=c+o_n(1),\label{5.5}\\
		[u_n-u]^p_{s,p}+\mm{\nabla u_n-\nabla u}_p^p+\mm{\nabla u_n-\nabla u}^q_{q,a}-\lambda\|u_n-u\|_{r,w_2}^{r}&=o_n(1).\label{5.6}
	\end{align}
	Define, \begin{align*}
		\kappa=\lim_{n\to\infty}([u_n-u]^p_{s,p}+\mm{\nabla u_n-\nabla u}_p^p+\mm{\nabla u_n-\nabla u}^q_{q,a})\geq 0.
	\end{align*}
	From \eqref{5.6}, it follows that
	\begin{align*}
		\lim_{n\to\infty}\lambda\|u_n-u\|_{r,w_2}^{r}= \kappa.
	\end{align*}
	If $\kappa=0$, then
	\begin{align*}
		\lim_{n\to\infty}([u_n-u]^p_{s,p}+\mm{\nabla u_n-\nabla u}_p^p+\mm{\nabla u_n-\nabla u}^q_{q,a})=0
	\end{align*}
	and Lemma \ref{rel} yields $u_n\to u$ strongly in $W$, completing the proof.

	If $r<p^*$, then $u_n\to u$ strongly in $L^{r}(\Omega)$ for all $r\in[1,p^{*})$. Hence,
	\begin{align*}
		\lim_{n\to\infty}\lambda\|u_n-u\|_{r,w_2}^{r}= \kappa=0.
	\end{align*}
	It remains to consider the case $r=p^{*}$. Using \eqref{a2}, we obtain
	\begin{align*}
		\kappa=\lim_{n\to\infty}\lambda\|u_n-u\|_{p^{*},w_2}^{p^{*}}\leq \lambda \|w_2\|_{\infty} \So^{-\frac{p^{*}}{p}}\lim_{n\to\infty}\|\nabla u_n-\nabla u\|_{p}^{p^{*}}\leq \lambda \|w_2\|_{\infty} \So^{-\frac{p^{*}}{p}}\kappa^{\frac{p^{*}}{p}}.
	\end{align*}
	This yields
	\begin{align*}
		\frac{\So^{\frac{N}{p}}}{(\lambda \|w_2\|_{\infty})^{\frac{N-p}{p}}}<\kappa.
	\end{align*}
	Now \eqref{5.5} gives
	\begin{align*}
		c-I_{\lambda}(u)
		& =\lim_{n\to\infty}\left[\frac{1}{p} [u_n-u]^p_{s,p}+\frac{1}{p}\mm{\nabla u_n-\nabla u}_p^p+\frac{1}{q}\mm{\nabla u_n-\nabla u}^q_{q,a}-\frac{\lambda}{p^{*}}\|u_n-u\|_{p^{*},w_2}^{p^{*}}\right] \\
		& \geq \left(\frac{1}{q}-\frac{1}{p^{*}}\right)k \\
		& > \left(\frac{1}{q}-\frac{1}{p^{*}}\right)\frac{\So^{\frac{N}{p}}}{(\lambda \|w_2\|_{\infty})^{\frac{N-p}{p}}}.
	\end{align*}
	Hence,
	\begin{align*}
		c >\left(\frac{1}{q}-\frac{1}{p^{*}}\right)\frac{\So^{\frac{N}{p}}}{(\lambda \|w_2\|_{\infty})^{\frac{N-p}{p}}}+I_{\lambda}(u)                    \geq \left(\frac{1}{q}-\frac{1}{p^{*}}\right)\frac{\So^{\frac{N}{p}}}{(\lambda \|w_2\|_{\infty})^{\frac{N-p}{p}}}-\tilde{c}\lambda^{\frac{p}{p-k}},
	\end{align*}
	contradicting \eqref{level}.  Therefore,  $\kappa=0$, and hence $u_n\to u$ strongly in $W$.
\end{proof}

\begin{proof}[Proof of Theorem \ref{t1}]
	Let $\lambda_{**}>0$ be chosen so that, for every $\lambda\in(0,\lambda_{**})$, the inequality
	\begin{align*}
		\left(\frac{1}{q}-\frac{1}{p^{*}}\right)\frac{\So^{\frac{N}{p}}}{(\lambda \|w_2\|_{\infty})^{\frac{N-p}{p}}}-\tilde{c}\lambda^{\frac{p}{p-k}}>0.
	\end{align*}
	holds. Define $\lambda_{0}=\min\{\lambda_{*},\lambda_{**}\}$.
	For all $\lambda \in (0, \lambda_0)$, by Lemma \ref{ps_sequence}, there exists a minimizing sequence $\{u_n\}_{n\in\mathbb{N}}\subset\nl$. Moreover, $\{u_n\}_{n\in\mathbb{N}}$ is a \textnormal{(PS)$_{m_\lambda}$}-sequence for $I_\lambda$ at the level $m_\lambda$. By Lemma \ref{ps_condition}, there exists $u_\lambda \in W$ such that, up to a subsequence,
	\begin{align*}
		u_n \to  u_\lambda \quad \text{strongly in } W.
	\end{align*}
	Consequently, for these values of $\lambda$, the function $u_\lambda$ satisfies $\langle I'_{\la}(u_\lambda),v\rangle=0$ for all $v\in W$, that is, $u_\lambda$ is a critical point of $I_\la$. In particular,
	\begin{align*}
		I_{\lambda}(u_\lambda)
		& \geq\left(\frac{1}{p}-\frac{1}{r}\right)\mm{\nabla u_\lambda}_p^p-\lambda\left(\frac{1}{k}-\frac{1}{r}\right)\int_{\Omega}w_1(x)|u_\lambda|^{k}\dx \\
		& \geq-\lambda\left(\frac{1}{k}- \frac{1}{r}\right)\int_{\Omega}w_1(x)|u_\lambda|^{k}\dx.
	\end{align*}
	Therefore, since $I_\la(u_\lambda)=m_\lambda$, it follows that
	\begin{align*}
		\int_{\Omega}w_1(x)|u_\lambda|^{k}\dx \geq -\frac{m_\lambda}{\lambda} \left( \frac{1}{k} - \frac{1}{r} \right)^{-1} > 0,
	\end{align*}
	implying $u_\lambda \not\equiv 0$. Hence, $u_\lambda \in N_\lambda$ and $I_\lambda(u_\lambda) = m_\lambda$.

	Next, we prove that $u_\lambda \in N_\lambda^{+}$. Suppose, by contradiction, that $u_\lambda \in N_\lambda^{-}$. Then, by Lemma \ref{geometry}, there exist numbers $t_1 < t_2 = 1$ such that $t_1 u_\lambda \in N_\lambda^{+}$ and $t_2 u_\lambda \in N_\lambda^{-}$. Since the map $t \mapsto I_\lambda(t u_\lambda)$ is increasing on $[t_1, t_2)$, we obtain
	\begin{align*}
		m_\lambda\leq I_\lambda(t_1 u_\lambda)< I_\lambda(t u_\lambda)\leq I_\lambda(u_\lambda) = m_\lambda, \quad \text{for all }t \in (t_1, 1),
	\end{align*}
	which is a contradiction. Hence, $u_\lambda \in N_\lambda^{+}$, and therefore $m_\lambda = I_\lambda(u_\lambda) = m_\lambda^{+}$. If $u_\lambda \ge 0$, then $u_\lambda $ itself is a nonnegative weak solution of \eqref{1.1} and minimizes $ I_\lambda $ on $ N_\lambda^{+} $. Otherwise, assume that $ u_\lambda $ changes sign. By Lemma \ref{geometry}, there exists a unique $ t_1 > 0 $ such that $ t_1 |u_\lambda| \in N_\lambda^{+} $. Moreover,
	\begin{align*}
		\ga_{|u_\lambda|}(1) \leq \ga_{u_\lambda}(1)=\lambda \int_{\Omega} w_1(x)\, |u_\lambda|^k \dx= \ga_{|u_\lambda|}(t_1)\leq  \ga_{u_\lambda}(t_1).
	\end{align*}
	Since $ u_\lambda \in N_\lambda^{+} $, we have $ \ga_{u_\lambda}'(1) > 0 $, which implies $ t_1 \ge 1 $. Consequently,
	\begin{align*}
		m_\lambda^{+} \leq  \sigma_{|u_\lambda|}(t_1)\leq  \sigma_{u_\lambda}(t_1) \leq  \sigma_{u_\lambda}(1)= m_\lambda^{+}.
	\end{align*}
	Therefore equality holds throughout, and we obtain
	\begin{align*}
		I_\lambda\bigl( t_1\, |u_\lambda| \bigr) = \sigma_{|u_\lambda|}(t_1) = m_\lambda^{+},
	\end{align*}
	with $ t_1  |u_\lambda| \in N_\lambda^{+} $. Hence, $ t_1 |u_\lambda| $ is a nonnegative weak solution of \eqref{1.1} lying in $ N_\lambda^{+} $.
\end{proof}

In the second part of this section, we prove Theorems \ref{t2} and \ref{t3}. Recall that  $w_2= \frac{1}{\lambda}$ and $r=p^*$ while still \eqref{H1} \eqref{H1i} and \eqref{H1ii} are assumed. The energy functional $\I_{\la}\colon W\to\R$  associated with problem \eqref{1.2} is defined by
\begin{align*}
	\I_{\la}(u)=
	& \frac{1}{p}\int_{\R^N}\int_{\R^N} \dfrac{|u(x)-u(y)|^p}{|x-y|^{N+sp}}\dxy+\int_{\Omega}\lt(\frac{|\nabla u|^p}{p}+a(x)\frac{|\nabla u|^q}{q}\ri)\dx \\
	& -\lambda\int_{\Omega}\frac{w_1(x)}{k}|u|^{k}\dx -\int_{\Omega} \frac{1}{p^*}|u|^{p^*}\dx.
\end{align*}
As in Section \ref{Section_3}, we define the Nehari manifold by
\begin{align*}
	\nla:=\left\{u\in W\setminus \{0\}\colon \langle \I'_{\la}(u),u \rangle=0\right\}.
\end{align*}
For each $u \in W\setminus\{0\}$, we define the fibering map $\varsigma_u(t):=\I_{\la}(tu)$. Furthermore, we decompose the set $\nla$ into three disjoint subsets, namely $\nla^{\pm}$ and $\nla^0$ in the same way as in Section \ref{Section_3}. We also define the quantities $\m$ and $\m^{+}$ analogously.

Note that $\I_{\la}$ is coercive and bounded from below on $\nla$ by Lemma \ref{coer}. Moreover, arguing as in Lemma \ref{empty}, one has $\nla^0=\emptyset$
for $\la \in (0,\overline{\la}_*)$, where $\overline{\la}_*$ is defined in \eqref{lambda2}. Proceeding as in Lemma \ref{geometry}, one can show that for $\la \in (0,\overline{\la}_*)$ and every $u\in W\setminus\{0\}$, there exist unique $t_m>0$ and numbers $t_1,t_2$ with $t_1<t_m<t_2$ such that $\varsigma''_u(t_m)=0$, $t_1u\in \nla^+$, and $t_2 u\in \nla^-$. Repeating the argument of Lemma \ref{inf}, we deduce $\m^+<0$.

Next, proceeding as in Lemmas \ref{imp}, \ref{ps_sequence} and \ref{ps_condition}, we obtain the following auxiliary results.

\begin{lemma}
	Let $\la \in (0,\overline{\la}_*)$, where $\overline{\la}_*$ is defined in \eqref{lambda2}, and let $u \in \nla$. Then there exist $\e>0$ and a differentiable function $\g\colon  B_{\e}(0) \to (-\infty, \infty)$ such that $\g(0)=1$ and $\g(z)(u-z)\in \nla$. Moreover,
	\begin{align*}
		\langle \g'(0), z \rangle=\frac{p B_p(1,u,z)+q B_q(a(x),u,z) +p A_p(u,z)-\int_{\Omega}(\la k w_1(x)\ve{u}^{k-2}+p^*\ve{u}^{p^*-2})uz\dx}{(p-k)([u]^p_{s,p}+\mm{\na u}_p^p)+(q-k)\mm{\na u}_{q,a}^{q}-(p^*-k)\int_{\Omega}\ve{u}^{p^*}\dx}.
	\end{align*}
\end{lemma}

\begin{lemma}\label{psc}
	For $\la \in (0,\overline{\la}_*)$, where $\overline{\la}_*$ is defined in \eqref{lambda2}, there exists a minimizing sequence $\{u_n\}_{n\in\mathbb{N}}\subset \nla$ such that
	\begin{align*}
		\I_{\la}(u_n)=\m+o_n(1)\quad\text{and}\quad \I'_{\la}(u_n)=o_n(1).
	\end{align*}
\end{lemma}

\begin{lemma}\label{ps_condition1}
	Suppose that $\{u_n\}_{n\in\mathbb{N}}\subset \nla$ is a \textnormal{(PS)$_c$}-sequence of $\I_{\lambda}$ at the level $c$. Then $\{u_n\}_{n\in\mathbb{N}}$ has a convergent subsequence in $W$ provided that
	\begin{align*}
		-\infty<c<c_{0}:=\left(\frac{1}{q}-\frac{1}{p^{*}}\right)\So^{\frac{N}{p}}-\tilde{c}\lambda^{\frac{p}{p-k}},
	\end{align*}
	where $\tilde{c}$ is defined in \eqref{tilde{c}}.
\end{lemma}

\begin{proof}[Proof of Theorem \ref{t2}]
	Let $\overline{\la}_{**}>0$ be chosen so that, for all $\lambda\in(0,\overline{\la}_{**})$, the inequality
	\begin{align*}
		\left(\frac{1}{q}-\frac{1}{p^{*}}\right)\So^{\frac{N}{p}}-\tilde{c}\lambda^{\frac{p}{p-k}}>0
	\end{align*}
	is satisfied. We define $\lambda_{1}=\min\{\overline{\la}_{*} ,\overline{\la}_{**}\}$. For all $\lambda \in (0, \lambda_1)$, by Lemma \ref{psc}, there exists a minimizing sequence $\{u_n\}_{n\in\mathbb{N}}\subset \nla$, which is also a Palais-Smale sequence for $\I_\lambda$ at the level $\m$. Applying Lemma \ref{ps_condition1}, we find a function $u_\lambda \in W$ such that, up to a subsequence,
	\begin{align*}
		u_n \to  u_\lambda \quad \text{strongly in } W.
	\end{align*}
	Proceeding exactly as in the proof of Theorem \ref{t1} yields the conclusion.
\end{proof}

Next, we additionally assume that $w_1\equiv 1$. Consider $\tilde{\la}>0$ such that
\begin{align}\label{tiltla}
	\tilde{\la}< \frac{\So^{\frac{p^*-k}{p^*-p}}}{\ve{\Omega}^{\frac{p^*-k}{p^*}}}\lt(\frac{1}{q}-\frac{1}{p^*}\ri)\lt(\frac{1}{k}-\frac{1}{q}\ri)^{-1}.
\end{align}

\begin{lemma}\label{PSc}
	Let $\la\in(0,\tilde{\la})$. Then $\I_{\la}$ satisfies the \textnormal{(PS)$_c$}-condition at level $c$ for every $c<0$.
\end{lemma}

\begin{proof}
	Let $\{u_n\}_{n\in\mathbb{N}}$ be a \textnormal{(PS)$_c$}-sequence for $\I_{\la}$. Proceeding as in Lemma \ref{l5.1}, we see that $\{u_n\}_{n\in\mathbb{N}}$ is bounded in $W$. Since $W$ is reflexive, there exists $u\in W$ and a subsequence, not relabeled, such that $u_{n}\rightharpoonup u$ weakly in $W$. Consequently, up to a subsequence,
	\begin{align}\label{con}
		\begin{cases}
			u_n \to u,  & \text{strongly in } L^\delta(\Omega)  \text{ for } \delta \in [1, p^*),\\
			u_n \to u, & \text{a.e.\,in } \Omega, \\
			|u_n|^{\delta-2}u_n \rightharpoonup |u|^{\delta-2}u, & \text{weakly in } L^{\delta'}(\Omega)  \text{ for }  \delta \in (1, p^*], \\
			\na u_n(x)  \to \na u(x)  & \text{a.e.\,in } \Omega, \\
			\mm{u_n-u}_{p^*}     \to l,  & \text{where }  l\geq 0.
		\end{cases}
	\end{align}
	Let $K := \{x \in \Omega \colon a(x)=0\}$ as in \eqref{def-K}. Since $a(\cdot)$ is Lipschitz, $\Omega\setminus K$ is open. From \eqref{con}, we get
	\begin{align*}
		\lim_{n \to \infty}\int_{\Omega}\ve{\na u_n}^{p-2}\na u_n\cdot \na u\dx=\mm{\na u}_p^p.
	\end{align*}
	Further, $ \{ \ve{\na u_n}^{q-2}\na u_n\}_{n\in\mathbb{N}}$ is bounded in $L^{q'}(\Om\setminus K, a(x))$. Using \eqref{con} and Proposition A.8 of Autuori--Pucci \cite{Autuori-Pucci-2013}, we arrive at
	\begin{align*}
		\lim_{n \to \infty} B_q(a(x),u_n,u)=\lim_{n \to \infty}\int_{\Omega\setminus K} a(x)\ve{\na u_n}^{q-2}\na u_n\cdot \na u \dx
		= \mm{\na u}_{q,a}^q.
	\end{align*}
	In addition, by \eqref{con} and Lemma 2.2 of Chen--Mosconi--Squassina \cite{Chen-Mosconi-Squassina-2018},
	\begin{align}\label{cv3}
		\lim_{n\to\infty}A_p(u_n,u)
		=A_p(u,u)=[u]_{s,p}^p.
	\end{align}
	Therefore, using \eqref{con}, \eqref{cv3}, and the Brezis-Lieb Lemma, we compute
	\begin{align*}
		o_n(1)
		& =\langle \I'_{\la}(u_n), u_n-u \rangle\\
		& =B_p(1, u_n,u_n-u)+B_q(a(x),u_n,u_n-u)+A_p(u_n,u_n-u)-\la \int_{\Omega} \ve{u_n}^{k-2}u_n(u_n-u)\dx\\
		& \quad -\int_{\Omega} \ve{u_n}^{p^*-2}u_n(u_n-u)\dx\\
		& = \mm{\na u_n}_p^p-\mm{\na u}_p^p+\mm{\na  u_n}_{q,a}^q-\mm{\na  u}_{q,a}^q+[u_n]_{s,p}^p-[u]_{s,p}^p-\mm{u_n}_{p^*}^{p^*}+\mm{u}_{p^*}^{p^*}+o_n(1) \\
		& =\mm{\na u_n-\na u}_{p}^p+\mm{\na u_n-\na u}_{q,a}^q+[u_n-u]_{s,p}^p-\mm{u_n-u}_{p^*}^{p^*}+o_n(1).
	\end{align*}
	Hence,
	\begin{align}\label{cv4}
		\lim_{n \to \infty}\lt(\mm{\na u_n-\na u}_{p}^p+\mm{\na u_n-\na u}_{q,a}^q+[u_n-u]_{s,p}^p \ri)=\lim_{n \to \infty} \mm{u_n-u}_{p^*}^{p^*}=l^{p^*}.
	\end{align}
	If $l=0$, then Lemma \ref{rel} yields $u_n \to u$ strongly in $W$, and we are done. Suppose, for contradiction, that $l>0$. From \eqref{cv4} and \eqref{a2}, we obtain
	\begin{align*}
		\So l^p \leq \mm{\na u_n-\na u}_p^{p}\leq l^{p^*},
	\end{align*}
	and therefore
	\begin{align}\label{cv6}
		l \geq \So^{\frac{1}{p^*-p}}.
	\end{align}
	Next, observe that
	\begin{align*}
		c+o_n(1) & =\I_{\la}(u_n)-\frac{1}{q}\langle \I'_{\la}(u_n),u_n\rangle \\
		& \geq \lt(\frac{1}{p}-\frac{1}{q}\ri)(\mm{\na u_n}_p^p+[u_n]_{s,p}^p)-\la\lt(\frac{1}{k}-\frac{1}{q}\ri) \mm{u_n}_k^k+\lt(\frac{1}{q}-\frac{1}{p^*}\ri)\mm{u_n}_{p^*}^{p^*}.
	\end{align*}
	Letting $n \to \infty$ and using \eqref{con}, the Brezis-Lieb Lemma, H\"{o}lder's inequality, Young's inequality and \eqref{cv6}, we deduce
	\begin{align*}
		c & \geq  \lt(\frac{1}{q}-\frac{1}{p^*}\ri)(l^{p^*}+\mm{u}_{p^*}^{p^*})- \la\lt(\frac{1}{k}-\frac{1}{q}\ri) \mm{u}_k^k  \\
		& \geq \lt(\frac{1}{q}-\frac{1}{p^*}\ri)(l^{p^*}+\mm{u}_{p^*}^{p^*})- \la\lt(\frac{1}{k}-\frac{1}{q}\ri)  \ve{\Omega}^{\frac{p^*-k}{p^*}}\mm{u}_{p^*}^k \\
		& \geq  \lt(\frac{1}{q}-\frac{1}{p^*}\ri)(l^{p^*}+\mm{u}_{p^*}^{p^*})-\lt(\frac{1}{q}-\frac{1}{p^*}\ri)\mm{u}_{p^*}^{p^*}-\ve{\Omega} \lt(\frac{1}{q}-\frac{1}{p^*}\ri)^{-\frac{k}{p^*-k}}\lt(\la  \lt(\frac{1}{k}-\frac{1}{q}\ri)\ri)^{\frac{p^*}{p^*-k}} \\
		& \geq \lt(\frac{1}{q}-\frac{1}{p^*}\ri)\So^{\frac{p^*}{p^*-p}}-\ve{\Omega} \lt(\frac{1}{q}-\frac{1}{p^*}\ri)^{-\frac{k}{p^*-k}}\lt(\la \lt(\frac{1}{k}-\frac{1}{q}\ri)\ri)^{\frac{p^*}{p^*-k}}.
	\end{align*}
	For every $\la \in (0,\tilde{\la})$ with $\tilde{\la}$ satisfying \eqref{tiltla}, the right-hand side is strictly positive, hence $c>0$, contradicting the assumption $c<0$. Hence, $l=0$, and thus $u_n\to u$ strongly in $W$. This completes the proof.
\end{proof}

Since the energy functional $\I_{\la}$ involves the critical Sobolev exponent, it is not bounded from below on $W$. To address this, following the approach of Farkas--Fiscella--Winkert \cite{Farkas-Fiscella-Winkert-2022} and Garc\'{\i}a Azorero--Peral Alonso \cite{GarciaAzorero-PeralAlonso-1991}, we introduce the auxiliary function $g_{\la}\colon [0,\infty)\to \R$ defined by
\begin{align*}
	g_{\la}(t)=\frac{1}{q}t^q-\frac{\la}{k}C^k_k t^k-\frac{1}{p^*}C^{p^*}_{p^*}t^{p^*}.
\end{align*}
Since $k<q<p^*$, we have $g_{\la}(t)<0$ for $t>0$ sufficiently small and $g_{\la}(t)\to -\infty$ as $t\to \infty$. Moreover, because $1<k<p<q<p^*$, there exists $\tilde{\la}_1>0$ such that $g_{\la}$ attains its positive maximum for every $\la \in (0,\tilde{\la}_1)$. Let $\al_0(\la)$ and $\al_1(\la)$ be the two positive zeros of $g_{\la}$ such that $0<\al_0<\al_1$.

We now show that $\al_0(\la) \to 0$ as $\la \to 0$. Since $g_{\la}(\al_0)=0$ and $g'_{\la}(\al_0)>0$, we have, for $\la \in (0,\tilde{\la}_1)$,
\begin{align}
	\frac{1}{q}\al_0^q(\la)&=\frac{\la}{k}C^k_k \al_0^k(\la)+\frac{1}{p^*}C^{p^*}_{p^*}\al_0^{p^*}(\la),\label{6.9}\\
	\al_0^{q-1}(\la)&>\la C^k_k \al_0^{k-1}(\la)+C^{p^*}_{p^*}\al_0^{p^*-1}(\la).\label{6.10}
\end{align}
From \eqref{6.9} it follows that $\al_0(\la)$ is bounded. Suppose, for contradiction, that $\al_0(\la)\to \al\neq 0$ as $\la \to 0$. Passing to the limit in \eqref{6.9} and \eqref{6.10} yields
\begin{align*}
	\frac{1}{q}\al^q=\frac{1}{p^*}C^{p^*}_{p^*}\al^{p^*}  \quad \text{and} \quad \al^{q-1}>C^{p^*}_{p^*}\al^{p^*-1},
\end{align*}
which is impossible because $q<p^*$. Therefore, $\al_0(\la)\to 0$ as $\la \to 0$. Hence, there exists $\tilde{\la}_2$ such that $\al_0(\la)<1$ for every $\la \in (0,\tilde{\la}_2)$, and consequently $\al_0(\la)<\min\{\al_1(\la),1\}$.

Next, for $\la \in (0,\min\{\tilde{\la}_1,\tilde{\la}_2\}),$ we fix a $C^{\infty}$-function $\theta\colon [0,\infty) \to [0,1]$ such that
\begin{align*}
	\theta(t):=
	\begin{cases}
		1 &\text{if } t \in [0,\al_0(\la)] \\
		0 &\text{if } t \in [\min\{\al_1(\la),1\},\infty).
	\end{cases}
\end{align*}
We then define the truncated energy functional $\tilde{\I}_{\la}\colon W \to \R$ by
\begin{align*}
	\tilde{\I}_{\la}(u)
	& =\frac{1}{p}\int_{\R^N}\int_{\R^N} \dfrac{|u(x)-u(y)|^p}{|x-y|^{N+sp}}\dxy+\int_{\Omega}\lt(\frac{|\nabla u|^p}{p}+a(x)\frac{|\nabla u|^q}{q}\ri)\dx\\
	& \quad-\frac{\la}{k} \int_{\Omega}|u|^k\dx  -\frac{1}{p^*}\mm{u}_{p^*}^{p^*}\theta(\mm{u}_W).
\end{align*}
It is immediate that
\begin{align*}
	\tilde{\I}_{\la}(u)=\I_{\la}(u)\quad \text{for any } \|u\|_W\in [0,\al_0(\la)]
\end{align*}
and that $\tilde{\I}_{\la}$ is coercive and bounded from below.

Finally, choose $\Lambda>0$ such that
\begin{align*}
	\Lambda \leq \min\{\tilde{\la},\tilde{\la}_1,\tilde{\la}_2,\tilde{\la}_3\},
\end{align*}
where $\tilde{\la}$ is defined in \eqref{tiltla}, $\tilde{\la}_1,\tilde{\la}_2$ are as above and $\tilde{\la}_3=\frac{k}{qC^k_k}$.

\begin{lemma}\label{lem6.2}
	Let $\la \in (0,\Lambda)$. If $\tilde{\I}_{\la}(u)< 0$, then $\mm{u}_W<\al_0(\la)$ and $\tilde{\I}_{\la}(z)=\I_{\la}(z)$ for every $z$ in a sufficiently small neighborhood of $u$. Moreover, $\tilde{\I}_{\la}$ satisfies the \textnormal{(PS)$_c$}-condition at level $c$ for every $c<0$.
\end{lemma}

\begin{proof}
	Fix $\la \in (0,\Lambda)$ and suppose that $\tilde{\I}_{\la}(u)< 0$. We consider two cases.

	\textbf{Case 1:} $\mm{u}_W\geq 1$

	In this case, $\theta(\mm{u}_W)=0$. Hence
	\begin{align*}
		\tilde{\I}_{\la}(u)
		&=\frac{1}{p}([u]_{s,p}^p+\mm{\na u}_p^p )+\frac{1}{q}\mm{\na u}_{q,a}^q-\frac{\la}{k}\int_{\Omega}\ve{u}^k\dx\\
		&\geq \frac{1}{q}\mm{u}_W^p-\frac{\la}{k} C^k_k \mm{u}_W^k=:\sigma_{\la}(\mm{u}_W),
	\end{align*}
	where $\si_{\la}(t):=\frac{1}{q}t^p-\frac{\la}{k} C^k_kt^k$ for all $t\in[1,\infty)$. The function $\si_{\la}$ is minimized at
	\begin{align*}
		t_*=\lt(\frac{\la C^k_k q}{p}\ri)^{\frac{1}{p-k}},
	\end{align*}
	and
	\begin{align*}
		\si_{\la}(t_*)=\frac{1}{q}\lt(\frac{\la C^k_k q}{p}\ri)^{\frac{p}{p-k}}\lt(1-\frac{p}{k}\ri)<0,
	\end{align*}
	since $k<p$. Moreover, $\si_{\la}(t)\geq 0$ if and only if $t\geq \lt(\frac{\la C^k_k q}{k}\ri)^{\frac{1}{p-k}}$. If $\la\leq \tilde{\la}_3=\frac{k}{qC^k_k}$, we deduce that $\min_{t\in[1,\infty)}\si_{\la}(t)\geq 0$, which further implies $\tilde{\I}_{\la}(u)\geq 0$ for $\mm{u}_W\geq 1$, contradicting $\tilde{\I}_{\la}(u)<0$. Therefore, $\mm{u}_W\geq1$ is impossible.

	\textbf{Case 2:} $\mm{u}_W<1$

	We estimate
	\begin{align*}
		\tilde{\I}_{\la}(u)\geq\frac{1}{q}\mm{u}_W^q-\frac{\la}{k}C^k_k\mm{u}_W^k-\frac{1}{p^*}C^{p^*}_{p^*}\mm{u}_W^{p^*}\theta(\mm{u}_W)
		=:\tilde{g}_{\la}(\mm{u}_W),
	\end{align*}
	where $\tilde{g}_{\la}(t):=\frac{1}{q}t^q-\frac{\la}{k}C^k_k t^k-\frac{1}{p^*}C^{p^*}_{p^*}t^{p^*}\theta(t)$. Since $0 \leq \theta\leq 1$, we have
	\begin{align}\label{6.13}
		\tilde{g}_{\la}(t)\geq {g}_{\la}(t)\geq 0 \quad \text{for all }  t\in[\al_0(\la),\min\{\al_1(\la),1\}].
	\end{align}
	Thus, if $\min\{\al_1(\la),1\}=1$, then $\tilde{\I}_{\la}(u)<0$ together with \eqref{6.13} implies $\mm{u}_W<\al_0(\la)$.

	If $\min\{\al_1(\la),1\}=\al_1(\la)$, then $\al_1(\la)<1$. For $\al_1(\la)<\mm{u}_W<1$, we estimate as in Case 1 by dropping the critical term (it is truncated, hence nonnegative in the energy with a minus sign removed when $\theta=0$) and obtain
	\begin{align*}
		\tilde{\I}_{\la}(u)\geq \tilde{\si}_{\la}(\mm{u}_W),\quad  \text{where }  \tilde{\si}_{\la}(t)=\frac{1}{q}t^q-\frac{\la}{k}C^k_kt^k,
	\end{align*}
	which yields again a contradiction for $\lambda\leq \tilde{\lambda}_3$, exactly as in Case 1. If instead $\al_0(\la)<\mm{u}_W\leq \al_1(\la)$, then \eqref{6.13} implies $\tilde{\I}_{\la}(u)\geq 0$, contradicting $\tilde{\I}_{\la}(u)< 0$. Hence in all cases, $\mm{u}_W<\al_0(\la)$. Since $\tilde{\I}_{\la}$ is continuous, there exists a sufficiently small neighborhood $B\subset B_{\delta}(\al_0)$ of $u$ such that $\tilde{\I}_{\la}(z)<0$ for all $z\in B$ and ${\I_{\la}}(z)=\tilde{\I}_{\la}(z)$ follows from above. Furthermore, for $c<0$, we have $\tilde{\I}_{\la}=\I_{\la}$. Thanks to Lemma \ref{PSc}, $\tilde{\I}_{\la}$ satisfies \textnormal{(PS)$_c$}-condition. This completes the proof.
\end{proof}

\begin{lemma}\label{lem6.3}
	Let $\la \in(0,\min\{\tilde{\la}_1,\tilde{\la}_2\})$. Then, for every $m\in \N$, there exists $\delta=\delta(\la,m)>0$ such that $\ga\lt(\tilde{\I}_{\la}^{-\delta}\ri)\geq m$, where
	\begin{align*}
		\tilde{\I}_{\la}^{-\delta}=\{u\in W\colon  \tilde{\I}_{\la}(u)\leq -\delta\}.
	\end{align*}
\end{lemma}

\begin{proof}
	Fix $\la \in(0,\min\{\tilde{\la}_1,\tilde{\la}_2\})$ and $m\in \N$. Let $W_m\subset W$ be an $ m$-dimensional subspace. Since $W_m$ is finite dimensional, all norms are equivalent. In particular, there exists a constant $C(m)>0$ such that
	\begin{align}\label{6.31}
		\mm{u}_{k}^k\geq C(m)\mm{u}_W^k, \quad \text{for all }u\in W_m.
	\end{align}
	Let $u\in W_m$ with $\mm{u}_W\leq \alpha_0<1$. Using the definition of $\tilde{\I}_{\la}$, the truncation property, and \eqref{6.31}, we obtain
	\begin{align}\label{6.32}
		\tilde{\I}_{\la}(u) \leq \frac{1}{p} \mm{u}_W^p-\frac{\la}{k}C(m)\mm{u}_W^k.
	\end{align}
	Choose $r_0>0$ such that
	\begin{align}\label{6.33}
		r_0 <\min\lt(\al_0(\la), \lt(\frac{\la C(m)p}{k}\ri)^{\frac{1}{p-k}}\ri).
	\end{align}
	Next, define $S_{r_0}=\{u\in W_m\colon  \mm{u}_W=r_0\}$. Since $W_m$ is $m$-dimensional, $S_{r_0}$ is homeomorphic to the $(m-1)$-dimensional sphere. Moreover, from \eqref{6.32} and \eqref{6.33}, we obtain
	\begin{align*}
		\tilde{\I}_{\la}(u)<0, \quad \text{for all }u \in S_{r_0}.
	\end{align*}
	Therefore, by continuity of $\tilde{\I}_{\la}$ and compactness of $S_{r_0}$, there exists $\delta=\delta(\la,m)>0$ such that $\tilde{\I}_{\la}(u)<-\delta$ for every  $u \in S_{r_0}$. Hence, $S_{r_0} \subset \tilde{\I}_{\la}^{-\delta}\setminus\{0\}$ and by Proposition \ref{genus}, we get
	\begin{align*}
		\gamma\lt(\tilde{\I}_{\la}^{-\delta}\ri)\geq \gamma(S_{r_0})=m.
	\end{align*}
\end{proof}

\begin{proof}[Proof of Theorem \ref{t3}]
	Define
	\begin{align}\label{15-12-2}
		P_{c}=\{u\in W\colon \tilde{\I}_{\la}(u)=c,  \tilde{\I}'_{\la}(u)=0\}.
	\end{align}
	For each $m\in\N$, set
	\begin{align*}
		c_{m}=\inf_{A\in \Sigma_{m}}\sup_{u\in A}\tilde{\I}_{\la}(u),
	\end{align*}
	where
	\begin{align*}
		\Sigma_{m}=\{A\subseteq W\setminus \{0\}\colon  A \text{ is closed, symmetric, and } \gamma(A)\geq m\}.
	\end{align*}
	By Lemma \ref{lem6.3}, for every $m\in\mathbb{N}$, there exists $\delta_{m}>0$ such that
	\begin{align*}
		\gamma(\tilde{\I}_{\la}^{-\delta_{m}})\geq m, \quad  \tilde{\I}_{\la}^{-\delta_{m}}=\{u\in W\colon \tilde{\I}_{\la}(u)\leq -\delta_{m}\}.
	\end{align*}
	Hence, $\tilde{\I}_{\la}^{-\delta_{m}}\in\Sigma_{m}$. Since $\tilde{\I}_{\la}$ is bounded from below, it follows that
	\begin{align*}
		-\infty<c_{m}\leq\sup\limits_{u\in \tilde{\I}_{\la}^{-\delta_{m}}}\tilde{\I}_{\la}(u)\leq -\delta_{m}<0.
	\end{align*}
	Thus each $c_m$ is finite and negative.

	Let $P_{c_m}$ be as defined in \eqref{15-12-2}. We claim that $P_{c_{m}}$ is compact. To this end, let $\{u_{n}\}_{n\in\mathbb{N}}\subset P_{c_{m}}$ be a bounded sequence. Then $\tilde{\I}_{\la}(u_{n})=c_{m}$ and $\tilde\I'_{\la}(u_{m})=0$. By Lemma \ref{lem6.2},  $\tilde{\I}_{\la}$ satisfies the  \textnormal{(PS)$_{c_{m}}$} condition, i.e., there exist $u\in W$ such that $\{u_{n}\}_{n\in\mathbb{N}}$ converges to $u$ strongly in $W$, up to a subsequence if necessary. This implies $\tilde{\I}_{\la}(u)=c_{m}$ and  $\tilde{\I}'_{\la}(u)=0$, hence $u\in P_{c_{m}}$. Thus, $\{u_{n}\}_{n\in\mathbb{N}}$ admits a convergent subsequence in $P_{c_{m}}$, consequently $P_{c_{m}}$ is compact. Moreover, by Proposition \ref{genus}, we have $\gamma{(P_{c_{m}})}<\infty$.

	Since $\Sigma_{m+1}\subset \Sigma_{m}$, we have $c_{m}\leq c_{m+1}$. If all $c_{m}$ are distinct,
	then we obtain an infinite sequence of distinct critical values $c_m<0$, hence infinitely many distinct critical points of $\tilde{\I}_{\la}$ with negative energy.

	Suppose instead that for some $m\in\mathbb{N}$ and some $j\geq 1$,
	\begin{align*}
		c=c_{m}=c_{m+1}=\cdots=c_{m+j}.
	\end{align*}
	Then, by Lemma 3.6 of da Silva--Fiscella--Viloria \cite{daSilva-Fiscella-Viloria-2024}, we have $\gamma{(P_{c_{m}})}\geq j+1$. Following the classical argument presented by Rabinowitz \cite[Remark 7.3]{Rabinowitz-1986}, it follows that $P_{c_{m}}$ contains infinitely many distinct elements. Hence $\tilde{\I}_{\la}$ possesses infinitely many distinct critical points with negative energy.

	Finally, by Lemma \ref{lem6.2}, on negative levels the truncated functional $\tilde{\I}_{\la}$ coincides locally with the original functional $\I_{\la}$. Therefore, $\I_{\la}$ also has infinitely many distinct critical points with negative energy. This completes the proof.
\end{proof}

\section{Local \texorpdfstring{$p$}{p}-Laplacian with fractional double phase}\label{Section_4}

In this section, we study problem \eqref{frac_dbl} and prove Theorems \ref{t4} and \ref{t5}. We show that problem \eqref{frac_dbl} admits at least two nontrivial constant sign solutions by employing the Mountain Pass Theorem, and we obtain a least energy sign-changing solution by means of the Poincar\'{e}-Miranda existence theorem and the quantitative deformation lemma. Throughout this section we suppose that \eqref{H2} holds.

The energy functional $J\colon  E\to\R$ associated with \eqref{frac_dbl} is
\begin{align*}
	J(u)=\frac{1}{p}\int_{\Omega}|\nabla u|^p\dx+ \frac{1}{p}\int_{Q}|D_s(u)|^p\dv+\frac{1}{q}\int_{Q}b(x,y)|D_s(u)|^q\dv -\int_{\Omega} F(x,u)\dx
\end{align*}
where $F(x,t)=\int_{0}^{t}f(x,\tau)\,\mathrm{d}\tau$, and $D_s(u)$ as well as $\dv$ are defined in \eqref{notion-ds-dv}. By Lemma 3.10 of de Albuquerque--de Assis--Carvalho--Salort \cite{deAlbuquerque-deAssis-Carvalho-Salort-2023} we have $J\in C^{1}(E,\R)$, and its derivative is given by
\begin{align*}
	\langle J'(u),v\rangle & =\int_{\Omega} |\nabla u|^{p-2}\nabla u\cdot \nabla v\dx+\int_{Q}|D_s(u)|^{p-2}D_s(u)D_s(v)\dv \\
	& \quad +\int_{Q}b(x,y)|D_s(u)|^{q-2}D_s(u)D_s(v)\dv
	-\int_{\Omega} f(x,u) v \dx
\end{align*}
for $u,v\in E$. We also note that if $u \in E$ satisfies $u^{+} \neq 0 \neq u^{-}$, then
\begin{align*}
	J(u) > J(u^{+}) + J(-u^{-}), \quad \langle J'(u), u^{+} \rangle > \langle J'(u^{+}), u^{+} \rangle, \quad \langle J'(u), -u^{-} \rangle > \langle J'(-u^{-}), -u^{-} \rangle.
\end{align*}

\begin{remark}\label{rem_condition}
	The following properties are immediate consequences of \eqref{H2}:
	\begin{enumerate}
		\item[\textnormal{(i)}]
			By \eqref{H2}\eqref{H2ii} and \eqref{H2iii}, for every $\varepsilon > 0$ there exists $c_{*} > 0$ such that
			\begin{align*}
				F(x,t) \le \varepsilon |t|^{p} + c_{*} |t|^{r}, \quad \text{for a.a.\,} x \in \Omega \text{ and for all } t \in \mathbb{R}.
			\end{align*}
		\item[\textnormal{(ii)}]
			By \eqref{H2}\eqref{H2ii} and \eqref{H2iv}, for every $M > 0$, there exists $c_{M} > 0$ such that
			\begin{align*}
				F(x,t) \ge \frac{M}{q}|t|^{q} - c_{M}, \quad \text{for a.a.\,} x \in \Omega \text{ and for all } t \in \mathbb{R}.
			\end{align*}
		\item[\textnormal{(iii)}]
			From \eqref{H2}\eqref{H2ii} and \eqref{H2iv} we infer that $q < r$. Moreover,
			\begin{align*}
				\lim_{t\to \pm\infty} \frac{F(x,t)}{|t|^{q}} = \infty\quad \text{uniformly for a.a.\,} x \in \Omega.
			\end{align*}
		\item[\textnormal{(iv)}]
			From \eqref{H2}\eqref{H2iii} it follows that $f(x,0) = 0$ for a.a.\,$x \in \Omega$.
	\end{enumerate}
\end{remark}

To obtain constant sign solutions of problem \eqref{frac_dbl}, we introduce the truncated functionals $J_{\pm}\colon E \to \R$ defined by
\begin{align*}
	J_{\pm}(u)=\frac{1}{p}\int_{\Omega}|\nabla u|^p\dx+ \frac{1}{p}\int_{Q}|D_s(u)|^p\dv+\frac{1}{q}\int_{Q}b(x,y)|D_s(u)|^q\dv -\int_{\Omega} F(x,\pm u^{\pm})\dx.
\end{align*}
Next, we define the operator $\mathcal{A}\colon  E \to E^{*}$ corresponding to the principal part of the energy functional by
\begin{align*}
	\langle \mathcal{A}(u),v\rangle
	& =\int_{\Omega} |\nabla u|^{p-2}\nabla u \cdot \nabla v\dx+\int_{Q}|D_s(u)|^{p-2}D_s(u)D_s(v)\dv \\
	& \quad +\int_{Q}b(x,y)|D_s(u)|^{q-2}D_s(u)D_s(v)\dv,
\end{align*}
where $\langle \cdot,\cdot \rangle$ denotes the duality pairing between $E$ and $E^*$.

The following lemma establishes the key compactness property of the operator $\mathcal{A}$, which will be crucial in proving the existence of constant sign solutions.

\begin{lemma}\label{s_+}
	The operator $\mathcal{A}$ satisfies the \textnormal{(S$_+$)}-property. More precisely, if $\{u_{n}\}_{n\in\mathbb{N}}\subset E$ is such that
	$u_{n}\rightharpoonup u$ weakly in $E$ and
	\begin{align*}
		\limsup_{n\to\infty}\, \langle \mathcal{A}(u_{n}),u_{n}-u\rangle \leq 0,
	\end{align*}
	then $u_{n}\to u$ strongly in $E$.
\end{lemma}

\begin{proof}
	Let $\{u_{n}\}_{n\in\mathbb{N}}\subset E$ satisfy $u_{n}\rightharpoonup u$ in $E$ and
	\begin{align}\label{s+1}
		\limsup_{n\to\infty}\,\langle \mathcal{A}(u_{n}),u_{n}-u\rangle \le 0.
	\end{align}
	Since the space $E$ is continuously embedded into both $W_{0}^{1,p}(\Omega)$ and $W_{0}^{s,\mathcal{G}}(\Omega)$, the weak convergence in $E$ implies
	\begin{align}
		u_{n} & \rightharpoonup u \quad \text{in } W_{0}^{1,p}(\Omega), \label{weak1}\\
		u_{n} & \rightharpoonup u \quad \text{in } W_{0}^{s,\mathcal{G}}(\Omega). \label{weak2}
	\end{align}
	Moreover, the weak convergence in $E$ yields
	\begin{align}\label{s+2}
		\lim_{n\to\infty}\langle \mathcal{A}(u),u_{n}-u\rangle = 0.
	\end{align}
	Combining \eqref{s+2}, \eqref{s+1} and Lemma \ref{inequalities}, we obtain
	\begin{align*}
		\lim_{n\to\infty}\langle \mathcal{A}(u_{n})-\mathcal{A}(u),u_{n}-u\rangle = 0.
	\end{align*}
	Using the definition of $\mathcal{A}$, we have
	\begin{align*}
		&\langle \mathcal{A}(u_n)-\mathcal{A}(u),u_n-u\rangle\\
		&=  \int_{\Omega}\big(|\nabla u_{n}|^{p-2}\nabla u_{n}-|\nabla u|^{p-2}\nabla u\big)\cdot\nabla (u_{n}-u)\dx \\
		& \quad+ \int_{Q}\big(|D_{s}u_{n}|^{p-2}D_{s}u_{n}-|D_{s}u|^{p-2}D_{s}u\big)D_{s}(u_{n}-u)\dv  \\
		&\quad +\int_{Q}b(x,y)\big(|D_{s}u_{n}|^{q-2}D_{s}u_{n}-|D_{s}u|^{q-2}D_{s}u\big)D_{s}(u_{n}-u)\dv.
	\end{align*}

	By the standard monotonicity inequalities for the $p$-Laplace operator and the fractional $p$- and $q$-Laplace  operators, each term in the above expression is nonnegative, see Lemma \ref{inequalities}. Therefore, since their sum converges to zero, each term must converge to zero separately. In particular,
	\begin{align}\label{s+4}
		\lim_{n\to\infty}\int_{\Omega} \big(|\nabla u_{n}|^{p-2}\nabla u_{n}-|\nabla u|^{p-2}\nabla u\big)\cdot\nabla (u_n-u)\dx = 0,
	\end{align}
	and
	\begin{equation}\label{s+5}
		\begin{aligned}
			& \lim_{n\to\infty}\Bigg(\int_{Q}\big(|D_{s}(u_{n})|^{p-2}D_{s}(u_{n})-|D_{s}(u)|^{p-2}D_{s}(u)\big)D_{s}(u_n-u)\dv \\
			&\qquad\qquad +\int_{Q}b(x,y)\big(|D_{s}(u_{n})|^{q-2}D_{s}(u_{n})-|D_{s}(u)|^{q-2}D_{s}(u)\big)D_{s}(u_n-u)\dv \Bigg)= 0.
		\end{aligned}
	\end{equation}
	From \eqref{weak1}, \eqref{s+4} and Lemma 2.5 by Colasuonno--Pucci--Varga \cite{Colasuonno-Pucci-Varga-2012}, it follows that $u_{n}\to u$ strongly in $W_{0}^{1,p}(\Omega)$. Similarly, using \eqref{weak2}, \eqref{s+5} and Theorem 3.14 of de Albuquerque--de Assis--Carvalho--Salort \cite{deAlbuquerque-deAssis-Carvalho-Salort-2023}, we deduce
	$u_{n}\to u$ strongly in $W_{0}^{s,\mathcal{G}}(\Omega)$. Combining these two strong convergences yields $u_{n}\to u$ strongly in $E$. This completes the proof.
\end{proof}

\begin{lemma}\label{cerami_condition}
	The functionals $J_{\pm}$ satisfy the \textnormal{(C)}-condition.
\end{lemma}

\begin{proof}
	We prove the result for $J_{+}$. The proof for $J_{-}$ is analogous.

	Let $\{u_{n}\}_{n\in\mathbb{N}}\subseteq E$ be a sequence such that, for some $c_{1}>0$,
	\begin{align*}
		|J_{+}(u_{n})|<c_{1}, \qquad (1+\|u_{n}\|_{E})J_{+}'(u_{n})\to 0 \quad \text{in } E^{*}.
	\end{align*}
	Thus,
	\begin{equation}\label{p11}
		\begin{aligned}
			&\bigg|\frac{1}{p}\int_{\Omega}|\nabla u_n|^p\dx+ \frac{1}{p}\int_{Q}|D_s(u_{n})|^p\dv\\
			&+\frac{1}{q}\int_{Q}b(x,y)|D_s(u_{n})|^q\dv-\int_{\Omega} F(x, u_n^{+})\dx\bigg|< c_1
		\end{aligned}
	\end{equation}
	and, for all $v\in E$,
	\begin{equation}\label{p222}
		\begin{aligned}
			& \bigg|\int_{\Omega}|\nabla u_n|^{p-2}\nabla u_n \cdot\nabla v\dx +\int_{Q}|D_s(u_{n})|^{p-2}D_s(u_{n})D_s(v)\dv \\
			&
			+\int_{Q}b(x,y)|D_s(u_{n})|^{q-2}D_s(u_{n})D_s(v)\dv-\int_{\Omega} f(x,u^+_n)v\dx\bigg|\\
			&\leq \frac{o_n(1)\|v\|_{E}}{(1+\|u_{n}\|_{E})}.
		\end{aligned}
	\end{equation}
	Taking $v=-u^{-}_n$ as a test function in \eqref{p222}, we obtain
	\begin{align*}
		\left|\int_{\Omega}|\nabla u^{-}_n|^p\dx+ \int_{Q}|D_s(u^{-}_n)|^p\dv+\int_{Q}b(x,y)|D_s(u^{-}_n)|^q\dv\right|\leq o_n(1).
	\end{align*}
	Therefore, using Lemma \ref{re2}, we have $\|u^{-}_{n}\|_{E}\to 0$, and consequently
	\begin{align}\label{c71}
		u^{-}_{n}\to 0 \quad\text{in } E.
	\end{align}
	Next, taking $v=u_n$ as a test function in \eqref{p222} and subtracting it from \eqref{p11} after multiplying with $q$, we obtain
	\begin{align}\label{p33}
		\left(\frac{q}{p}-1\right)\int_{\Omega}|\nabla u_n|^p\dx & + \left(\frac{q}{p}-1\right)\int_{Q}|D_s(u_{n})|^p\dv +\int_{\Omega} (u_n^{+}f(x, u_n^{+})-qF(x, u_n^{+}))\dx< \hat{c}_2
	\end{align}
	for some $\hat{c}_2>0$.

	We claim that $\{u^{+}_{n}\}_{n\in\mathbb{N}}$ is bounded in $E$. Suppose by contradiction that
	\begin{align}\label{c_3}
		\|u^{+}_{n}\|_{E}\to\infty.
	\end{align}
	Define $y_{n}=\frac{u^{+}_{n}}{\|u^{+}_{n}\|_{E}}$, then $\|y_{n}\|_{E}=1,$ which implies the boundedness of $\{y_{n}\}_{n\in\mathbb{N}}$ in $E$. By reflexivity of $E$, there exists $y\in E$ such that $y\geq 0$ and up to a subsequence, we have
	\begin{align}\label{conv_yn}
		\begin{cases}
			y_{n}   \rightharpoonup y & \text{in } E, \\
			y_{n}(x) \to y(x)  & \text{for a.a.\,} x\in\Omega, \\
			y_{n}  \to y & \text{in } L^{r}(\Omega) \text{ for } r\in[1,p^*).
		\end{cases}
	\end{align}
	We show that $y\equiv 0$. Suppose that  $\Omega^{+}=\{x\in\Omega\colon y^{+}(x)> 0\}$ has positive measure. Then
	\begin{align*}
		\lim_{n\to\infty}u^+_{n}(x)=\lim_{n\to\infty}y_{n}(x)\|u^+_{n}\|_{E}\to\infty \quad \text{in } \Omega^{+}.
	\end{align*}
	By Remark \ref{rem_condition} (iii), one gets
	\begin{align*}
		\lim_{n\to\infty}\frac{F(x,u^+_{n}(x))}{\|u^+_{n}\|_{E}^{q}}=\infty \quad \text{in } \Omega^{+}.
	\end{align*}
	Fatou's lemma yields
	\begin{align}\label{c1}
		\liminf_{n\to\infty}\int_{\Omega^+}\frac{F(x,u^+_{n}(x))}{\|u^+_{n}\|_{E}^{q}}\dx=\infty.
	\end{align}
	Using \eqref{H2} \eqref{H2ii} and \eqref{H2iv} gives
	\begin{align}\label{c2}
		F(x,s)\geq -c_2\quad \text{for a.a.\,} x\in\Omega \text{ and for all } s\in\R,
	\end{align}
	for some $c_2>0$. From \eqref{c1} and \eqref{c2}, we obtain
	\begin{align*}
		\int_{\Omega}\frac{F(x,u^+_{n}(x))}{\|u^+_{n}\|_{E}^{q}}\dx
		& =\int_{\Omega^+}\frac{F(x,u^+_{n}(x))}{\|u^+_{n}\|_{E}^{q}}\dx+\int_{\Omega\setminus \Omega^+}\frac{F(x,u^+_{n}(x))}{\|u^+_{n}\|_{E}^{q}}\dx \\
		& \geq \int_{\Omega^+}\frac{F(x,u^+_{n}(x))}{\|u^+_{n}\|_{E}^{q}}\dx-\frac{c_2}{\|u^+_{n}\|_{E}^{q}}|\Omega|.
	\end{align*}
	Thus, from \eqref{c_3} and   \eqref{c1}, we have
	\begin{align}\label{c4}
		\lim_{n\to\infty}\int_{\Omega}\frac{F(x,u^+_{n}(x))}{\|u^+_{n}\|_{E}^{q}}\dx=\infty.
	\end{align}
	By using \eqref{p11} and Lemma \ref{re2}, for sufficiently large $n$, we have, for some $c_3>0$,
	\begin{align*}
		\int_{\Omega}\frac{F(x,u^+_{n}(x))}{\|u^+_{n}\|_{E}^{q}}\dx
		& < c_3+\frac{1}{p}\int_{\Omega}\frac{|\nabla y_{n}(x)|^{p}}{\|u^+_{n}\|_{E}^{q-p}}\dx+ \frac{1}{p}\int_{Q}\dfrac{|y_{n}(x)-y_{n}(y)|^p}{\|u^+_{n}\|_{E}^{q-p}|x-y|^{N+sp}}\dxy \\
		&\quad +\frac{1}{q}\int_{Q}b(x,y)\dfrac{|y_{n}(x)-y_{n}(y)|^q}{|x-y|^{N+sq}}\dxy  \\
		& \leq c_3+ \frac{\eta(y_n)}{p}\\
		& = c_3+\frac{1}{p},
	\end{align*}
	which is a contradiction to \eqref{c4}. Hence, $y\equiv 0$ a.e.\,in $\Omega$.

	Let $0\leq t_{n}\leq 1$ be such that
    \begin{align*}
		J_+(t_{n}u^+_{n})=\max_{0\leq t\leq 1}J_+(tu^+_{n}).
	\end{align*}
	For any fixed $\lambda\geq 1$ there exists $n_0>0$ such that $0<\frac{\lambda}{\|u^+_{n}\|_{E}}\leq 1$ for all $n\geq n_0$. Hence
	\begin{align*}
		J_+(t_{n}u^+_{n})\geq J_+\left( \frac{\lambda}{\|u^+_{n}\|_{E}}u^+_{n}\right) =J_+(\lambda y_{n}),
	\end{align*}
	for all $n\geq n_0$. By \eqref{conv_yn}, we have
	\begin{align*}
		y_{n} \to 0 \quad\text{in } L^{r}(\Omega) \text{ for } r\in[1,p^*)
	\end{align*}
	which implies
	\begin{align}\label{conv_f(yn)}
		\int_{\Omega} F(x,\lambda y_{n})\dx\to 0.
	\end{align}
	By using  norm-modular relation and \eqref{conv_f(yn)}, we get, for some $c_4>0$,
	\begin{align*}
		J_+(\lambda y_{n})
		& = \frac{1}{p}\int_{\Omega}|\nabla (\lambda y_{n})|^p\dx+ \frac{1}{p}\int_{Q}|D_s(\lambda y_{n})|^p\dv+\frac{1}{q}\int_{Q}b(x,y)|D_s(\lambda y_{n})|^q\dv \\
		&\quad -\int_{\Omega} F(x,\lambda y_{n})\dx \\
		& \geq \frac{\lambda^p}{q} \eta(y_n)-\int_{\Omega} F(x,\lambda y_{n})\dx\\
		& \geq \frac{\lambda^p}{q} -c_4
	\end{align*}
	for all $n \geq n_{0}$. Because $\lambda \geq 1$ is arbitrarily chosen, we see that
	\begin{align}\label{c7}
		J_{+}(t_{n} u^+_{n}) \to +\infty \quad \text{as } n \to \infty.
	\end{align}

	Taking \eqref{p11} into account gives
	\begin{align}\label{c6}
		J_{+}(0) = 0
		\quad \text{and} \quad
		J_{+}(u^{+}_{n}) \leq c_5 \quad \text{for all } n \in \mathbb{N},
	\end{align}
	for some $c_5 > 0$. From \eqref{c7} and \eqref{c6} there exists $n_{2} \in \mathbb{N}$ such that
	\begin{align*}
		t_{n} \in (0,1) \quad \text{for all } n \geq n_{2}.
	\end{align*}
	Therefore, we have
	\begin{align*}
		0=\frac{\mathrm{d}}{\mathrm{d}t}J_+(tu^+_{n})\bigg|_{t=t_n} & =t_n^{p-1}\int_{\Omega}|\nabla (u^+_{n})|^p\dx+t_n^{p-1}\int_{Q}|D_s(u^+_{n})|^p\dv+t_n^{q-1}\int_{Q}b(x,y)|D_s(u^+_{n})|^q\dv \\
		&\quad -\int_{\Omega} f(x,t_nu^+_{n})u^+_{n}\dx \\
		& =\langle J'_+(t_nu^+_{n}),u^+_n\rangle, \quad \text{for all } n \geq n_{2}.
	\end{align*}
	By  \eqref{p33}, it follows
	\begin{align*}
		qJ_+(t_nu^+_{n}) & =\frac{q}{p}\int_{\Omega}|\nabla (t_nu^+_{n})|^p\dx+ \frac{q}{p}\int_{Q}|D_s(t_nu^+_{n})|^p\dv+\int_{Q}b(x,y)|D_s(t_nu^+_{n})|^q\dv \\
		&\quad -q\int_{\Omega} F(x,t_nu^+_{n})\dx \\
		& =\left(\frac{q}{p}-1\right)\int_{\Omega}|\nabla (t_nu^+_{n})|^p\dx+ \left(\frac{q}{p}-1\right)\int_{Q}|D_s(t_nu^+_{n})|^p\dv \\
		&\quad +\int_{\Omega} (t_nu_n^{+}f(x, t_nu_n^{+})-qF(x, t_nu_n^{+}))\dx\\
		& \leq \left(\frac{q}{p}-1\right)\int_{\Omega}|\nabla (u_{n})|^p\dx+ \left(\frac{q}{p}-1\right)\int_{Q}|D_s(u_{n})|^p\dv\\
		& \quad+\int_{\Omega} (u_n^{+}f(x,u_n^{+})-q F(x, u_n^{+}))\dx\\
		& \leq c_2,
	\end{align*}
	which is a contradiction to \eqref{c7}. Hence,  \eqref{c_3} cannot be true. Therefore, $\{u^{+}_n\}_{n\in\mathbb{N}}$ is bounded in $E$. Further from \eqref{c71}, we conclude that $\{u_n\}_{n\in\mathbb{N}}$ is bounded in $E$ as well. By reflexivity of $E$, there exists $u\in E$ such that up to a subsequence, we have
	\begin{align*}
		\begin{cases}
			u_{n}\rightharpoonup u &\text{in } E,\\
			u_{n}(x)\to u(x)& \text{for a.a.\,} x\in\Omega,\\
			u_{n}\to u& \text{in }L^{r}(\Omega)\text{ for }r\in [1,p^*).
		\end{cases}
	\end{align*}
	Taking $v=u_n-u$ as a test function in \eqref{p222}, we obtain
	\begin{equation}\label{pa1}
		\begin{aligned}
			&\bigg|\int_{\Omega}|\nabla u_n|^{p-2} \nabla u_n \cdot\nabla (u_n-u)\dx +\int_{Q}|D_s(u_n)|^{p-2}D_s(u_n)D_s(u_n-u)\dv \\
			& +\int_{Q}b(x,y)|D_s(u_n)|^{q-2}D_s(u_n)D_s(u_n-u)\dv-\int_{\Omega} f(x,u^+_n)(u_n-u)\dx\bigg|\leq o_n(1)\|u_n-u\|_{E}.
		\end{aligned}
	\end{equation}
	On the other hand, from \eqref{H2} \eqref{H2ii} and H\"{o}lder's inequality, we get
	\begin{equation}\label{pa2}
		\begin{aligned}
			\int_{\Omega} f(x,u^+_n)(u_n-u)\dx
			& \leq  c\int_{\Omega}  (u_n-u)\dx+c\int_{\Omega}  |u^+_n|^{r-1}(u_n-u)\dx\\
			&\leq c\|u_n-u\|_{1}+c\|u^+_n\|_r^{r-1}\|u_n-u\|_r\to 0.
		\end{aligned}
	\end{equation}
	By passing to the $\limsup$ as $n\to\infty$ in \eqref{pa1} and using \eqref{pa2}, we have
	\begin{align*}
		\limsup_{n\to\infty}&\bigg(\int_{\Omega}|\nabla u_n|^{p-2}\nabla u_n\cdot \nabla (u_n-u)\dx +\int_{Q}|D_s(u_n)|^{p-2}D_s(u_n)D_s(u_n-u)\dv \\
		&\quad +\int_{Q}b(x,y)|D_s(u_n)|^{q-2}D_s(u_n)D_s(u_n-u)\dv\bigg)\leq 0,
	\end{align*}
	that is,
	\begin{align*}\limsup_{n\to\infty}\langle \mathcal{A}(u_{n}),u_{n}-u\rangle \le 0.\end{align*} By Lemma \ref{s_+}, we have $u_n\to u$ in $E$. This completes the proof.
\end{proof}

\begin{lemma}\label{mpt_geometry_1}
	There exist positive real numbers $\delta$ and $\beta$ such that
	\begin{align*}
		J(u) \geq \beta > 0 \quad \text{and} \quad J_{\pm}(u) \geq \beta  > 0 \quad \text{for all }\, 0 < \|u\|_{E} < \delta.
	\end{align*}
\end{lemma}

\begin{proof}
	We demonstrate the claim for the functional $J$, as the same reasoning applies to $J_{\pm}$. By Remark \ref{rem_condition}, for every $\varepsilon > 0$ there exists a constant $c_{*} > 0$ such that
	\begin{align}\label{f_growth}
		F(x,t) \le \varepsilon |t|^{p} + c_{*} |t|^{r}, \quad \text{for a.a.\,} x \in \Omega \text{ and for all } t \in \mathbb{R}.
	\end{align}
	Let $\|u\|_{E} < 1$. Using \eqref{f_growth}, the Sobolev embedding \eqref{embed} and Lemma \ref{re2}, we estimate
	\begin{align*}
		J(u)
		& \ge \frac{1}{p} \int_{\Omega} |\nabla u|^{p} \dx + \frac{1}{p} \int_{Q} |D_{s}(u)|^{p} \dv + \frac{1}{q} \int_{Q} b(x,y) |D_{s}(u)|^{q} \dv - \varepsilon \int_{\Omega} |u|^{p} \dx - c_{*} \int_{\Omega} |u|^{r} \dx \\
		& \ge  \frac{1}{p}\mm{\na u}_p^{p} + \frac{1}{q} \Bigg(\int_{Q} |D_{s}(u)|^{p} \dv + \int_{Q} b(x,y) |D_{s}(u)|^{q} \dv \Bigg) - \varepsilon S^p_{p} \mm{\na u}_p^{p} - c_{*} \mathbb{S}^r_{r} \|u\|_{E}^{r}\\
		& \ge\mm{\na u}_p^{p} \Big(\frac{1}{p} - \varepsilon S^p_{p}\Big) + \frac{1}{q} \Bigg(\int_{Q} |D_{s}(u)|^{p} \dv + \int_{Q} b(x,y) |D_{s}(u)|^{q} \dv \Bigg)
		- c_{*} \mathbb{S}^r_{r} \|u\|_{E}^{r} \\
		& \ge  \min\left\{\Big(\frac{1}{p} - \varepsilon S^p_{p}\Big),\frac{1} {q}\right\}\eta(u) - c_{*} \mathbb{S}^r_{r} \|u\|_{E}^{r}\\
		& \ge  \min\left\{\Big(\frac{1}{p} - \varepsilon   S^p_{p}\Big),\frac{1}{q}\right\}\|u\|_{E}^{q} - c_{*} \mathbb{S}^r_{r} \|u\|_{E}^{r}.
	\end{align*}
	Now choose $\varepsilon > 0$ small such that $\frac{1}{p} - \varepsilon S^p_{p} > 0$, and take $\delta > 0$ sufficiently small so that
	\begin{align*}
		\beta := \min\left\{ \left( \frac{1}{p} - \varepsilon S^p_{p} \right), \frac{1}{q} \right\} \delta^{q}
		- c_{*} \mathbb{S}_{r} \delta^{r} > 0,
	\end{align*}
	which is possible since $q < r$ by Remark~\ref{rem_condition}.
	It follows that $J(u)\geq \beta> 0$ for all $0 < \|u\|_{E} < \delta$ and for some $\delta > 0$.
\end{proof}

\begin{lemma}\label{mpt_geometry_2}
	If $u \ge 0$ a.e.\,in $\Omega$, then $J_{\pm}(tu) \to -\infty$ as $t \to \pm\infty$.
\end{lemma}

\begin{proof}
	We prove the statement for $J_{+}$. The case of $J_{-}$ follows analogously. By Remark \ref{rem_condition}, for every $M>0$ there exists a constant $c_M>0$ such that
	\begin{align}\label{f2_growth}
		F(x,t) \ge \frac{M}{q}|t|^{q} - c_M, \quad \text{for a.a.\,} x \in \Omega \text{ and for all }  t \in \mathbb{R}.
	\end{align}
	Let $u \in E$ with $u \ge 0$ a.e.\,in $\Omega$. For $t>1$, using \eqref{f2_growth} and Lemma \ref{re2}, we obtain
	\begin{align*}
		J_{+}(tu)
		& \le \frac{t^{p}}{p} \int_{\Omega} |\nabla u|^{p} \dx + \frac{t^{p}}{p} \int_{Q} |D_s(u)|^{p} \dv  + \frac{t^{q}}{q} \int_{Q} b(x,y) |D_s(u)|^{q} \dv - \frac{M}{q} \int_{\Omega} |t u|^{q} \dx + c_M |\Omega| \\
		& \le \frac{t^{q}}{p} \left[ \int_{\Omega} |\nabla u|^{p} \dx + \int_{Q} |D_s(u)|^{p} \dv + \int_{Q} b(x,y) |D_s(u)|^{q} \dv \right] - \frac{M t^{q}}{q} \int_{\Omega} |u|^{q} \dx + c_M |\Omega|  \\
		& \le \frac{t^{q}}{p} \left[ \eta(u) - \frac{M}{q} \int_{\Omega} |u|^{q} \dx \right] + c_M |\Omega|  \\
		& \le \frac{t^{q}}{p} \left[ \max\{\|u\|_{E}^{p}, \|u\|_{E}^{q}\} - \frac{M}{q} \int_{\Omega} |u|^{q} \dx \right] + c_M |\Omega|.
	\end{align*}
	Choosing $M>0$ sufficiently large so that
	\begin{align*}
		\max\{\|u\|_{E}^{p}, \|u\|_{E}^{q}\} - \frac{M}{q} \int_{\Omega} |u|^{q} \dx < 0,
	\end{align*}
	we conclude that $J_{+}(tu) \to -\infty$ as $t \to \infty$.
\end{proof}

\begin{proof}[Proof of Theorem \ref{t4}]
	By Lemmas \ref{mpt_geometry_1} and \ref{mpt_geometry_2}, the functionals $J_{\pm}$ satisfy the geometric assumptions of the mountain pass theorem stated in Theorem \ref{ce2}. Hence, by Theorem \ref{ce2}, for each functional $J_{\pm}$ there exists a (C)-sequence at the level
	\begin{align*}
		c_{\pm} = \inf_{\gamma \in \Gamma} \max_{t \in [0,1]}  J_{\pm}(\gamma(t)) > 0,
	\end{align*}
	where
	\begin{align*}
		\Gamma := \{ \gamma \in C([0,1],E) \colon  \gamma(0)  = 0, J_{\pm}(\gamma(1)) < 0 \}.
	\end{align*}
	Furthermore, by Lemma \ref{cerami_condition}, each functional $J_{\pm}$ satisfies the \textnormal{(C)}-condition. Therefore, by Theorem \ref{ce1}, there exist critical points $u_{*}, v_{*} \in E$ of $J_{+}$ and $J_{-}$ at the corresponding levels $c_{+}$ and $c_{-}$, respectively. Precisely,
	\begin{align*}
		J_{+}'(u_{*}) = 0, \quad J_{+}(u_{*}) = c_{+} > 0,\quad
		J_{-}'(v_{*}) = 0, \quad J_{-}(v_{*}) = c_{-} > 0.
	\end{align*}
	Since $c_{\pm} > 0$, the critical points $u_{*}$ and $v_{*}$ are nontrivial solutions of $J_{+}$ and $J_{-}$, respectively. As $\langle J_{+}'(u_{*}), v \rangle = 0$ for all $v \in E$, testing with $v = -u_{*}^{-}$ gives
	\begin{align*}
		\langle J_+'(u_{*}),-u_{*}^{-}\rangle
		& =-\int_{\Omega} |\nabla u_{*}|^{p-2} \nabla u_{*}\cdot\nabla u_{*}^{-}\dx-\int_{Q}|D_s(u_{*})|^{p-2}D_s(u_{*})D_s(u_{*}^{-})\dv \\
		& \quad-\int_{Q}b(x,y)|D_s (u_{*})|^{q-2}D_s(u_{*})D_s(u_{*}^{-})\dv+\int_{\Omega} f(x,(u^{+}_{*})) u_{*}^{-} \dx=0.
	\end{align*}
	Since $u_*^+u_*^-=0$ a.e.\,in $\Omega$, we have $f(x,(u^{+}_{*})) u_{*}^{-}=0$ a.e.\,in $\Omega$. Hence the last integral vanishes. Using the standard identities for positive and negative parts and the monotonicity inequalities, the remaining terms give $\eta(-u_{*}^{-}) \leq 0$. Then by Lemma \ref{re2}, we deduce that $u_{*}^{-} = 0$ a.e.\,in $\Omega$. Hence, $u_{*} \geq 0$ a.e.\,in $\Omega$. Similarly, by testing the equation for $J_{-}'(v_{*}) = 0$ with $v = v_{*}^{+}$, we obtain $v_{*} \leq 0$ a.e.\,in $\Omega$. This completes the proof.
\end{proof}

In the second part, we are interested in the existence of a least energy sign-changing solution of problem \eqref{frac_dbl}, which proves Theorem \ref{t5}. For this purpose, we introduce the set
\begin{align*}
	\s:=\left\{u\in E\colon  u^{\pm}\neq 0, \langle J'(u), u^+\rangle=\langle J'(u), -u^-\rangle=0\right\},
\end{align*}
which will serve as a natural constraint for sign-changing solutions. Before carrying out the variational analysis on $\s$, we derive several estimates that will be used later. For convenience, we set
\begin{align*}
	\R_+^N=\left\{x\in \R^N\colon  u(x)>0\right\} \quad\text{and}\quad \R_-^N=\left\{x\in \R^N\colon  u(x)\leq 0\right\}.
\end{align*}
Let $\tau, \gamma \geq 0$. We compute
\begin{align*}
	& \int_{\R^N}\int_{\R^N} \ve{D_s(\tau u^+-\ga u^-)}^{p-2}D_s(\tau u^+-\ga u^-)D_s(\tau u^+) \dv\\
	&=\int_{\R_+^N}\int_{\R_+^N}\ve{D_s(\tau u^+)}^{p} \dv\\
	& \quad +\int_{\R_+^N}\int_{\R_-^N}\ve{-\ga u^-(x)-\tau u^+(y)}^{p-2}(-\ga u^-(x)-\tau u^+(y)) (-\tau u^+(y)) \frac{\dv}{\ve{x-y}^{sp}}  \\
	& \quad +\int_{\R_-^N}\int_{\R_+^N}\ve{\tau u^+(x)+\ga u^-(y)}^{p-2}(\tau u^+(x)+\ga u^-(y))(\tau u^+(x))\frac{\dv}{\ve{x-y}^{sp}} \\
	& =\int_{\R_+^N}\int_{\R_+^N}\ve{D_s(\tau u^+)}^{p} \dv+\int_{\R_+^N}\int_{\R_-^N}({\ga u^-(x)+\tau u^+(y)})^{p-1}(\tau u^+(y))\frac{\dv}{\ve{x-y}^{sp}} \\
	& \quad +\int_{\R_-^N}\int_{\R_+^N}(\tau u^+(x)+\ga u^-(y))^{p-1}(\tau u^+(x))\frac{\dv}{\ve{x-y}^{sp}} \\
	& \geq \int_{\R_+^N}\int_{\R_+^N}\ve{D_s(\tau u^+)}^{p} \dv+\int_{\R_+^N}\int_{\R_-^N}(\tau u^+(y))^{p}\frac{\dv}{\ve{x-y}^{sp}}
	+\int_{\R_-^N}\int_{\R_+^N}(\tau u^+(x))^{p}\frac{\dv}{\ve{x-y}^{sp}} \\
	& =\int_{\R^N}\int_{\R^N} \ve{D_s(\tau u^+)}^p  \dv,
\end{align*}
where we have used the basic estimate $(a+b)^{p-1}a \geq a^p$ with $a,b \geq 0$ in the last inequality. Also, we have
\begin{align*}
	D_s(\tau u^+-\ga u^-)D_s(\tau u^+)  \leq D_s^2(\tau u^+-\ga u^-).
\end{align*}

Following similar steps, one can get the following estimates:
\begin{equation}\label{est}
	\begin{aligned}
		\int_{\R^N}\int_{\R^N} \ve{D_s(\tau u^+-\ga u^-)}^{p-2}D_s(\tau u^+-\ga u^-)D_s(\tau u^+) \dv & \geq \int_{\R^N}\int_{\R^N} \ve{D_s(\tau u^+)}^p  \dv, \\
		\int_{\R^N}\int_{\R^N} \ve{D_s(\tau u^+-\ga u^-)}^{p-2}D_s(\tau u^+-\ga u^-)D_s(-\ga u^-) \dv & \geq \int_{\R^N}\int_{\R^N} \ve{D_s(-\ga u^-)}^p  \dv,  \\
		\int_{\R^N}\int_{\R^N} \ve{D_s(\tau u^+-\ga u^-)}^{p-2}D_s(\tau u^+-\ga u^-)D_s(\tau u^+) \dv & \leq  \int_{\R^N}\int_{\R^N} \ve{D_s(\tau u^+-\ga u^-)}^{p} \dv, \\
		\int_{\R^N}\int_{\R^N} \ve{D_s(\tau u^+-\ga u^-)}^{p-2}D_s(\tau u^+-\ga u^-)D_s(-\ga u^-) \dv & \leq \int_{\R^N}\int_{\R^N} \ve{D_s(\tau u^+-\ga u^-)}^{p}\dv.
	\end{aligned}
\end{equation}
Analogous estimates hold when $p$ is replaced by $q$.

Next, we prove several propositions that will be used in the proof of Theorem \ref{t5}.

\begin{proposition}\label{sprop1}
	Let $u\in E$ with $u^{\pm}\neq 0$. Then there exists a unique pair of positive numbers $(\tau_u,\ga_u)$ such that $\tau_u u^+-\ga_u u^-\in \s$. In addition, if $u\in \s$, then
	\begin{align*}
		J(\al u^+-\ba u^-) \leq J(u^+-u^-)=J(u),
	\end{align*}
	for all $\al, \ba \geq 0$ and equality holds if and only if $(\al, \ba )=(1,1)$. Moreover, for every $u \in \s$, the following properties hold:
	\begin{enumerate}
		\item[\textnormal{(i)}]
			if $\tau>1$ and $0<\ga \leq \tau$, then $\langle J'(\tau u^+-\ga u^-), \tau u^+\rangle<0$;
		\item[\textnormal{(ii)}]
			if $\tau<1$ and $0<\tau \leq \ga$, then $\langle J'(\tau u^+-\ga u^-), \tau u^+\rangle>0$;
		\item[\textnormal{(iii)}]
			if $\ga>1$ and $0<\tau \leq \ga$, then $\langle J'(\tau u^+-\ga u^-), -\ga u^-\rangle<0$;
		\item[\textnormal{(iv)}]
			if $\ga<1$ and $0<\ga \leq \tau$, then $\langle J'(\tau u^+-\ga u^-), -\ga u^-\rangle>0$.
	\end{enumerate}
\end{proposition}

\begin{proof}
	We will establish this proposition in four steps.

	\noindent \textbf{Step I:} Existence of the pair $(\tau_u,\ga_u)$.

	For $u\in E$ with $u^{\pm}\neq 0$, we have $\ve{u(x)}>0$ for a.a.\,$x\in \Om$. For $s\in(0,1)$, by hypothesis \eqref{H2} \eqref{H2vi}, we have
	\begin{align*}
		\frac{f(x,su)su}{\ve{su}^{q}}\leq\frac{f(x,u)u}{\ve{u}^{q}} \quad \text{for a.a.\,} x \in \Om,
	\end{align*}
	which implies
	\begin{align}\label{f5}
		f(x,su) su \leq s^{q}f(x,u)u \quad \text{for a.a.\,} x \in \Om.
	\end{align}
	For every $\ga \geq 0$ and $\tau>0$ sufficiently small, using \eqref{est} and \eqref{f5}, we obtain
	\begin{align*}
		\langle J'(\tau u^+-\ga u^-), \tau u^+\rangle & =\int_{\Omega} \ve{\na(\tau u^+-\ga u^-)}^{p-2}\na(\tau u^+-\ga u^-) \na(\tau u^+)\dx\\
		&\quad +\int_{Q}\ve{D_s(\tau u^+-\ga u^-)}^{p-2} D_s(\tau u^+-\ga u^-)D_s(\tau u^+) \dv\\
		& \quad +\int_{Q}b(x,y)\ve{D_s(\tau u^+-\ga u^-)}^{q-2} D_s(\tau u^+-\ga u^-)D_s(\tau u^+) \dv  \\
		& \quad-\int_{\Omega} f(x,\tau u^+-\ga u^-)\tau u^+\dx  \\
		& \geq \mm{\na(\tau u^+)}_p^p+\int_{Q}\lt(\ve{D_s(\tau u^+)}^{p}+b(x,y)\ve{D_s(\tau u^+)}^q \ri)\dv \\
		& \quad-\int_{\Omega}f(x,\tau u^+)\tau u^+\dx\\
		& \geq \tau ^p \mm{\na u^+}_p^p-\tau^q \int_{\Omega}f(x, u^+)u^+\dx>0.
	\end{align*}
	For every $\tau \geq 0$ and $\ga>0$ sufficiently small, using \eqref{est} and \eqref{f5}, we obtain
	\begin{align*}
		\langle J'(\tau u^+-\ga u^-), -\ga u^-\rangle
		& \geq  \mm{\na(-\ga u^-)}_p^p+\int_{Q}\lt(\ve{D_s(-\ga u^-)}^p+b(x,y)\ve{D_s(-\ga u^-)}^q\ri) \dv \\
		&\quad  -\int_{\Omega}f(x,-\ga u^-)(-\ga u^-)\dx\\
		& \geq \ga^p \mm{\na u^-}_p^p-\ga^q \int_{\Omega}f(x,-u^-)(-u^-)\dx>0.
	\end{align*}
	Thus, there exists a positive number $\al>0$ such that
	\begin{align}\label{ex1}
		\langle J'(\al u^+-\ga u^-), \al u^+\rangle>0 \quad \text{and} \quad \langle J'(\tau u^+-\al u^-), -\al u^-\rangle>0
	\end{align}
	for all $\tau, \ga >0$.
	Further, choose a number $\ba>\max\{1,\al\}$. For $\ga \in [0,\ba]$, using \eqref{est} and \eqref{H2} \eqref{H2iv}, we get
	\begin{align*}
		\frac{\langle J'(\ba u^+-\ga u^-), \ba u^+\rangle}{\ba^q}
		& =\frac{1}{\ba^q}\int_{\Omega} \ve{\na(\ba u^+-\ga u^-)}^{p-2}\na(\ba u^+-\ga u^-) \na(\ba u^+)\dx\\
		& \quad +\frac{1}{\ba^q}\int_{Q}\ve{D_s(\ba u^+-\ga u^-)}^{p-2} D_s(\ba u^+-\ga u^-)D_s(\ba u^+) \dv\\
		& \quad +\frac{1}{\ba^q}\int_{Q}b(x,y)\ve{D_s(\ba u^+-\ga u^-)}^{q-2} D_s(\ba u^+-\ga u^-)D_s(\ba u^+) \dv \\
		& \quad -\frac{1}{\ba^q}\int_{\Omega} f(x,\ba u^+-\ga u^-)\ba u^+\dx\\
		& \leq \frac{\mm{\na(\ba u^+)}_p^p}{\ba^q}+\frac{1}{\ba^q} \int_{Q}\ve{D_s(\ba u^+-\ga u^-)}^{p}\dv \\
		& \quad +\frac{1}{\ba^q} \int_{Q} b(x,y)\ve{D_s(\ba u^+-\ga u^-)}^{q} \dv
		-\frac{1}{\ba^q} \int_{\Omega}f(x,\ba u^+)\ba u^+\dx\\
		& \leq\mm{\na(u^+)}_p^p+\int_{Q}\lt(\ve{D_s(u^+- u^-)}^{p}+b(x,y)\ve{D_s( u^+-u^-)}^{q}\ri) \dv\\
		& \quad - \int_{\Omega}\frac{f(x,\ba u^+) (u^+)^q}{(\ba u^+)^{q-1}}\dx <0
	\end{align*}
	for $\ba >0$ large enough.
	Following similar steps for $\tau \in [0,\ba]$ and for $\ba$ large enough, one can get
	\begin{align*}
		\frac{\langle J'(\tau u^+-\ba u^-), -\ba u^-\rangle}{\ba ^q}<0.
	\end{align*}
	Hence, for $\tau,\ga \in[0, \ba]$  and $\ba$ large enough, we have
	\begin{align}\label{ex2}
		\langle J'(\ba u^+-\ga u^-), \ba u^+\rangle<0
		\quad\text{and}\quad
		\langle J'(\tau u^+-\ba u^-), -\ba u^-\rangle<0.
	\end{align}
	Define the map $\Gamma_u \colon  [0,\infty)^2 \to \R^2$ by
	\begin{align*}
		\Gamma_u(\tau, \ga):=\lt(\langle J'(\tau u^+-\ga u^-), \tau u^+\rangle,\langle J'(\tau u^+-\ga u^-), -\ga u^-\rangle\ri).
	\end{align*}
	From \eqref{ex1}, \eqref{ex2} and Theorem \ref{IVP}, we can find a pair $(\tau_u,\ga_u)\in [\al,\ba] \times [\al,\ba]$ such that $\Gamma_u(\tau_u,\ga_u)=(0,0)$. Thus, we get the existence of a pair $(\tau_u,\ga_u)$ such that $\tau_u u^+-\ga_u u^-\in \s$.

	\noindent\textbf{Step II:} Sign of $\langle J'(\tau u^+-\ga u^-), \tau u^+\rangle$ and $\langle J'(\tau u^+-\ga u^-), -\ga u^-\rangle$.

	Let $u \in \s$. Then by definition of $S$,
	\begin{equation}\label{u+}
		\begin{aligned}
			0 = \langle J'(u),  u^+\rangle
			&= \mm{\na u^+}_p^p+\int_{Q} \lt(\ve{D_s(u)}^{p-2}+b(x,y) \ve{D_s(u)}^{q-2}\ri)D_s(u)D_s(u^+) \dv\\ &\quad-\int_{\Omega}f(x,u^+)u^+\dx
		\end{aligned}
	\end{equation}
	and
	\begin{equation}\label{u-}
		\begin{aligned}
			0 = \langle J'(u),  -u^-\rangle &=\mm{\na u^-}_p^p+\int_{Q} \lt(\ve{D_s(u)}^{p-2}+b(x,y) \ve{D_s(u)}^{q-2}\ri)D_s(u)D_s(-u^-) \dv\\
			&\quad-\int_{\Omega}f(x,-u^-)(-u^-)\dx.
		\end{aligned}
	\end{equation}
	We prove the four sign properties by contradiction.

	\noindent\textit{Case II.1:} Let $\tau>1$ and $0<\ga \leq \tau$. Assume that $\langle J'(\tau u^+-\ga u^-), \tau u^+\rangle\geq0$. Then
	\begin{align*}
		0
		& \leq \langle J'(\tau u^+-\ga u^-), \tau u^+\rangle\\
		& =\mm{\na (\tau u^+)}_p^p+\int_{Q} \ve{D_s(\tau u^+-\ga u^-)}^{p-2}D_s(\tau u^+-\ga u^-)D_s(\tau u^+) \dv \\
		& +\int_{Q}b(x,y) \ve{D_s(\tau u^+-\ga u^-)}^{q-2}D_s(\tau u^+-\ga u^-)D_s(\tau u^+) \dv
		-\int_{\Omega}f(x,\tau u^+)\tau u^+\dx
	\end{align*}
	and
	\begin{equation}\label{u+2}
		\begin{aligned}
			0
			& \leq \tau^q \lt(\mm{\na u^+}_p^p+\int_{Q} \lt(\ve{D_s(u)}^{p-2}+b(x,y) \ve{D_s(u)}^{q-2}\ri)D_s(u)D_s(u^+) \dv\ri) \\
			&\quad -\int_{\Omega}f(x,\tau u^+)\tau u^+\dx.
		\end{aligned}
	\end{equation}
	Divide \eqref{u+2} by $\tau^q$ and use \eqref{u+}, we deduce
	\begin{align*}
		\int_{\Omega}\lt(\frac{f(x,\tau u^+)}{(\tau u^+)^{q-1}}-\frac{f(x,u^+)}{{(u^+)}^{q-1}}\ri)(u^+)^q \dx \leq 0,
	\end{align*}
	which is a contradiction because of \eqref{H2} \eqref{H2vi}. Thus, \begin{align*}
		\langle J'(\tau u^+-\ga u^-), \tau u^+\rangle<0.
	\end{align*}
	\textit{Case II.2:} For $\tau<1$ and $0<\tau \leq \ga$, assume that $\langle J'(\tau u^+-\ga u^-), \tau u^+\rangle\leq 0$. Repeating a similar argument to that in Case II.1, we again obtain a contradiction. Therefore, $\langle J'(\tau u^+-\ga u^-), \tau u^+\rangle>0$.

	\noindent \textit{Case II.3:} For  $\ga>1$ and $0<\tau \leq \ga$, assume $\langle J'(\tau u^+-\ga u^-), -\ga u^-\rangle\geq 0$.
	Then
	\begin{align*}
		0 & \leq \langle J'(\tau u^+ - \ga u^-),-\ga u^-\rangle =\mm{\na(\ga u^-)}_p^p   \\
		& \quad +\int_{Q} \lt(\ve{D_s(\tau u^+-\ga u^-)}^{p-2}+b(x,y) \ve{D_s(\tau u^+-\ga u^-)}^{q-2}\ri)D_s(\tau u^+-\ga u^-)D_s(-\ga u^-) \dv \\
		& \quad -\int_{\Omega}f(x,- \ga u^-)(- \ga u^-)\dx
	\end{align*}
	and
	\begin{equation}
		\begin{aligned}
			\label{u-2} 0 & \leq \ga^q \lt(\mm{\na u^-}_p^p+\int_{Q} \lt(\ve{D_s(u)}^{p-2}+b(x,y) \ve{D_s(u)}^{q-2}\ri)D_s(u)D_s(-u^-)\ \dv \ri) \\
			& -\int_{\Omega}f(x,- \ga u^-)(- \ga u^-)\dx.
		\end{aligned}
	\end{equation}
	Divide \eqref{u-2} by $\ga^q$ and use \eqref{u-}, we deduce
	\begin{align*}
		\int_{\Omega} \lt(\frac{f(x,-u^-)}{(u^-)^{q-1}}-\frac{f(x,-\ga u^-)}{(\ga u^-)^{q-1}}\ri)(u^-)^{q}\dx \leq 0.
	\end{align*}
	Again this contradicts \eqref{H2} \eqref{H2vi}, since $\gamma>1$. Thus, $\langle J'(\tau u^+-\ga u^-), -\ga u^-\rangle< 0$.

	\noindent\textit{Case II.4:}
	If $\ga<1$ and $0<\gamma \leq \tau$, the same argument as in Case II.3 gives the opposite sign, and we omit the repeated details.

	\noindent\textbf{Step III:} The pair $(\tau_u, \ga_u)$ is unique.

	We prove uniqueness in two cases.

	\noindent\textit{Case III.1:} $u \in \s$

	We claim that $(\tau_u, \ga_u)=(1,1)$ is the unique pair such that $\tau_u u^+-\ga_u u^- \in \s$. Assume that there exists another pair $(\tau,\ga)\neq (1,1)$ such that $\tau u^+-\ga u^- \in \s$. If $0<\tau \leq \ga$, then Case II.2 and Case II.3 from Step II give $1\leq \tau \leq \ga \leq 1$. On the other hand, if $0< \ga \leq \tau$, then Case II.1 and Case II.4 from Step II imply $1 \leq \ga \leq \tau\leq 1$. In both situations we conclude $\tau=\ga=1$. Therefore, if $u\in\s$ the only pair producing an element of $\s$ is $(1,1)$.

	\noindent\textit{Case III.2:} $u\notin \s$

	Assume there exist two pairs $(\tau_1,\ga_1)$ and $(\tau_2, \ga_2)$ such that
	\begin{align*}
		w_1:=\tau_1 u^+-\ga_1 u^- \in \s \quad \text{and} \quad  w_2:=\tau_2 u^+-\ga_2 u^- \in \s.
	\end{align*}
	Then we obtain
	\begin{align}\label{unique}
		w_2=\frac{\tau_2}{\tau_1} \tau_1 u^+ - \frac{\ga_2}{\ga_1} \ga_1 u^-=\frac{\tau_2}{\tau_1} w_1^+- \frac{\ga_2}{\ga_1} w_1^-\in \s.
	\end{align}
	Since $w_1\in \s$, Case III.1 implies that $(1,1)$ is the unique pair such that $1\cdot w_1^+-1\cdot w_1^-\in \s$. Thus, from \eqref{unique}, we get
	\begin{align*}
		\frac{\tau_2}{\tau_1}=\frac{\ga_2}{\ga_1}=1,
	\end{align*}
	which leads to $\tau_1=\tau_2$ and $\ga_1=\ga_2$.

	\noindent\textbf{Step IV:} Unique maximum point

	Define $T_u\colon  [0,\infty)\times [0,\infty) \to \R$ by
	\begin{align*}
		T_u(\tau, \ga)=J(\tau u^+- \ga u^-).
	\end{align*}
	We claim that the unique pair $(\tau_u, \ga_u)$ obtained in Step III is the unique maximum of $T_u$ in $[0,\infty)\times [0,\infty)$. We will prove that $T_u$ has a maximum which cannot be achieved at the boundary of $[0,\infty)\times [0,\infty)$.

	For $\tau \geq \ga \geq 1$, we have
	\begin{align*}
		\frac{T_u(\tau, \ga)}{\tau^{q}}
		& = \frac{J(\tau u^+- \ga u^-)}{\tau^{q}} \\
		& =\frac{1}{p \tau^q}\mm{\na(\tau u^+- \ga u^-)}_p^p+\frac{1}{p\tau^{q}}\int_{Q}\ve{D_s(\tau u^+- \ga u^-)}^p\dv \\
		&\quad +\frac{1}{q\tau^{q}}\int_{Q}b(x,y)\ve{D_s(\tau u^+- \ga u^-)}^q \dv
		-\int_{\Omega}\frac{F(x,\tau u^+- \ga u^-)}{\tau^q}\dx \\
		& \leq \frac{1}{p}\eta(u)-\int_{\Omega}\frac{F(x,\tau u^+)}{(\tau u^+)^q}(u^+)^q\dx-\int_{\Omega}\frac{F(x,- \ga u^-)(u^-)^q}{\ve{-\ga u^-}^q}\frac{\ga^q}{\tau^q}\dx.
	\end{align*}
	By Remark \ref{rem_condition} and the above estimate, it follows that
	\begin{align*}
		\lim_{\ve{(\tau,\ga)}\to \infty}T_u(\tau, \ga)=-\infty,
	\end{align*}
	implying that $T_u$ admits a maximum. Further, assume that $(0,\ga_*)$ is a maximum point of $T_u$ with $\ga_* \geq0$. For $\tau<1$, we compute
	\begin{align*}
		\frac{\pa T_u(\tau, \ga_*)}{\pa \tau} & =\frac{\pa J(\tau u^+-\ga_*u^-)}{\pa \tau} \\
		& =\tau^{p-1}\mm{\na(u^+)}_p^p+\int_{Q}\ve{D_s(\tau u^+- \ga_* u^-)}^{p-2}D_s(\tau u^+- \ga_* u^-)D_s(u^+) \dv \\
		& \quad+\int_{Q}b(x,y)\ve{D_s(\tau u^+- \ga_* u^-)}^{q-2}D_s(\tau u^+- \ga_* u^-)D_s(u^+) \dv \\
		& \quad -\int_{\Omega} f(x,\tau u^+- \ga_* u^-)u^+\dx \\
		& \geq \tau^{p-1}\mm{\na(u^+)}_p^p+\int_{Q}\lt(\tau^{p-1}\ve{D_s(u^+)}^p+\tau^{q-1}\ve{D_s(u^+)}^q\ri)\dv      \\
		& \quad-\int_{\Omega} f(x,\tau u^+- \ga_* u^-)u^+\dx  \\
		& \geq \tau^{p-1}\lt(\mm{\na(u^+)}_p^p-\int_{\Omega}\frac{f(x,\tau u^+)}{(\tau u^+)^{p-1}}(u^+)^{p}\dx\ri).
	\end{align*}
	For $\tau>0$ sufficiently small, the above estimate together with \eqref{H2} \eqref{H2iii} yields $\frac{\pa T_u(\tau, \ga_*)}{\pa \tau}>0$. Thus, $(0,\ga_*)$ cannot be a maximum point since $T_u$ is increasing for $\tau \in (0, \varepsilon)$ with $\varepsilon>0$ small.

	A completely analogous argument shows that a point of the form $(\tau_*, 0)$ cannot be the maximum point of $T_u$. Hence, the global maximum $(\tau, \ga)$ lies in $(0,L) \times (0,L)$ for some $L>0$. From Step I and Step III, we conclude that $(\tau_u,\ga_u)$ is the only point of maximum.
\end{proof}

\begin{proposition} \label{sprop2}
	Let $u \in E$ with $u^\pm \neq 0$ such that $\langle J'(u), u^+\rangle \leq 0$ and $\langle J'(u), -u^-\rangle \leq 0$. Then $0<\tau_u,\ga_u\leq1$, where $\tau_u$ and $\ga_u$ are obtained in Proposition \ref{sprop1}. In addition, if $\langle J'(u), u^+\rangle \geq 0$ and $\langle J'(u), -u^-\rangle \geq 0$, then $\tau_u,\ga_u\geq1$.
\end{proposition}

\begin{proof}
	From Proposition \ref{sprop1}, for each $u\in E$ with $u^\pm \neq 0$, there exists a unique pair $(\tau_u,\ga_u)$ such that $\tau_u u^+-\ga_u u^-\in \s$. Assume $0<\ga_u \leq \tau_u$ with $\tau_u>1$. Then we have
	\begin{align*}
		0= & \,\langle J'(\tau_uu^+-\ga_u u^-), \tau _u u^+\rangle=\mm{\na (\tau_u u^+)}_p^p \\
		& +\int_{Q} \ve{D_s(\tau_uu^+-\ga_u u^-)}^{p-2}D_s(\tau_u u^+-\ga_u u^-)D_s(\tau_u u^+) \dv \\
		& +\int_{Q}b(x,y) \ve{D_s(\tau_u u^+-\ga_u u^-)}^{q-2}D_s(\tau_u u^+-\ga_u u^-)D_s(\tau_u u^+) \dv
		-\int_{\Omega}f(x,\tau_u u^+)\tau_u u^+\dx
	\end{align*}
	and
	\begin{equation}\label{u+3}
		\begin{aligned}
			0& \leq \tau_u^{q}\lt(\mm{\na u^+}_p^p+\int_{Q} \lt(\ve{D_s(u)}^{p-2}+b(x,y) \ve{D_s(u)}^{q-2}\ri)D_s(u)D_s(u^+) \dv \ri) \\
			& -\int_{\Omega}f(x,\tau_u u^+)\tau_u u^+\dx.
		\end{aligned}
	\end{equation}
	From the assumption $\langle J'(u), u^+\rangle \leq 0$, we have
	\begin{align}\label{u+4}
		\mm{\na u^+}_p^p+\int_{Q} \lt(\ve{D_s(u)}^{p-2}+b(x,y) \ve{D_s(u)}^{q-2}\ri)D_s(u)D_s(u^+) \dv -\int_{\Omega}f(x,u^+)u^+\dx \leq 0.
	\end{align}
	Dividing \eqref{u+3} by $\tau_u^q$ and using \eqref{u+4} gives
	\begin{align*}
		\int_{\Omega}  \lt(\frac{f(x,\tau_u u^+)}{(\tau_u u^+)^{q-1}}-\frac{f(x,u^+)}{(u^+)^{q-1}}\ri)(u^+)^{q}\dx \leq 0.
	\end{align*}
	But from \eqref{H2} \eqref{H2vi}, we have that the left-hand side of the above estimate is strictly positive. Hence, we get a contradiction and deduce $\tau_u\leq 1$.

	Now assume $0<\tau_u\leq \ga_u$ with $\ga_u>1$. Then, using $\langle J'(u), -u^-\rangle \leq 0$ and $ 0=\langle J'(\tau_uu^+-\ga_u u^-), \ga _u u^-\rangle$ and following as in \eqref{u+3} and \eqref{u+4}, we can deduce
	\begin{align*}
		\int_{\Omega} \lt(\frac{f(x,-u^-)}{(u^-)^{q-1}}-\frac{f(x,-\ga_u u^-)}{(\ga_u u^-)^{q-1}}\ri)(u^-)^{q-1}\dx \leq 0,
	\end{align*}
	which is again a contradiction because of \eqref{H2} \eqref{H2vi}. Thus, for $0<\tau_u \leq \ga_u$, we get $\ga_u \leq1$. Combining both parts, we conclude $0< \tau_u,\ga_u \leq 1$.
	The second statement, corresponding to the case $\langle J'(u), u^+\rangle \geq 0$ and $\langle J'(u), -u^-\rangle \geq 0$, follows by the same argument with reversed inequalities. This completes the proof.
\end{proof}

\begin{proposition}\label{coercive}
	The quantity $k:=\displaystyle\inf_{u\in \s}J(u)$ is strictly positive. Moreover, $J$ is coercive on $\s$.
\end{proposition}

\begin{proof}
	From Lemma \ref{mpt_geometry_1}, we have
	\begin{align}\label{min1}
		J(u) \geq \ba >0,
	\end{align}
	with $0<\mm{u}_E<\delta<1$, where $\delta$ is defined in Lemma \ref{mpt_geometry_1}. For $u\in \s$, choose $\tau, \ga>0$ such that $\mm{\tau u^+-\ga u^-}_{E}=\delta_1<\delta$. Such a choice is always possible by scaling. Using Proposition \ref{sprop1} and \eqref{min1}, we obtain
	\begin{align*}
		J(u)\geq J(\tau u^+-\ga u^-)\geq\ba >0.
	\end{align*}
	Hence, it follows that $k>0$.

	Let us now show that $J$ is coercive on $\s$. We need to show that for every sequence $\{u_n\}_{n\in\mathbb{N}} \subset \s$, $J(u_n) \to \infty$ whenever $\mm{u_n}_E \to \infty$. Let $\{u_n\}_{n\in\mathbb{N}} \subset \s$ be a sequence with $\lim_{n \to \infty}\mm{u_n}_E=\infty$. Without loss of generality, we may assume that $\mm{u_n}_E>1$. Define $v_n=\frac{u_n}{\mm{u_n}_E}$, then we have $\mm{v_n}_E=1$. Since the space $E$ is reflexive, up to a subsequence, we have
	\begin{equation}\label{coer1}
		\begin{cases}
			v_n  \rightharpoonup v & \text{in }  E,  \\
			v_n  \to v & \text{in } L^s(\Om) \text{ for } s\in[1,p^*), \\
			v_n \to v  & \text{a.e.\,in }  \Om, \\
			v_n^{\pm} \rightharpoonup v^{\pm} & \text{in }  E, \\
			v_n^{\pm} \to v^{\pm}  & \text{in }  L^s(\Om)  \text{ for } s\in[1,p^*), \\
			v_n^{\pm} \to v^{\pm} & \text{a.e.\,in } \Om,
		\end{cases}
	\end{equation}
	for some $v \in E$. We first suppose that $v\neq 0$. From Lemma \ref{re2}, we have
	\begin{equation}\label{coer2}
		\begin{aligned}
			J(u_n)
			& =\frac{1}{p}\mm{\na u_n}_p^p+\int_{Q}\lt(\frac{1}{p}\ve{D_s(u_n)}^p+\frac{b(x,y)}{q}\ve{D_s(u_n)}^q \ri)\dv \\
			& \quad -\int_{\Omega}F(x,u_n)\dx \\
			& \leq \frac{1}{p}\mm{u_n}_E^q-\int_{\Omega}F(x,u_n)\dx.
		\end{aligned}
	\end{equation}
	Dividing \eqref{coer2} by $\mm{u_n}_E^q$, passing to the limit as $n \to \infty$, and using Remark \ref{rem_condition}, we deduce
	\begin{align*}
		\lim_{n \to \infty}\frac{J(u_n)}{\mm{u_n}^q_E}=-\infty,
	\end{align*}
	which is a contradiction as we have $J(u_n)\geq k>0$. Thus, $v=0$ implying $v^-=v^+=0$. As $\{u_n\}_{n\in\mathbb{N}}\subset \s$, by applying Proposition \ref{sprop1}, Lemma \ref{re2} and \eqref{coer1}, for every pair $(\tau_1,\tau_2)\in (0,\infty) \times(0,\infty)$ with $0<\tau_1 \leq \tau_2$, we have
	\begin{align*}
		J(u_n)
		& \geq J(\tau_1 v_n^+-\tau_2 v_n^-)  \\
		& =\frac{1}{p }\mm{\na(\tau_1 v_n^+- \tau_2 v_n^-)}_p^p+\int_{Q}\lt(\frac{1}{p}\ve{D_s(\tau_1 v_n^+- \tau_2 v_n^-)}^p+\frac{b(x,y)}{q}\ve{D_s(\tau_1 v_n^+- \tau_2 v_n^-)}^q\ri) \dv \\
		& -\int_{\Omega}F(x,\tau_1 v_n^+- \tau_2 v_n^-)\dx \\
		& \geq \frac{1}{q} \min(\tau_1^p,\tau_1^q) \eta(v_n)-\int_{\Omega}F(x,\tau_1v_n^+)\dx-\int_{\Omega}F(x,-\tau_2v_n^-)\dx   \\
		& \geq \frac{1}{q} \min(\tau_1^p,\tau_1^q)-\int_{\Omega}F(x,\tau_1v_n^+)\dx-\int_{\Omega}F(x,-\tau_2v_n^-)\dx\\
		& \to \frac{1}{q} \min(\tau_1^p,\tau_1^q).
	\end{align*}
	Hence, for any given $M>0$, we can choose $\tau_1$ large enough such that for all $n \geq m=m(\tau_1)$, we have $J(u_n)>M$. Thus, $J$ is coercive on $\s$.
\end{proof}

\begin{proposition} \label{bound}
	For every $u\in \s$, there exists a constant $L>0$, independent of $u$, such that $\mm{u^{\pm}}_E \geq L$.
\end{proposition}

\begin{proof}
	Let $u\in \s$ with $\mm{u^{\pm}}_E<1$. By definition of $\s$, we have $\langle J'(u), u^+\rangle=0$, that is,
	\begin{equation}\label{u+5}
		\begin{aligned}
			0
			& = \mm{\na u^+}_p^p+\int_{Q} \lt(\ve{D_s(u)}^{p-2}+b(x,y) \ve{D_s(u)}^{q-2}\ri)D_s(u)D_s(u^+) \dv \\
			& -\int_{\Omega}f(x,u)u^+\dx                                                                                     \\
			  & \geq \mm{\na u^+}_p^p+\int_{Q}\lt(\ve{D_s(u^+)}^{p}+b(x,y)\ve{D_s(u^+)}^q\ri)\dv-\int_{\Omega}f(x,u^+)u^+ \dx.
		\end{aligned}
	\end{equation}
	From \eqref{H2} \eqref{H2ii} and \eqref{H2iii}, for a given $\delta>0$, there exists a positive constant $C_{\delta}>0$ such that
	\begin{align}\label{fgrowth}
		\ve{f(x,t)}\leq \delta \ve{t}^{p-1}+C_{\delta}\ve{t}^{r-1}, \quad \text{for a.a.\,}x\in \Omega\text{ and for all }t\in\R.
	\end{align}
	Combining \eqref{u+5} and \eqref{fgrowth} and using the embeddings $W_0^{1,p}(\Om) \hookrightarrow L^p(\Om)$ and $E \hookrightarrow L^r(\Om)$, we obtain
	\begin{align*}
		\mm{\na u^+}_p^p+\int_{Q}\lt(\ve{D_s(u^+)}^{p}+b(x,y)\ve{D_s(u^+)}^q\ri)\dv\leq \delta S^p_p\mm{\na u^+}_p^p+C_{\delta}\mathbb{S}^r_{r}\mm{u^+}^r_{E}.
	\end{align*}
	For $0<\delta<\frac{1}{S_p^p}$, by using Lemma \ref{re2}, we have
	\begin{align*}
		C_{\delta}\mathbb{S}^r_{r}\mm{u^+}^r_E\geq (1-\delta S_p^p)\lt(\mm{\na u^+}_p^p+\int_{Q}\lt(\ve{D_s(u^+)}^{p}+b(x,y)\ve{D_s(u^+)}^q\ri)\dv\ri)
	\end{align*}
	and
	\begin{align*}
		\mm{u^+}^r_E\geq c \eta(u^+)\geq c\mm{u^+}^q_E.
	\end{align*}
	Since $q<r$, there exists $L>0$ such that $\mm{u^+}_E\geq  L$. An analogous argument applies to the case $u^-$.
\end{proof}

\begin{proposition}\label{infm}
	There exists $v_0 \in \s$ such that $J(v_0)=k$.
\end{proposition}

\begin{proof}
	Let $\{v_n\}_{n\in\mathbb{N}}\subset \s$ be a minimizing sequence, that is, $J(v_n) \to k$. By Proposition \ref{coercive}, the sequence $\{v_n\}_{n\in\mathbb{N}}$ is bounded in $E$. Consequently, $\{v_n^+\}_{n\in\mathbb{N}}$ and $\{v_n^-\}_{n\in\mathbb{N}}$ are bounded in $E$ as well. Thus, up to a subsequence,
	\begin{align}\label{min2}
		\begin{cases}
			v_n^{\pm}  \rightharpoonup v_0^{\pm} & \text{in } E, \ v_0^{\pm} \geq 0,\\
			v_n^{\pm}  \to v_0^{\pm} &\text{in } L^s(\Om) \text{ for } s\in[1,p^*) \text { and a.e.\,in }\Omega
		\end{cases}
	\end{align}
	From \eqref{min2}, \eqref{H2} \eqref{H2ii} and the dominated convergence theorem, it follows that
	\begin{equation} \label{dct}
		\begin{aligned}
			\lim_{n \to \infty}\int_{\Omega}f(x,\tau v_n^+)\tau v_n^+\dx& =\int_{\Omega}f(x,\tau v_0^+)\tau v_0^+ \dx, \\
			\lim_{n \to \infty}\int_{\Omega}f(x,-\ga v_n^-)(-\ga v_n^-)\dx & = \int_{\Omega}f(x,-\ga v_0^-)(-\ga v_0^-) \dx, \\
			\lim_{n \to \infty}\int_{\Omega}F(x,\tau v_n^+)\dx & = \int_{\Omega}F(x,\tau v_0^+) \dx,  \\
			\lim_{n \to \infty} \int_{\Omega}F(x,-\ga v_n^-)\dx& =\int_{\Omega}F(x,-\ga v_0^-)\dx
		\end{aligned}
	\end{equation}
	for every $\tau,\ga>0$.

	We claim  that $v_0^+\neq 0$ and  $v_0^- \neq0$. Assume that $v^+_0=0$. Since  $v_n\in \s$, we have
	\begin{align*}
		0
		& =\langle J'(v_n), v_n^+ \rangle \\
		& =\mm{\na v_n^+}_p^p+\int_{Q} \lt(\ve{D_s(v_n)}^{p-2}+b(x,y) \ve{D_s(v_n)}^{q-2}\ri)D_s(v_n)D_s(v_n^+) \dv \\
		& \quad -\int_{\Omega}f(x,v_n^+)v_n^+\dx  \\
		& \geq  \eta(v_n^+)-\int_{\Omega} f(x,v_n^+)v_n^+\dx.
	\end{align*}
	Passing the limit as $n \to \infty$ in the above estimate and using \eqref{dct}, we get $\lim_{n\to \infty}\eta(v_n^+)=0$ which is, by Lemma \ref{re2} (v), equivalent to $\mm{v_n^+}_E \to 0$. But this is a contradiction as, by Proposition \ref{bound}, we have  $\mm{v_n^+}_E\geq L>0$. Hence, $v_0^+ \neq 0$. By employing a similar argument, one can show that $v_0^-\neq 0$.

	Since $v_0^\pm \neq0$, by Proposition \ref{sprop1}, there exists a unique pair $(\tau_{v_0},\ga_{v_0})$ satisfying $\tau_{v_0}v_0^+-\ga_{v_0}v_0^-\in \s$. By weak lower semicontinuity of norms and by \eqref{dct}, we obtain
	\begin{align*}
		\langle J'(v_0), \pm v_0^{\pm}\rangle & =\mm{\na v_0^{\pm}}_p^p+\int_{Q} \lt(\ve{D_s(v_0)}^{p-2}+b(x,y) \ve{D_s(v_0)}^{q-2}\ri)D_s(v_0)D_s(v_0^{\pm}) \dv\\
		& \quad-\int_{\Omega}f(x,v_0)v_0^{\pm}\dx\\
		& \leq \liminf_{n \to \infty}{\lt(\mm{\na v_n^\pm }_p^p+\int_{Q} \lt(\ve{D_s(v_n)}^{p-2}+b(x,y) \ve{D_s(v_n)}^{q-2}\ri)D_s(v_n)D_s(v_n^\pm) \dv\ri)} \\
		& \quad -\lim_{n \to \infty}\int_{\Omega}f(x,v_n)v_n^{\pm}\dx\\
		& = \liminf_{n \to \infty}\langle J'(v_n), \pm v_n^{\pm}\rangle =0,
	\end{align*}
	as $v_n\in\s$. Taking Proposition \ref{sprop2} into account yields
	$\tau_{v_0},\ga_{v_0}\in(0,1]$. From \eqref{H2} \eqref{H2v}, for a.a.\,$x\in \Om$, it follows that
	\begin{equation}\label{f}
		\begin{aligned}
			\frac{1}{q}f(x,\tau_{v_0} v_0^+)
			\tau_{v_0} v_0^+
			-F(x,\tau_{v_0} v_0^+)
			 & \leq \frac{1}{q} f(x, v_0^+) v_0^+ -F(x, v_0^+),     \\
			\frac{1}{q}f(x,-\ga_{v_0} v_0^-)
			(-\ga_{v_0} v_0^-)
			-F(x,-\ga_{v_0} v_0^-)
			 & \leq \frac{1}{q} f(x, -v_0^-) (-v_0^-)-F(x, -v_0^-).
		\end{aligned}
	\end{equation}
	Then, from $\tau_{v_0}v_0^+-\ga_{v_0}v_0^-\in \s$, $0<\tau_{v_0},\ga_{v_0}\leq1$, \eqref{dct}, \eqref{f} and $v_n \in \s$, we have
	\begin{align*}
		& k  \leq J(\tau_{v_0}v_0^+ -\ga_{v_0}v_0^-)-\frac{1}{q}\langle J'(\tau_{v_0}v_0^+-\ga_{v_0}v_0^-),\tau_{v_0}v_0^+-\ga_{v_0} v_0^-\rangle\\
		& =\lt(\frac{1}{p}-\frac{1}{q}\ri)\mm{\na (\tau_{v_0}v_0^+-\ga_{v_0} v_0^-)}_p^p\\
		&\quad -\int_{Q}\lt(\frac{1}{p}\ve{D_s (\tau_{v_0}v_0^+-\ga_{v_0}v_0^-)}^p+\frac{b(x,y)}{q}\ve{D_s(\tau_{v_0}v_0^+-\ga_{v_0}v_0^-)}^q\ri)\dv \\
		&\quad -\int_{\Omega}F(x,\tau_{v_0}v_0^+) \dx-\int_{\Omega} F(x,-\ga_{v_0}v_0^-)\dx \\
		&\quad -\frac{1}{q}\int_{Q}\lt(\ve{D_s (\tau_{v_0}v_0^+-\ga_{v_0}v_0^-)}^p+{b(x,y)}\ve{D_s(\tau_{v_0}v_0^+-\ga_{v_0}v_0^-)}^q\ri)\dv\\
		& \quad+\frac{1}{q} \int_{\Omega}f(x,\tau_{v_0}v_0^+)\tau_{v_0}v_0^+\dx+\frac{1}{q} \int_{\Omega}f(x,-\ga_{v_0}v_0^-)(-\ga_{v_0}v_0^-)\dx \\
		& \leq \lt(\frac{1}{p}-\frac{1}{q}\ri)\lt(\mm{\na v_0^+-\na v_0^-}_p^p+\int_{Q} \ve{D_s(v_0^+-v_0^-)}^p\dv\ri)\\
		&\quad +\int_{\Omega}\frac{1}{q} f(x, v_0^+) v_0^+ \dx-\int_{\Omega}F(x, v_0^+)\dx \\
		& \quad+\int_{\Omega}\frac{1}{q} f(x, -v_0^-) (-v_0^-) \dx-\int_{\Omega} F(x, -v_0^-)\dx \\
		& \leq \liminf_{n\to \infty} \lt(J(v_n^+-v_n^-)-\frac{1}{q}\langle J'(v_n^+-v_n^-),v_n^+-v_n^-\rangle\ri)=k.
	\end{align*}
	Thus, we conclude that $\tau_{v_0}=\ga_{v_0}=1$ and $J(v_0)=k$.
\end{proof}

\begin{proof}[Proof of Theorem \ref{t5}]
	We prove that the function $v_0$, obtained in Proposition \ref{infm}, is a least energy sign-changing solution of problem \eqref{frac_dbl}.

	Suppose by contradiction that $J'(v_0) \neq 0$. Then there exists $\beta>0$ and $\delta_1>0$ such that
	\begin{align*}
		\mm{J'(v)}_{*}\geq \ba \quad \text{for all } v\in E \text{ with } \mm{v-v_0}_{E}<3 \delta_1.
	\end{align*}
	Recall that $v_0^\pm \neq 0$. Then, for any $v \in E$, using \eqref{embed}, we obtain
	\begin{align*}
		\mm{v_0-v}_{E} \geq \mathbb{S}^{-1}_p \mm{v_0-v}_p \geq
		\begin{cases}
			\mathbb{S}^{-1}_p \mm{v_0^-}_p, &\text{if}\ v^-=0, \\
			\mathbb{S}^{-1}_p \mm{v_0^+}_p,&\text{if}\ v^+=0.
		\end{cases}
	\end{align*}
	Choose $\delta_2>0$ such that $\delta_2<\min\{\mathbb{S}^{-1}_p \mm{v_0^-}_p,\mathbb{S}^{-1}_p \mm{v_0^+}_p\}$. Then, every $v\in E$ satisfying $\mm{v_0-v}_E<\delta_2$ must fulfill $v^{\pm}\neq0$.

	Let $\delta_0=\min\{\delta_1,\delta_2/2\}$ and note that the map $(\tau, \ga) \to \tau v_0^+-\ga v_0^-$ from $[0,\infty)^2$ into $E$ is continuous. Thus, we find $\ell\in(0,1)$ such that for all $\tau, \ga \geq 0$ with $\ell>\max\{|\tau-1|,|\ga-1|\}$ it holds $\mm{\tau v_0^+-\ga v_0^--v_0}_E<\delta_0$.

	Let $L=(1-\ell,1+\ell) \times (1-\ell,1+\ell)$. By Proposition \ref{sprop1}, for every $\tau, \ga \geq 0$ with $(\tau,\ga)\neq (1,1)$, we obtain
	\begin{align}\label{L1}
		J(\tau v_0^+-\ga v_0^-)<J(v_0^+-v_0^-)=k.
	\end{align}
	Thus, we have
	\begin{align*}
		m=\max_{(\tau,\ga)\in \pa L}J(\tau v_0^+-\ga v_0^-)<J(v_0^+-v_0^-)=k.
	\end{align*}
	Now we are able to apply the quantitative deformation lemma, stated in Lemma \ref{defm}, with
	\begin{align*}
		M=B(v_0,\delta_0), \quad c=k, \quad \varepsilon=\min\left\{\frac{k-m}{4},\frac{\ba \delta_0}{8}\right\},
	\end{align*}
	where $\delta_0$ is defined as above. Note that $M_{2\delta_0}= B(v_0,3\delta_0)$. From the construction of $\varepsilon$, it follows that for every pair $(\tau,\ga)\in \pa L$ we have
	\begin{align}\label{L2}
		J(\tau v_0^+-\ga v_0^-)< m= m+k-k\leq k-\frac{(k-m)}{2}\leq k-2\varepsilon.
	\end{align}
	Further, let $\Theta\colon  [0,\infty) \times[0,\infty)\to E$ and $\Psi\colon [0,\infty) \times[0,\infty)\to \R^2$ be defined by
	\begin{align*}
		\Theta(\tau,\ga)&=\Pi(1,\tau v_0^+-\ga v_0^-),\\ \Psi(\tau,\ga)&=(\Psi_1,\Psi_2)=\left(\langle J'(\Theta(\tau,\ga)),\Theta^+(\tau,\ga)\rangle,\langle J'(\Theta(\tau,\ga)),-\Theta^-(\tau,\ga)\rangle\right),
	\end{align*}
	where $\Pi$ is the map as  in Lemma \ref{defm}.
	The continuity of the maps $\Theta$ and $\Psi$ follows from the continuity of $\Pi$ and the differentiability of $J$, respectively.
	Then \eqref{L2} together with Lemma \ref{defm} (i) gives $\Theta(\tau,\ga)=\tau v_0^+-\ga v_0^-$ for all $(\tau,\ga)\in \pa L$.
	From Step II of Proposition \ref{sprop1}, for $\zeta\in[1-\ell,1+\ell]$ we deduce that
	\begin{align*}
		\Psi_1(1-\ell, \zeta)&>0>\Psi_1(1+\ell, \zeta),\\
		\Psi_2(\zeta,1-\ell)&>0>\Psi_2(\zeta,1+\ell).
	\end{align*}
	From the Poincar\'{e}–Miranda existence theorem stated in Theorem \ref{IVP}, it follows that there exists a pair $(\tau_0,\ga_0)\in L$ with $\Psi(\tau_0,\ga_0)=0$, which means
	\begin{align}\label{theta1}
		\langle J'(\Theta^+(\tau_0,\ga_0)),\Theta^+(\tau_0,\ga_0)\rangle=0=\langle J'(-\Theta^-(\tau_0,\ga_0)),-\Theta^-(\tau_0,\ga_0)\rangle.
	\end{align}
	In view of our choice of $\ell$, Lemma \ref{defm} (iv) yields
	\begin{align*}
		\mm{\Theta(\tau_0,\ga_0)-v_0}_E < 2\delta_0 \leq \delta_2.
	\end{align*}
	Using the above inequality and the definition of $\delta_2$, we get
	$\Theta^\pm(\tau_0,\ga_0)\neq 0$ which together with \eqref{theta1} gives $\Theta(\tau_0,\ga_0)\in \s$.
	By our choice of $\ell$, \eqref{L1}, and  Lemma \ref{defm} (ii), we get  $J(\Theta(\tau_0,\ga_0))\leq k-\varepsilon$. This is a contradiction since $k$ is the infimum of $J$ over $\s$. Hence, $v_0$ is a least energy sign-changing solution of problem \eqref{frac_dbl}.
\end{proof}


\end{document}